\documentclass[11pt, reqno]{amsart}
\usepackage{amsmath,amsthm,amssymb,mathrsfs}
\usepackage[abbrev,non-sorted-cites]{amsrefs}
\usepackage{color}
\usepackage{verbatim}
\usepackage{hyperref}

\theoremstyle{plain}
  \newtheorem{thm}{Theorem}[section]
  \newtheorem{lem}[thm]{Lemma}
  \newtheorem{prop}[thm]{Proposition}
          \newtheorem*{prop*a}{Proposition A}
   
	\newtheorem*{thm*a}{Theorem A}
	\newtheorem*{thm*b}{Theorem B}
\theoremstyle{definition}
  \newtheorem{defn}[thm]{Definition}
	
  \newtheorem{rmk}[thm]{Remark}
  \newtheorem*{ack*}{Acknowledgement}
  \newtheorem*{ques*}{Question}
\theoremstyle{plain}

\numberwithin{equation}{section}

\newcommand\ip[2]{\langle{#1},{#2}\rangle}

\newcommand\RRM[4]{R({e}_{#1},{e}_{#2},{e}_{#3},{e}_{#4})}
\newcommand\RM[4]{{R}_{#1 #2 #3 #4}}

\newcommand\RTD[4]{{R}_{\tilde{#1} \tilde{#2} \tilde{#3} \tilde{#4}}}

\newcommand\hm[3]{{h}_{#1 #2 #3}}
\newcommand\hmd[4]{{h}_{#1 #2 #3 ; #4}}
\newcommand\htd[3]{\tilde{h}_{#1 #2 #3}}
\newcommand\OMTD[1]{\tilde{\Omega}_{#1}}
\newcommand\px[1]{\frac{\partial~}{\partial x^{#1}}}
\newcommand\py[1]{\frac{\partial~}{\partial y^{#1}}}

\newcommand\ltn[2]{|\!|#1|\!|_{H^{#2}}}

\newcommand\bn{\overline{\nabla}}
\newcommand\nt{\nabla^{\Gamma_t}}
\newcommand\nres{\tilde{\nabla}}

\newcommand\pt[1]{\partial_{\,\tilde{#1}}}
\newcommand\htdf[4]{\tilde{h}_{#1 #2 #3, #4}}

\newcommand\pl{\partial}

\newcommand\oh{\frac{1}{2}}
\newcommand\dd{{\mathrm d}}

\newcommand\w{\wedge}
\newcommand\sm{\sigma}
\newcommand\dt{\delta}
\newcommand\ep{\epsilon}
\newcommand\vep{\varepsilon}
\newcommand\vph{\varphi}
\newcommand\om{\omega}
\newcommand\ta{\theta}
\newcommand\gm{\gamma}
\newcommand\kp{\kappa}
\newcommand\af{\alpha}
\newcommand\bt{\beta}

\newcommand\Om{\Omega}
\newcommand\Sm{\Sigma}
\newcommand\Gm{\Gamma}
\newcommand\Ld{\Lambda}

\newcommand\CC{\mathcal{C}}
\newcommand\CO{\mathcal{O}}
\newcommand\CH{\mathcal{H}}
\newcommand\CN{\mathcal{N}}
\newcommand\CV{\mathcal{V}}
\newcommand\CA{\mathcal{A}}
\newcommand\CR{\mathcal{R}}
\newcommand\CS{\mathcal{S}}
\newcommand\CP{\mathcal{P}}
\newcommand\CT{\mathcal{T}}
\newcommand\CE{\mathcal{E}}

\newcommand\CF{\mathcal{F}}

\newcommand\RO{\mathrm{O}}
\newcommand\RSO{\mathrm{SO}}

\newcommand\BS{\mathbb{S}}

\newcommand\BCP{\mathbb{CP}}
\newcommand\BR{\mathbb{R}}
\newcommand\BN{\mathbb{N}}

\newcommand\fs{\mathfrak{s}}

\newcommand\fR{\mathfrak{R}}

\newcommand\ul{\underline}
\newcommand\td{\tilde}

\newcommand\ot{\otimes}

\newcommand\me{e}
\newcommand\te{\tilde{e}}

\newcommand\tom{\tilde{\om}}

\newcommand\namb{\nabla}
\newcommand\nsub{\nabla^\Sm}
\newcommand\nnor{\nabla^\perp}
\newcommand\Dt{\Delta^{\Gm_t}}

\newcommand\N{N}

\newcommand\bx{\mathbf{x}}
\newcommand\by{\mathbf{y}}
\newcommand\bp{\mathbf{p}}
\newcommand\bq{\mathbf{q}}
\newcommand\bs{\mathbf{s}}
\newcommand\bu{\mathbf{u}}
\newcommand\bv{\mathbf{v}}
\newcommand\bo{\mathbf{0}}

\DeclareMathOperator{\tr}{tr}

\DeclareMathOperator{\Hess}{Hess}

\DeclareMathOperator{\vol}{vol}
\DeclareMathOperator{\spn}{span}

\newcommand{\II}{{\mathrm I}\!{\mathrm I}}
\newcommand\IIt{\II^{t}}

\newcommand\Ric{\mathrm{Ric}}
\newcommand\dv{\mathrm{dvol}}

\begin{document}

\title{A strong stability condition on minimal submanifolds and its implications}

\author{Chung-Jun Tsai}
\address{Department of Mathematics\\
National Taiwan University\\ Taipei 10617\\ Taiwan}
\email{cjtsai@ntu.edu.tw}

\author{Mu-Tao Wang}
\address{Department of Mathematics\\
Columbia University\\ New York\\ NY 10027\\ USA}
\email{mtwang@math.columbia.edu}


\thanks{Supported in part by Taiwan MOST grants 105-2115-M-002-012, 106-2115-M-002-005-MY2 and NCTS Young Theoretical Scientist Award (C.-J.\ Tsai). This material is based upon work supported by the National Science Foundation under Grants No. DMS-1405152 and No. DMS-1810856 (Mu-Tao Wang).  Part of this work was carried out when Mu-Tao Wang was visiting the National Center of Theoretical Sciences at National Taiwan University in Taipei, Taiwan. }

\begin{abstract} We identify a strong stability condition on minimal submanifolds that implies uniqueness and dynamical stability properties. In particular, we prove a uniqueness theorem and a $\CC^1$ dynamical stability theorem of the mean curvature flow for minimal submanifolds that satisfy this condition. The latter theorem states that the mean curvature flow of any other submanifold in a $\CC^1$ neighborhood of such a minimal submanifold exists for all time, and converges exponentially to the minimal one. This extends our previous uniqueness and stability theorem \cite{ref_TsaiW} which applies only to calibrated submanifolds of special holonomy ambient manifolds. 

\end{abstract}

\maketitle


\section{Introduction}

In our previous work \cite{ref_TsaiW}, we study the uniqueness and $\CC^1$ dynamical stability of calibrated submanifolds in manifolds of special holonomy with explicitly constructed Riemannian metrics.  The result is extended to minimal submanifolds of general Riemannian manifolds in this paper.  The assumption for the uniqueness and dynamical stability theorem is identified as a strongly stable condition which implies the stability of the minimal submanifold in the usual sense of the second variation of the volume functional.  Recall that the mean curvature flow is the negative gradient flow of the volume functional.  It is thus natural to ask whether a local minimizer (a stable minimal submanifold) of the volume functional is stable under the mean curvature flow.  Such a question of great generality has been addressed in the celebrated work of L.\ Simon \cite{ref_Simon}: when is a local minimizer dynamically stable under the gradient flow, i.e.\ does the gradient flow of a small perturbation of a local minimizer still converge back to the local minimizer?  The question in the context of \cite{ref_Simon} concerns a nonlinear parabolic system defined on a compact manifold, and it was proved that the analyticity of the functional and the smallness in $\CC^2$ norm are sufficient for the validity of the dynamical stability.
The question we addressed here corresponds to the specialization to the volume functional of compact submanifolds.  A natural measurement of the distance between two submanifolds is the $\CC^1$ (or Lipschitz) norm, {in terms of which the ``closeness" condition of our current result is formulated.} \footnote{It was suggested by a reviewer that, since the volume functional is well-defined for varifolds, it is possible that some measure
theoretical ``closeness" condition for varifolds works for such a dynamical stability theorem. Indeed, a recent preprint by J. D. Lotay and F. Schulze ``Consequences of strong stability of minimal submanifolds" (arXiv: 1802.03941) generalized our result to the setting of integral currents under the enhanced Brakke flow.}

As derived in \cite[\S3]{ref_Simons}, the Jacobi operator of the second variation of the volume functional is $(\nnor)^*\nnor + \CR - \CA$, where $(\nnor)^*\nnor$ is the Bochner Laplacian of the normal bundle, $\CR$ is an operator constructed from the restriction of the ambient Riemann curvature, and $\CA$ is constructed from the second fundamental form.  The precise definition can be found in \S\ref{sec_stable}.  A minimal submanifold is said to be strongly stable if $\CR-\CA$ is a positive operator, see \eqref{sstable}.  Since $(\nnor)^*\nnor$ is a non-negative operator, strong stability implies stability in the sense of the second variation of the volume functional. In particular, the strong stability condition is satisfied by all the calibrated submanifolds considered in \cite{ref_TsaiW} which include ($M$ denotes the ambient Riemannian manifold and $\Sigma$ denotes the minimal submanifold):

\begin{enumerate}
\item $M$ is the total space of the cotangent bundle of a sphere, $T^*S^n$ (for $n>1$), with the Stenzel metric \cite{ref_St}, and $\Sigma$ is the zero section;

\item $M$ is the total space of the cotangent bundle of a complex projective space, $T^*\BCP^n$, with the Calabi metric \cite{ref_Calabi} and $\Sigma$ is the zero section;

\item $M$ is the total space of  one of the vector bundles $\BS(S^3)$, $\Lambda^2_-(S^4)$, $\Lambda^2_-(\BCP^2)$, and $\BS_-(S^4)$ with the Ricci flat metric constructed by Bryant--Salamon \cite{ref_BS}, where $\BS$ is the spinor bundle and $\BS_-$ is the spinor bundle of negative chirality, and $\Sigma$ is the zero section of the respective vector bundle.

\end{enumerate}
These are all metrics of special holonomy that are known to be written in a closed form, {and to have simplest non-trivial topology}. 
Note that in all these examples, the metrics of the total space are Ricci flat, and the zero sections are totally geodesic.  Hence, the strong stability in these examples is equivalent to the positivity of the operator $\CR$. In \cite{ref_TsaiW}, we proved uniqueness and dynamical stability theorems for the corresponding calibrated submanifolds and the proofs rely on the explicit knowledge of the ambient metric, whose coefficients are governed by solutions of ODE systems. A natural question was how general such rigidity phenomenon is. In this article, we discover that the strong stability condition is precisely the condition that makes everything work. 
Moreover, we identify more examples that satisfy the strong stability condition:

\begin{prop*a}

Each of the following pairs $(\Sigma, M)$ of minimal submanifolds $\Sigma$ and their ambient Riemannian manifolds $M$ satisfy the strong stability condition \eqref{sstable} :

\begin{enumerate}
\item $M$ is any Riemannian manifold of negative sectional curvature and $\Sigma$ a totally geodesic submanifold; {in particular, geodesics in hyperbolic surfaces or $3$-manifolds are stongly stable;}
\item $M$ is any K\"ahler manifold and $\Sigma$ is a complex submanifold whose normal bundle has positive holomorphic curvature.  
\item $M$ is any Calabi--Yau manifold and $\Sigma$ is a special Lagrangian with positive Ricci curvature;
\item $M$ is any $G_2$ manifold and $\Sigma$ is a coassociative submanifold with positive definite $-2W_- + \frac{s}{3}$ on $\Ld_-^2$; {this is a curvature condition on $g|_\Sm$, see \eqref{coa_stable}.}
\end{enumerate}

\end{prop*a}

For example (i), the strong stability can be checked directly. The examples (ii), (iii), and (iv) will be explained in \S\ref{sec_ss} and Appendix \ref{apx_stable}.

We now state the main results of this paper.  The first one says that a strongly stable minimal submanifold is rather unique.

\begin{thm*a}
Let $\Sm^n\subset (M,g)$ be a compact, minimal submanifold which is strongly stable in the sense of \eqref{sstable}.  Then there exists a tubular neighborhood $U$ of $\Sm$ such that $\Sm$ is the only compact minimal submanifold in $U$ with dimension no less than $n$.
\end{thm*a}

The second one is on the dynamical stability of a strongly stable minimal submanifold.

\begin{thm*b}
Let $\Sm\subset (M,g)$ be a compact, oriented minimal submanifold which is strongly stable in the sense of \eqref{sstable}.  If $\Gm$ is a submanifold that  is close to $\Sm$ in $\CC^1$, the mean curvature flow $\Gm_t$ with $\Gm_0 = \Gm$ exists for all time, and $\Gm_t$ converges to $\Sm$ smoothly as $t\to\infty$.
\end{thm*b}

The precise statements can be found in Theorem \ref{uniqueness} (Theorem A) and Theorem \ref{longtime} (Theorem B), respectively.
{For defining a measurement for the slope, the minimal submanifold $\Sm$ in Theorem B is required to be oriented.  The slope measurement is based on certain extension of the volume form of $\Sm$ (see \S\ref{sec_linear}).}

The $\CC^1$ dynamical stability of the mean curvature flow for those calibrated submanifolds considered in \cite{ref_TsaiW} was proved in the same paper.  In this regard, this theorem is a generalization of our previous result.

Here are some remarks on the strong stability condition.  In the viewpoint of the second variational formula, the condition is natural, and is stronger than the positivity of the Jacobi operator.  The main results of this paper are basically saying that the strong stability has nice geometric consequences.
In particular, the minimal submanifold $\Sm$ needs not be totally geodesic, while most known results about the convergence of higher codimensional mean curvature flow are under the totally geodesic assumption, e.g. \cite{ref_W3}.

\smallskip

{{\bf Note addded.}
One may wonder whether the stability condition already implies the dynamical stability.  More precisely, if the Jacobi operator has only positive spectrum, is the minimal submanifold $\Sm^n$ stable under the mean curvature flow?  The answer is yes, provided one requires more on the initial condition.  This was studied by Naito \cite{ref_Naito} for general negative gradient flows, and by Deckelnick \cite{ref_Dkk} for surface mean curvature flows in $\BR^3$ with Dirichlet boundary condition.  The result of Naito says that if $\Gm$ is close to $\Sm$ in $H^r (= L^2_r)$ for $r>\frac{n}{2}+2$, then the mean curvature flow $\Gm_t$ exists for all time, and converges to $\Sm$ in $H^r$ as $t\to\infty$.  Since $r>\frac{n}{2}+2$, $H^r\hookrightarrow \CC^{2,\af}$ for some $\af\in(0,1]$, and $\Gm$ is close to $\Sm$ in $\CC^{2,\af}$.
Section \ref{only_stable} is added to explain more on the results of Naito.

}

\begin{ack*}
The authors are grateful to Prof.\ Gerhard Huisken for his comments on the stability of the mean curvature flow and for pointing out the reference \cite{ref_Dkk}.  The authors would like to thank Yohsuke Imagi for helpful discussions, and to thank the anonymous referee for helpful comments on the earlier version of this paper.
\end{ack*}

\section{Local geometry near a submanifold}\label{sec_local}

\subsection{Notations and basic properties}\label{sec_notation}

Let $(M, g)$ be a Riemannian manifold of dimension $n+m$, and $\Sm\subset M$ be a compact (embedded) submanifold of dimension $n$. We use $\langle \cdot, \cdot  \rangle$ to denote the evaluation of two tangent vectors by the metric tensor $g$. The notation $\langle \cdot, \cdot  \rangle$ is also abused to denote the evaluation with respect to the induced metric on $\Sigma$. Denote by $\namb$ the Levi-Civita connection of $(M,g)$, and by $\nsub$ the Levi-Civita connection of the induced metric on $\Sigma$.

Denote by $\N\Sm$ the normal bundle of $\Sm$ in $M$.  The metric $g$ and its Levi-Civita connection induce a bundle metric (also denoted by $\langle \cdot, \cdot  \rangle$ ) and a metric connection for $\N\Sm$.  The bundle connection on $N\Sigma$ will be denoted by $\nnor$.

In the following discussion, we are going to choose a local orthonormal frame $\{\me_1 ,\cdots,\me_n,$ $\me_{n+1},\cdots,\me_{n+m}\}$ for $TM$ near a point $p\in\Sm$ such that the restriction of $\{\me_1,\cdots,\me_n\}$ on $\Sm$ is a frame for $T\Sm$ and the restrictions of $\{\me_{n+1},\cdots,\me_{n+m}\}$ is a frame for $\N\Sm$.  The indexes $i, j, k$ range from $1$ to $n$, the indexes $\af, \bt, \gm$ range from $n+1$ to $n+m$, the indexes $A, B, C$ range from $1$ to $n+m$, and repeated indexes are summed.

The convention of the Riemann curvature tensor is
\begin{align*}
R(e_C,e_D)e_B &= \namb_{\me_C}\namb_{\me_D}\me_B - \namb_{\me_D}\namb_{\me_C}\me_B - \namb_{[\me_C,\me_D]}\me_B ~, \\
\RM{A}{B}{C}{D} &= \RRM{A}{B}{C}{D} = \ip{R(e_C,e_D)e_B}{\me_A} ~.
\end{align*}
What follows are some basic properties of the geometry of a submanifold.  The details can be found in, for example \cite[ch. 6]{ref_doCarmo}.
\begin{enumerate}
\item $\nsub$ is the projection of $\namb$ onto $T\Sm\subset TM|_\Sm$, and $\nnor$ is the projection of $\namb$ onto $\N\Sm\subset TM|_\Sm$.  Their curvatures are denoted by
\begin{align*}
\RM{\;k}{l}{i}{j}^\Sm &= \ip{\nsub_{\me_i}\nsub_{\me_j}\me_l - \nsub_{\me_j}\nsub_{\me_i}\me_l - \nsub_{[\me_i,\me_j]}\me_l}{\me_k} ~, \\
\RM{\;\af}{\bt}{i}{j}^\perp &= \ip{\nnor_{\me_i}\nnor_{\me_j}\me_\bt - \nnor_{\me_j}\nnor_{\me_i}\me_\bt - \nnor_{[\me_i,\me_j]}\me_\bt}{\me_\af} ~,
\end{align*}
\item Given any two tangent vectors $X, Y$ of $\Sm$, the \emph{second fundamental form} of $\Sm$ in $M$ is defined by $\II(X,Y) = (\namb_X Y)^\perp$, where $(\cdot)^\perp: TM\rightarrow N\Sigma$ is the projection onto the normal bundle.   The \emph{mean curvature} of $\Sm$ is the normal vector field defined by $H = \tr_\Sm\II$.  With a normal vector $V$, $\II(X,Y,V)$ is defined to be $\ip{\II(X,Y)}{V} = \ip{\namb_X Y}{V}$.  In terms of the frame,
\begin{align*}
\hm{\af}{i}{j} = \II(\me_i,\me_j,\me_\af)  \quad\text{and}\quad  H = \hm{\af}{i}{i}\,\me_\af ~.
\end{align*}
\item For any tangent vectors $X,Y,Z$ of $\Sm$ and a normal vector $V$, the Codazzi equation says that
\begin{align}
\ip{R(X,Y)Z}{V} &= (\namb_X\II)(Y,Z,V) - (\namb_Y\II)(X,Z,V)
\label{Codazzi1} \end{align}
where
\begin{align}
(\namb_X\II)(Y,Z,V) &= X\left( \II(Y,Z,V) \right) - \II(\nsub_X Y,Z,V) - \II(Y, \nsub_X Z,V) - \II(X,Y,\nnor_X V) ~.
\label{Codazzi2} \end{align}
In terms of the frame, denote $(\namb_{\me_i}\II)(\me_j,\me_k,\me_\af)$ by $\hmd{\af}{j}{k}{i}$, and \eqref{Codazzi1} is equivalent to that $\RM{\af}{k}{i}{j} = \hmd{\af}{j}{k}{i} - \hmd{\af}{i}{k}{j}$.
\end{enumerate}

\subsection{Geodesic coordinate and geodesic frame}\label{sec_coordinate}

For any $p\in\Sm$, we can construct a ``partial" geodesic coordinate and a geodesic frame on a neighborhood of $p$ in $M$ as follows:

\begin{enumerate}
\item Choose an oriented, orthonormal basis $\{\me_1,\cdots,\me_n\}$ for $T_p\Sm$.  The map
\begin{align*}
F_0: \bx = (x^1,\cdots, x^n) &\mapsto \exp^\Sm_p(x^j\me_j)
\end{align*}
parametrizes an open neighborhood of $p$ in $\Sm$, where $\exp^\Sm$ is the exponential map of the induced metric on $\Sm$.  For any $\bx$ of unit length, the curve $\gm(t) = F_0(t\bx)$ is called a \emph{radial geodesic on $\Sm$} (\emph{at $p$}).  By using $\nsub$ to parallel transport $\{\me_1,\cdots,\me_n\}$ along these radial geodesics, we get a local orthonormal frame for $T\Sm$ on a neighborhood of $p$ in $\Sm$.  The frame is still denoted by $\{\me_1,\cdots, \me_n\}$.
\item Choose an orthonormal basis $\{\me_{n+1},\cdots,\me_{n+m}\}$ for $\N_p\Sm$.  By using $\nnor$ to parallel transport $\{\me_{n+1},\cdots,\me_{n+m}\}$ along radial geodesics on $\Sm$, we obtain a local orthonormal frame for $\N\Sm$ on a neighborhood of $p$ in $\Sm$.  This frame is still denoted by $\{\me_{n+1},\cdots,\me_{n+m}\}$. It is clear that $\{\me_1,\cdots,\me_n,$ $\me_{n+1},\cdots,\me_{n+m}\}$ is a local orthonormal frame for $TM|_\Sm$.
\item The map
\begin{align*}
F: (\bx,\by) = \big((x^1,\cdots, x^n), (y^{n+1},\cdots,y^{n+m})\big) &\mapsto \exp_{F_0(\bx)}(y^\af \me_{\af})
\end{align*}
parametrizes an open neighborhood of $p$ in $M$.  The map $\exp$ is the exponential map of $(M,g)$.  For any $\by$ of unit length, the curve $\sm(t) = F(\bx,t\by) = \exp_{F_0(\bx)}(t\by)$ is called a \emph{normal geodesic for $\Sm\subset M$}.
\item For any $\bx$, step (ii) gives an orthonormal basis $\{\me_1,\cdots,\me_{n+m}\}$ for $T_{F(\bx,0)}M$.  By using $\namb$ to parallel transport it along normal geodesics, we have an orthonormal frame for $TM$ on a neighborhood of $p$ in $M$.  This frame is again denoted by $\{\me_1,\cdots,\me_{n+m}\}$.
\end{enumerate}
The freedom in the above construction is the choice of $\{\me_1,\cdots,\me_n\}$ and $\{\me_{n+1},\cdots,\me_{n+m}\}$ at $p$, which is $\RO(n)\times\RO(m)$.  A particular choice will be made later on.  {When $\Sm$ is \emph{oriented}, $\{\me_1,\cdots,\me_n\}$ is required to form an oriented frame.  In this case, the freedom is $\RSO(n)\times\RO(m)$.}

\begin{rmk}
We will consider the curves $s \mapsto\exp_p^\Sm(x^i\me_i+s\me_j)$ and $s \mapsto \exp_{F_0(\bx)}(y^\bt\me_\bt + s\me_\af)$ in the following discussion.  They will be abbreviated as $F_0(\bx+s\me_j)$ and $F(\bx,\by+s\me_\af)$, respectively.
\end{rmk}

\begin{rmk}
The frames $\{e_1,\cdots,e_n,e_{n+1},\cdots,e_{n+m}\}$ are constructed by parallel transport along radial geodesics on $\Sm$ and then normal geodesic for $\Sm$.  They are indeed {smooth}.  We briefly explain the smoothness of $\{e_1,\cdots,e_n\}$ on a neighborhood of $p$ in $\Sm$.  Write $e_i = S_{ij}(\bx)\px{j}$.  The smoothness of the frame is equivalent to the smoothness of $S_{ij}(\bx)$.  Let $\Gm_{jk}^l(\bx)$ be the Christoffel symbols of $\nsub$, i.e. $\nsub_{\px{j}}{\px{k}} = \Gm_{jk}^l(\bx)\px{l}$.  The Christoffel symbols $\Gm_{jk}^l(\bx)$ are smooth functions.  Since $e_i$ is parallel along radial geodesics,
\begin{align*}
\nsub_{x^l\px{l}}e_i = 0 = \left( x^l\frac{\pl\,S_{ij}(\bx)}{\pl x^l} + x^l S_{ik}(\bx)\Gm_{ik}^j(\bx) \right) \px{j} ~.
\end{align*}
To avoid confusion, fix $\xi = (\xi^1,\cdots,\xi^n)\in\BR^n$.  Let $\gm(t) = t\xi$ for $t\in[0,1]$.  Since $\frac{\dd}{\dd t}f(\gm(t)) = \frac{1}{t}(x^l\frac{\pl}{\pl x^l}f(\bx))|_{\gm(t)}$,
\begin{align*}
\frac{\dd S_{ij}(t\xi)}{\dd t} &= \xi^l\,S_{ik}(t\xi)\,\Gm_{ik}^j(t\xi) ~.
\end{align*}
In other words, $[S_{ij}(\xi)]$ is the solution to the ODE system of the form $\frac{\dd S}{\dd t} = F(S,t,\xi)$ at $t=1$, with the identity as the initial condition.  Therefore, $S_{ij}(\xi)$ is smooth in $\xi$.
\end{rmk}

\subsubsection{The tubular neighborhood $U_\vep$ and the distance function}\label{sec_distance}

\begin{defn} \label{u_epsilon}
For any $\dt>0$, let $U_{\dt}$ be the image of $\{V\in\N\Sm~|~|V|<\dt\}$ under the exponential map along $\Sigma$.
By the implicit function theorem, there exists $\vep>0$, which is determined by the geometry of $\Sm$ and $M$, such that the following statements hold for $U_\vep$:

(1) The map $\exp:\{V\in\N\Sm~|~|V|<2\vep\}\to U_{2\vep}$ is a diffeomorphism. 

(2) There exist the local coordinate system $(x^1, \cdots x^n, y^{n+1}, \cdots y^{n+m})$ and the frame $ \{ e_1, \cdots, e_{n+m}\}$ constructed in the last subsection.

(3) The function $\sum_{\af}(y^\af)^2$ is a well-defined smooth function on $U_{\vep}$. 

(4) On $U_{\vep}$, the square root of $\sum_{\af}(y^\af)^2$ is the distance function to $\Sm$.

(5) For any $q\in U_\vep$, there exists a unique $p\in\Sm$ such that there is a unique normal geodesic in $U_\vep$ connecting $p$ and $q$.
\end{defn}

We now analyze the gradient of the function $\sum_{\af}(y^\af)^2$.  To avoid confusion, let
\begin{align*}
\xi=(\xi^1,\cdots,\xi^n)\in\BR^n  \qquad\text{and}\qquad  \eta = (\eta^{n+1},\cdots,\eta^{n+m})\in\BR^m
\end{align*}
be constant vectors.  Consider the normal geodesic $\sm(t) = F(\xi, t\eta)$; its tangent vector field is $\sm'(t) = \eta^\af\py{\af}$.  On the other hand, $\sm'(0)$ is also equal to $\eta^\af\me_\af$, and $\eta^\af\me_\af$ is defined and parallel along $\sm(t)$.  Thus, $\eta^\af\py{\af} = \eta^\af e_\af$ on $\sm(t)$.  Since the $y$-coordinate of $\sm(t)$ is $t\eta$, we find that
\begin{align}
y^\af\py{\af}\big|_{\sm(t)} &= t\eta^\af\py{\af}\big|_{\sm(t)} = t\eta^\af\,e_{\af} = t\,\sm'(t) ~; \label{tangent}
\end{align}
at $t=1$ it gives
\begin{align}
y^\af\py{\af} &= y^\af\me_\af ~. \label{grad1}
\end{align}

By modifying the standard geodesic argument \cite[p.4--9]{ref_CB}, the vector field $y^\af\py{\af}|_{\sm(t)}$ is half of the gradient vector field of $\sum_\af(y^\af)^2$.  In addition, note that \eqref{grad1} implies that $\ip{y^\af\py{\af}}{y^\af\py{\af}} = \sum_\af(y^\af)^2$.  The Gauss lemma implies that $\ip{y^\af\py{\af}}{s^\bt \py{\bt}} = 0$ if $\sum_{\af}y^\af s^\af = 0$.  By considering the first variational formula of the one-parameter family of geodesics $\sm(t,s) = \exp_{F_0(\xi+s\me_j)}(t\eta)$, one finds that $\ip{y^\af \py{\af}}{\px{j}} = 0$.  It follows from these relations that
\begin{align}
\namb\left(\sum_\af(y^\af)^2\right) = 2y^\af\py{\af} ~.
\label{grad2} \end{align}

For a locally defined smooth function near $p$, the following lemma establishes its expansion in terms of the coordinate system constructed above.

\begin{lem} \label{Taylor}
Let $U_\vep$ be a neighborhood of $p\in\Sm$ in $M$ as in Definition \ref{u_epsilon} with the coordinate system $(\bx,\by)=(x^1, \cdots x^n, y^{n+1}, \cdots y^{n+m})$ and the frame $\{e_1, \cdots, e_n, e_{n+1}, \cdots e_{n+m}\}$.  Then, any smooth function $f(\bx,\by)$ on $U_\vep$ has the following expansion:
\begin{align*}
f(\bx,\by) = f(\bo,\bo) + x^i\,e_i(f)|_p + y^\af\,e_\af(f)|_p + \CO(|\bx|^2 + |\by|^2) ~.
\end{align*}
More precisely, it means that $\left| f(\bx,\by) - f(\bo,\bo) - x^i\,e_i(f)|_p - y^\af\,e_\af(f)|_p \right| \leq c(|\bx|^2 + |\by|^2)$ for some constant $c$ determined by the $\CC^2$-norm of $f$ and the geometry of $M$ and $\Sm$.
\end{lem}

\begin{proof}
Let $q\in U_\vep$ be any point.  To avoid confusion, denote the coordinate of $q$ by $(\xi,\eta)$, where $\xi\in\BR^n$ and $\eta\in\BR^m$ are regarded as constant vectors.  Let $q_0\in\Sm$ be the point with normal coordinate $(\xi,\bo)$, and consider the radial geodesic on $\Sm$ joining $q_0$ and $p$, $\sm_0(t) = F_0(t\xi)$.  Applying Taylor's theorem on $f(\sm_0(t))$ gives
\begin{align*}
f(\xi,\bo) = f(\bo,\bo) + \frac{\dd}{\dd t}\big|_{t=0}f(\sm_0(t)) + \int_0^1(1-t)\frac{\dd^2\,f(\sm_0(t))}{\dd t^2}\,\dd t~.
\end{align*}
Since $\sm_0'(t) = \xi^i\,e_i$, we find that
\begin{align}
f(\xi,\bo) = f(\bo,\bo) + \xi^i\,e_i(f)|_p + \xi^i\xi^j\int_0^1(1-t)\,(e_j(e_i(f)))(\sm_0(t))\,\dd t ~.
\label{Taylor1} \end{align}

Next, consider the normal geodesic joining $q$ and $q_0$, $\sm(t) = F(\xi, t\eta)$.  Remember that $\sm'(t) = \eta^\af e_\af$.  By considering $f(\sm(t))$,
\begin{align}
f(\xi,\eta) &= f(\xi,\bo) + \eta^\af\,(e_\af(f))|_{q_0} + \eta^\af\eta^\bt\int_0^1(1-t)\,(e_\bt(e_\af(f)))(\sm(t))\,\dd t ~.
\label{Taylor2} \end{align}
Similar to \eqref{Taylor1}, $(e_\af(f))|_{q_0} = (e_\af(f))|_{p} + \xi^j\int_0^1e_j(e_\af(f))(\sm_0(t))\,\dd t$.  Putting these together finishes the proof of this lemma.
\end{proof}

\subsubsection{The expansions of coordinate vector fields}

\begin{lem}
Let $U_\vep$ be a neighborhood of $p\in\Sm$ in $M$ as in Definition \ref{u_epsilon} with the coordinate system $(\bx,\by)=(x^1, \cdots x^n, y^{n+1}, \cdots y^{n+m})$ and the frame $\{e_1, \cdots, e_n, e_{n+1}, \cdots e_{n+m}\}$.  Write
\begin{align*}
\px{i} = \ip{\px{i}}{e_A}e_A  \quad\text{ and }\quad  \py{\mu} = \ip{\py{\mu}}{e_A}e_A ~,
\end{align*}
then $\ip{\px{i}}{e_A}$ and $\ip{\py{\mu}}{e_A}$, considered as locally defined multi-indexed functions, has the following expansions:
\begin{align}  \begin{split}
\ip{\px{i}}{e_j}\big|_{(\bx, \by)} &= \dt_{ij} - y^\af \hm{\af}{i}{j}\big|_p + \CO(|\bx|^2 + |\by|^2) ~,\\
\ip{\py{\mu}}{e_\bt}\big|_{(\bx, \by)} &= \dt_{\mu\bt} + \CO(|\bx|^2 + |\by|^2) ~,
\end{split} \label{vf_expansion}  \end{align}
and both $\ip{\px{i}}{e_\bt}\big|_{(\bx, \by)}$ and $\ip{\py{\mu}}{e_j}\big|_{(\bx, \by)}$ are of the order $|\bx|^2 + |\by|^2$.  By inverting the matrices,
\begin{align}
e_i = \frac{\pl}{\pl x^i} + y^\af\hm{\af}{i}{j}\frac{\pl}{\pl x^j} + \CO(|\bx|^2 + |\by|^2)
\qquad\text{and}\qquad  e_\af &= \frac{\pl}{\pl y^\af} +  \CO(|\bx|^2 + |\by|^2) ~.
\label{vf_expansion1} \end{align}
\end{lem}

\begin{proof}
We apply Lemma \ref{Taylor} to these locally defined functions.

By construction, $\ip{\px{i}}{e_j}|_p = \dt_{ij}$.  With a similar argument as that for \eqref{grad1}, $x^i\px{i} = x^i e_i$ on $\Sm\cap U_\vep$.  It follows that
\begin{align*}
x^j &= x^\ell\ip{\px{\ell}}{e_j} ~.
\end{align*}
Differentiating the above equation first with respect to $x^i$ and then with respect to $x^k$, and then evaluating at $p$ which has $x^\ell=0$ for all $\ell$, we obtain
\begin{align*}
\left.\left(\px{k}\ip{\px{i}}{e_j}\right)\right|_p + \left.\left(\px{i}\ip{\px{k}}{e_j}\right)\right|_p & = 0 ~.
\end{align*}
On the other hand, it follows from the construction that $(\nsub e_j)|_p = 0$, and
\begin{align*}
\left.\left(\px{k}\ip{\px{i}}{e_j}\right)\right|_p = \left.\ip{\nsub_{\px{k}}\px{i}}{e_j}\right|_p = \left.\ip{\nsub_{\px{i}}\px{k}}{e_j}\right|_p = \left.\left(\px{i}\ip{\px{k}}{e_j}\right)\right|_p ~.
\end{align*}
Hence, $\px{k}\ip{\px{i}}{e_j}$ is zero at $p$.

Since $e_A$ is parallel with respect to $\namb$ along normal geodesics and $e_\af = \py{\af}$ at $p$, $(\namb_{e_\af}e_A)|_p = 0 = (\namb_{\py{\af}}e_A)|_p$.  It follows that
\begin{align*}
\left.\left(\py{\af}\ip{\px{i}}{e_j}\right)\right|_p &= \left.\ip{\namb_{\py{\af}}\px{i}}{e_j}\right|_p = \left.\ip{\namb_{\px{i}}\py{\af}}{e_j}\right|_p = -\left.\ip{\py{\af}}{\namb_{\px{i}}e_j}\right|_p = -\hm{\af}{i}{j}\big|_p
\end{align*}
where the third equality follows from the fact that $\ip{\py{\af}}{e_j}\equiv0$ on $\Sm\cap U_\vep$.

Note that $\ip{\px{i}}{e_\bt}$ vanishes on $\Sm\cap U_\vep$.  Since $e_\bt$ is parallel with respect to $\namb$ along normal geodesics, $(\namb_{e_\af} e_\bt)|_p = 0$, and then
\begin{align*}
\left.\left(\py{\af}\ip{\px{i}}{e_\bt}\right)\right|_p &= \left.\ip{\namb_{\py{\af}}\px{i}}{e_\bt}\right|_p = \left.\ip{\namb_{\px{i}}\py{\af}}{e_\bt}\right|_p ~.
\end{align*}
By construction, $\py{\af} = e_\af$ on $\Sm\cap U_\vep$ and $(\nnor e_\af)|_p = 0$.  Therefore, $\py{\af}\ip{\px{i}}{e_\bt}$ is zero at $p$.

The term $\ip{\py{\mu}}{e_j}$ also vanishes on $\Sm\cap U_\vep$.  It follows from \eqref{grad1} that $y^\mu\ip{\py{\mu}}{e_j} = 0$. Differentiating the above equation first with respect to $y^\alpha$ and then with respect to $y^\beta$, we obtain \begin{align*}
\left.\left(\py{\af}\ip{\py{\bt}}{e_j}\right)\right|_p + \left.\left(\py{\bt}\ip{\py{\af}}{e_j}\right)\right|_p = 0 ~.
\end{align*}
Since $\namb_{e_\nu} e_j = 0$, the above two terms are always equal to each other, and thus both vanish.

For $\ip{\py{\mu}}{e_\bt}$, it follows from the construction that $\ip{\py{\mu}}{e_\bt} = \dt_{\mu\bt}$ on $\Sm\cap U_\vep$.  According to \eqref{grad1}, $y^\mu = y^\nu\ip{\py{\nu}}{e_\mu}$.  By a similar argument as that for $\px{k}\ip{\px{i}}{e_j}$, $\py{\nu}\ip{\py{\mu}}{e_\bt}$ also vanishes at $p$.
\end{proof}

\subsubsection{The expansions of connection coefficients}

\begin{prop}\label{exp_connection}
Let $U_\vep$ be a neighborhood of $p\in\Sm$ in $M$ as in Definition \ref{u_epsilon} with the coordinate system $(\bx,\by)=(x^1, \cdots x^n, y^{n+1}, \cdots y^{n+m})$ and the frame $\{e_1, \cdots, e_n, e_{n+1}, \cdots e_{n+m}\}$.
 Let
$$ \ta_A^B = \ip{\namb_{e_C} e_A}{e_B} \om^C = \ta_A^B(e_C) \om^C $$
be the connection $1$-forms of the frame fields on $U_\vep$, where $\{\om^A\}_{A=1}^{n+m}$ is the dual coframe of $\{e_A\}_{A=1}^{n+m}$.  Then, at a point $q\in U_\vep$ with coordinates $(\bx, \by)$,  
$\theta_A^B(e_C)$, considered as locally defined multi-indexed functions, has the following expansions:
\begin{align}
\begin{split}
\ta_i^j (e_k) |_{(\bx, \by)} &= \oh x^l\,\RM{j}{i}{l}{k}^\Sm\big|_p + y^\af\,\RM{j}{i}{\af}{k}\big|_p + \CO(|\bx|^2+|\by|^2) ~, \\
\ta_i^j(e_\bt) |_{(\bx, \by)} &=  \oh y^\af\,\RM{j}{i}{\af}{\bt}\big|_p + \CO(|\bx|^2+|\by|^2) ~,  \end{split} \label{conn01} \\
\begin{split}
\ta_i^\af (e_j) |_{(\bx, \by)} &= \hm{\af}{i}{j}\big|_p + x^k\,\hmd{\af}{i}{j}{k}\big|_p + y^\bt\,(\RM{\af}{i}{\bt}{j}+\sum_k h_{\alpha ik} h_{\beta jk})\big|_p + \CO(|\bx|^2 + |\by|^2) ~, \end{split} \label{conn02} \\
\begin{split}
\ta_i^\af (e_\bt) |_{(\bx, \by)} &=  \oh y^\gm\,\RM{\af}{i}{\gm}{\bt}\big|_p + \CO(|\bx|^2 + |\by|^2) ~, \\
  \ta_\bt^\af (e_i) |_{(\bx, \by)} &= \oh x^j\,\RM{\af}{\bt}{j}{i}^\perp\big|_p + y^\gm\,\RM{\af}{\bt}{\gm}{i}\big|_p 
 + \CO(|\bx|^2+|\by|^2) ~\\
\ta_\bt^\af (e_\gm)|_{(\bx, \by)} &=  \oh y^\dt\,\RM{\af}{\bt}{\dt}{\gm}\big|_p  + \CO(|\bx|^2+|\by|^2) ~, \end{split} \label{conn03}
\end{align}
where $\RM{j}{i}{l}{k}^\Sm\big|_p$, $\RM{j}{i}{\af}{k}\big|_p$, $\RM{j}{i}{\af}{\bt}\big|_p$ $\hm{\af}{i}{j}\big|_p$, $\hmd{\af}{i}{j}{k}\big|_p$,
$ \RM{\af}{i}{\bt}{j}\big|_p$, $\RM{\af}{i}{\gm}{\bt}\big|_p$,$\RM{\af}{\bt}{j}{i}^\perp\big|_p$, $\RM{\af}{\bt}{\gm}{i}\big|_p$, $\RM{\af}{\bt}{\dt}{\gm}\big|_p$
all represent the evaluation of the corresponding tensors at $p$ and with respect to the frame fields $\{e_i\}_{i=1}^n$ and $\{e_\af\}_{\af=n+1}^{n+m}$. 
\end{prop}

\begin{proof}
Since the restriction of the frame $\{e_i\}_{i=1}^n$ on $\Sm$ is parallel with respect to $\nsub$ along the radial geodesics, $x^k\ta_i^j(e_k)\big|_{(\bx,\bo)} = 0$ for any $i,j\in\{1,\ldots,n\}$.  It follows that
\begin{align}
\ta_i^j(e_k)\big|_{(\bx,\bo)} &= -x^l\frac{\pl\, \ta_i^j(e_l)}{\pl x^k}\big|_{(\bx,\bo)}
\quad\text{and thus }~ \ta_i^j(e_k)\big|_{p} = 0 ~.
\label{geod01} \end{align}
By taking the partial derivative in $x^l$ and evaluating at $p=(\bo,\bo)$, we find that
\begin{align}
\frac{\pl\,\ta_i^j(e_k)}{\pl x^l}\big|_p = - \frac{\pl\, \ta_i^j(e_l)}{\pl x^k}\big|_p ~, \quad\text{or equivalently, }\quad
e_l(\ta_i^j(e_k))|_p = - e_k(\ta_i^j(e_l))|_p
\label{geod02} \end{align}
Similarly, since the restriction of $\{e_\mu\}_{\mu = n+1}^{n+m}$ on $\Sm$ is parallel with respect to $\nnor$ along radial geodesics, $x^k\ta_\mu^\nu(e_k)\big|_{(\bx,\bo)} = 0$ for any $\mu,\nu\in\{n+1,\ldots,n+m\}$.  It follows that
\begin{align}
\ta_\mu^\nu(e_k)\big|_p &= 0  \qquad\text{and}  \label{geod07} \\
e_l(\ta_\mu^\nu(e_k))\big|_p &= - e_k(\ta_\mu^\nu(e_l))\big|_p ~.  \label{geod08}
\end{align}

Since the frame $\{e_A\}_{i=1}^{n+m}$ is parallel with respect to $\namb$ along normal geodesics, $y^\mu\ta_A^B(e_\mu) = 0$ and it follows that
\begin{align}
\ta_A^B(e_\mu) = - y^\nu\frac{\pl\,\ta_A^B(e_\nu)}{\pl y^\mu} \quad\Rightarrow\quad \ta_A^B(e_\mu)\big|_{(\bx,\bo)} = 0.
\label{geod03} \end{align}
By taking partial derivatives,
\begin{align}
\frac{\pl\,\ta_A^B(e_\mu)}{\pl x^k} = - y^\nu\frac{\pl^2\,\ta_A^B(e_\nu)}{\pl x^k\pl y^\mu}  \quad\text{and}\quad
\frac{\pl\,\ta_A^B(e_\mu)}{\pl y^\nu} = - \frac{\pl\,\ta_A^B(e_\nu)}{\pl y^\mu} - y^\dt \frac{\pl^2\,\ta_A^B(e_\dt)}{\pl y^\nu\pl y^\mu} ~.
\label{geod04} \end{align}
Note that on $\Sm$, $\{\px{i}\}_{i=1}^n$ and $\{e_i\}_{i=1}^n$ are both bases for $T\Sm$.  Therefore,
\begin{align}
e_k(\ta_A^B(e_\mu))\big|_{(\bx,\bo)} = 0 ~.
\label{geod05} \end{align}
By construction, $\py{\mu} = e_\mu$ on $\Sm = \{y^\mu = 0 \text{ for all }\mu\}$.  It follows from \eqref{geod04} that
\begin{align}
e_\nu(\ta_A^B(e_\mu))\big|_{(\bx,\bo)} = - e_\mu(\ta_A^B(e_\nu))\big|_{(\bx,\bo)} ~.
\label{geod06} \end{align}

In terms of the connection $1$-forms, the components of the Riemann curvature tensor are
\begin{align}
&\RM{A}{B}{C}{D} \notag \\
=\,& \ip{\namb_{\me_C}\namb_{\me_D}\me_B - \namb_{\me_D}\namb_{\me_C}\me_B - \namb_{[\me_C,\me_D]}\me_B}{\me_A} \notag \\
=\,& e_C(\ta_B^A(e_D)) - e_D(\ta_B^A(e_C)) - (\ta_B^E\w\ta_E^A)(e_C,e_D) - \ta_B^A(e_E)\ta_D^E(e_C)  + \ta_B^A(e_E)\ta_C^E(e_D) ~.
\label{curv} \end{align}

With these preparations, we proceed to prove all the expansion formulae:

(The expansion of $\ta_i^j(e_k)$)~  It follows from \eqref{geod01} that the zeroth order term is zero.  By \eqref{geod02}, the coefficient of $x^l$ in the expansion is
\begin{align*}
e_l(\ta_i^j(e_k))|_p &= \oh\left.\left[e_l(\ta_i^j(e_k)) - e_k(\ta_i^j(e_l))\right]\right|_p = \oh\left.\RM{j}{i}{l}{k}^\Sm\right|_p ~.
\end{align*}
Note that for $\RM{\af}{\bt}{j}{i}^\Sm$, all the indices of summation in \eqref{curv} go from $1$ to $n$.  Due to \eqref{geod05}, the coefficient of $y^\af$ in the expansion is
\begin{align*}
\left.e_\af(\ta_i^j(e_k))\right|_{p} = \left.\left[ e_\af(\ta_i^j(e_k)) - e_k(\ta_i^j(e_\af))\right]\right|_p = \left.\RM{j}{i}{\af}{k}\right|_p ~.
\end{align*}

(The expansion of $\ta_i^j(e_\bt)$)~  By \eqref{geod03}, the zeroth order term is zero, and the coefficient of $x^l$ in the expansion is zero.  According to \eqref{geod06}, the coefficient of $y^\af$ in the expansion is
\begin{align*}
\left.e_\af(\ta_i^j(e_\bt))\right|_{p} = \oh\left.\left[ e_\af(\ta_i^j(e_\bt)) - e_\bt(\ta_i^j(e_\af))\right]\right|_p = \left.\RM{j}{i}{\af}{\bt}\right|_p ~.
\end{align*}

(The expansion of $\ta_i^\af(e_j)$)~  On $\Sm\cap U_\vep$, $\ta_i^\af(e_j) = \ip{\namb_{e_j}e_i}{e_\af} = \hm{\af}{i}{j}$.  Its derivative along $e_k$ is
\begin{align*}
e_k(\II(e_i,e_j,e_\af)) &= (\namb_{e_k}\II)(e_i,e_j,e_\af) + \II(\nsub_{e_k}e_i,e_j,e_\af) + \II(e_i,\nsub_{e_k}e_j,e_\af) + \II(e_i,e_j,\nnor_{e_k}e_\af) ~.
\end{align*}
Due to \eqref{geod01} and \eqref{geod07}, the last three terms vanish at $p$.  It follows that $e_k(\II(e_i,e_j,e_\af))|_p$ is equal to $\hmd{\af}{i}{j}{k}|_p$.

The coefficient of $y^\bt$ is $e_\bt(\ta_i^\af(e_j))|_p$.  By \eqref{geod05} and \eqref{geod03},
\begin{align*}
\left.\RM{\af}{i}{\bt}{j}\right|_p &= \left.\left[ e_\bt(\ta_i^\af(e_j)) + \ta_i^\af(e_k)\ta_\bt^k(e_j) \right]\right|_p = \left.\left[ e_\bt(\ta_i^\af(e_j)) - \hm{\af}{i}{k}\hm{\bt}{k}{j} \right]\right|_p ~.
\end{align*}

(The expansion of $\ta_i^\af(e_\bt)$)~  According to \eqref{geod03}, the zeroth order term is zero, and the coefficient of $x^l$ in the expansion is zero.  By \eqref{geod06} and \eqref{geod03},
\begin{align*}
\left.e_\gm(\ta_i^\af(e_\bt))\right|_p &= \oh\left.\left[ e_\gm(\ta_i^\af(e_\bt)) - e_\bt(\ta_i^\af(e_\gm)) \right]\right|_p = \oh\left.\RM{\af}{i}{\gm}{\bt}\right|_p ~.
\end{align*}

(The expansion of $\ta_\bt^\af(e_i)$)~  By \eqref{geod07}, the zeroth order term vanishes.  With \eqref{geod08}, \eqref{geod07} and \eqref{geod01},
\begin{align*}
\left.e_j(\ta_\af^\bt(e_i))\right|_p &= \oh\left.\left[ e_j(\ta_\af^\bt(e_i)) - e_i(\ta_\af^\bt(e_j)) \right]\right|_p = \oh\left.\RM{\af}{\bt}{j}{i}^\perp\right|_p ~.
\end{align*}
Note that for $\RM{\af}{\bt}{j}{i}^\perp$, the index of summation in the third term of \eqref{curv} goes from $n+1$ to $n+m$, and the indices of summation in the last two terms of \eqref{curv} go from $1$ to $n$.  By \eqref{geod05}, \eqref{geod03} and \eqref{geod07},
\begin{align*}
\left.e_\gm(\ta_\af^\bt(e_i))\right|_p &= \left.\left[ e_\gm(\ta_\af^\bt(e_i)) - e_i(\ta_\af^\bt(e_\gm)) \right]\right|_p = \left.\RM{\af}{\bt}{\gm}{i}\right|_p ~.
\end{align*}

(The expansion of $\ta_\bt^\af(e_\gm)$)~  Due to \eqref{geod03}, $\ta_\bt^\af(e_\gm)$ vanishes on $\Sm\cap U_\vep$.  According to \eqref{geod06} and \eqref{geod03},
\begin{align*}
\left.e_\dt(\ta_\bt^\af(e_\gm))\right|_p &= \oh \left.\left[ e_\dt(\ta_\bt^\af(e_\gm)) - e_\gm(\ta_\bt^\af(e_\dt)) \right]\right|_p = \oh\left.\RM{\af}{\bt}{\dt}{\gm}\right|_p ~.
\end{align*}
This finishes the proof of this proposition.
\end{proof}

\subsubsection{Horizontal and vertical subspaces}\label{sec_linear}

For any $q\in U_\vep\subset M$, there exists a unique $p\in\Sm$ such that there is a unique normal geodesic inside $U_\vep$ connecting $q$ and $p$.  Any tensor defined on $\Sm$ can be extended to $U_\vep$ by parallel transport of $\namb$ along normal geodesics.  Here are some notions that will be used in this paper.

Assume that $\Sm$ is oriented.
The parallel transport of $T\Sm$ along normal geodesics defines an $n$-dimensional distribution of $TM|_{U_\vep}$, which is called the \emph{horizontal distribution}, and is denoted by $\CH$.  Its orthogonal complement in $TM$ is called the \emph{vertical distribution}, and is denoted by $\CV$.  It is clear that $\CH = \spn\{\me_1,\cdots,\me_n\}$ and $\CV = \spn\{\me_{n+1},\cdots,\me_{n+m}\}$.  The parallel transport of the volume form of $\Sm$ along normal geodesics defines an $n$-form on $U_\vep$, which is  denoted by $\Om$.  Let $\{\om^1,\cdots,\om^n,$ $\om^{n+1},\cdots,\om^{n+m}\}$ be the dual coframe of $\{\me_1,\cdots,\me_n,$ $\me_{n+1},\cdots,\me_{n+m}\}$.  In terms of the coframe,
\begin{align}
\Om = \om^1\w\cdots\w\om^n ~.
\end{align}
From the construction, $\Om$ has comass $1$.  That is to say, the evaluation of $\Om$ on any oriented $n$-plane takes the value between $-1$ and $1$.

For any $q\in U_\vep$ and any oriented $n$-plane $L\subset T_q M$, consider the orthogonal projection onto $\CV_q$, $\pi_\CV$, and the evaluation of $\Om$ on $L$.  Suppose that $\Om(L)>0$.  By the singular value decomposition, there exist oriented orthonormal basis $\{\me_1,\cdots,\me_n\}$ for $\CH_q$, orthonormal basis $\{\me_{n+1},\cdots,\me_{n+m}\}$ for $\CV_q$ and angles $\phi_1,\cdots,\phi_n\in[0,\pi/2)$ such that
\begin{align}
\{\te_j = \cos\phi_j\,\me_j + \sin\phi_j\,\me_{n+j}\}_{j=1}^n
\label{linear1} \end{align}
constitutes an oriented, orthonormal basis for $L$.  If $n>m$, $\phi_j$ is set to be zero for $j>m$.  It follows that
\begin{align}
\Om(L) = \cos\phi_1\cdots\cos\phi_n ~,
\label{linear3} \end{align}
and the operator norm of $\pi_\CV$ is
\begin{align}
\fs(L) &:= \big|\!\big| \pi_\CV|_L \big|\!\big|_{\text{op}} = \max\{\sin\phi_1,\cdots,\sin\phi_n\} ~.
\label{linear4} \end{align}

\begin{rmk} \label{rmk_linear}
The construction \eqref{linear1} works for  $\Om(L) = 0$ as well, and some of the angles would be $\pi/2$.  The formulae \eqref{linear3} and \eqref{linear4} remain valid.  We briefly explain this linear-algebraic construction.  Consider the orthogonal projection onto $\CH_q$, $\pi_\CH$.  Let $L_\CV = \ker(\pi_\CH:L\to\CH_q)$; it is a linear subspace of $\CV_q$.  Let $L'$ be the orthogonal complement of $L_\CV$ in $L$.  Then, $L = L'\oplus L_\CV$, and $\pi_\CH:L'\to \CH_q$ is injective.  Note that $\pi_\CV(L')$ is orthogonal to $L_\CV$.  The linear subspace $L'$ is the graph of a linear map from $\pi_\CH(L')\subset \CH_q$ to $\CV_q$.  The basis \eqref{linear1} is constructed by applying the singular value decomposition to this linear map together with an orthonormal basis for $L_\CV$.
\end{rmk}

\begin{rmk} \label{rmk_o}
The singular value decomposition does not require orientation.  If $\Sm$ and $L$ are not assumed to be oriented, one can still construct the frame \eqref{linear1} for $L$ with $\phi_1,\cdots,\phi_n\in[0,\pi/2]$.  But to define $\Om(L)$, both $\Sm$ and $L$ have to be oriented.
\end{rmk}

It is easy to see that the orthogonal complement of $L$ has the following orthonormal basis:
\begin{align}
\{\te_\af = -\sin\phi_{\af}\,\me_{\af-n} + \cos\phi_{\af}\,\me_{\af}\}_{\af=n+1}^m
\label{linear2} \end{align}
where $\phi_\af = \phi_{\af-n}$.  If $m>n$, $\phi_{\af}$ is set to be zero for $\af>2n$.  The following estimates will be needed later, and are straightforward to come by:
\begin{align} \label{linear5}
\sum_{i=1}^n \left| (\om^j\ot\om^k)(\te_i,\te_i) \right| \leq n ~,\quad
\sum_{i=1}^n \left| (\om^{\af}\ot\om^j)(\te_i,\te_i) \right| \leq n\fs
\end{align}
and
\begin{align}  \begin{split}
&\left| (\om^1\w \cdots \w \om^n)(\te_{\af}, \te_1,\cdots ,\widehat{\,\te_i\,}, \cdots, \te_n) \right| \leq \fs ~, \\
&\left| (\om^1\w \cdots \w \om^n)(\te_{\af}, \te_{\bt},  \te_1, \cdots, \widehat{\,\te_i\,}, \cdots, \widehat{\,\te_j\,} \cdots, \te_n) \right| \leq \fs^2 ~, \\
&\left| (\om^{\af}\w \om^1 \w \cdots \w \widehat{\om^i} \w \cdots \w \om^n)(\te_1,  \cdots, \te_n) \right|\leq n \fs ~, \\
&\left| (\om^{\af}\w \om^1 \w \cdots \w \widehat{\om^i} \w \cdots \w \om^n)(\te_{\bt}, \te_1,\cdots, \widehat{\,\te_j\,}, \cdots, \te_n) \right| \leq 1 ~, \\
&\left| (\om^{\af} \w \om^{\bt} \w \om^1\w \cdots \w\widehat{\om^i}\w \cdots\w \widehat{\om^j} \w \cdots \w \om^n)(\te_1,\cdots, \te_n) \right| \leq n(n-1)\fs^2 \\
\end{split} \label{linear6} \end{align}
for any $i,j,k \in \{1, \ldots, n\}$ and $\af,\bt\in\{n+1, \ldots, n+m\}$.

The above estimates are the zeroth order estimate.  For the first, third and fourth inequalities of \eqref{linear6}, a more refined version will also be needed.  It follows from \eqref{linear1} and \eqref{linear2} that
\begin{align}
(\om^1\w \cdots \w \om^n)(\te_{\af}, \te_1,\cdots ,\widehat{\,\te_i\,}, \cdots, \te_n) &= (-1)^{i}\delta_{\af(n+i)}\,\frac{\sin\phi_i}{\cos\phi_i}\,\prod_{k=1}^n\cos\phi_k ~.
\label{linear7} \end{align}
Let $\tom^1,\cdots,\tom^n,\tom^{n+1},\cdots,\tom^{n+m}$ be the dual basis of $\te_1,\cdots,\te_n,\te_{n+1},\cdots,\te_{n+m}$.  According to \eqref{linear1} and \eqref{linear2},
\begin{align*}
\om^j = \cos\phi_j\,\tom^j - \sin\phi_j\,\tom^{n+j}  \quad\text{and}\quad  \om^\af = \sin\phi_\af\,\tom^{\af-n} + \cos\phi_\af\,\tom^\af ~.
\end{align*}
Hence,
\begin{align} \begin{split}
& (\om^{\af}\w \om^1 \w \cdots \w \widehat{\om^i} \w \cdots \w \om^n)(\te_1,  \cdots, \te_n) \\
=\;& \left( (\sin\phi_\af\tom^{\af-n}) \w (\cos\phi_1\tom^1) \w \cdots \w (\widehat{\cos\phi_i\tom^i}) \w \cdots \w (\cos\phi_n\tom^n) \right)(\te_1,  \cdots, \te_n) \\
=\;& (-1)^{i+1}\delta_{\af(n+i)}\,\frac{\sin\phi_i}{\cos\phi_i}\,\prod_{k=1}^n\cos\phi_k ~.
\end{split} \label{linear8} \end{align}

It follows from
\begin{align*}
& (\om^{\af}\w \om^1 \w \cdots \w \widehat{\om^j} \w \cdots \w \om^n)(\te_{\af}, \te_1,\cdots, \widehat{\,\te_j\,}, \cdots, \te_n) \\
=\,& \om^\af(\te_\af)\cdot[(\om^1 \w \cdots \w \widehat{\om^j} \w \cdots \w \om^n)(\te_1,\cdots, \widehat{\,\te_j\,}, \cdots, \te_n)] \\
& + \sum_{k=1}^{j-1} (-1)^k\om^{\af}(\te_k)\cdot[(\om^1 \w \cdots \w \widehat{\om^j} \w \cdots \w \om^n)(\te_{\af},\te_1,\cdots, \widehat{\,\te_k\,},\cdots \widehat{\,\te_j\,}, \cdots, \te_n)] \\
&\quad + \sum_{k=j+1}^{n} (-1)^{k+1}\om^{\af}(\te_k)\cdot[(\om^1 \w \cdots \w \widehat{\om^j} \w \cdots \w \om^n)(\te_{\af},\te_1,\cdots, \widehat{\,\te_j\,},\cdots \widehat{\,\te_k\,}, \cdots, \te_n)] ~,
\end{align*}
that
\begin{align}
\left| (\om^{\af}\w \om^1 \w \cdots \w \widehat{\om^j} \w \cdots \w \om^n)(\te_{\af}, \te_1,\cdots, \widehat{\,\te_j\,}, \cdots, \te_n) - \frac{\cos\phi_\af}{\cos\phi_j} \,\prod_{k=1}^n\cos\phi_k \right| \leq (n-1)\,\fs^2 ~.
\label{linear9} \end{align}

\section{Minimal submanifolds and stability conditions}

\subsection{The stability of a minimal submanifold} \label{sec_stable}

A submanifold $\Sm\subset(M,g)$ is said to be \emph{minimal} if its mean curvature vanishes, $H = 0$.  It means that $\Sm$ is a critical point of the volume functional.  A minimal submanifold $\Sm$ is said to be \emph{strictly stable} if the second variation of the volume functional is positive at $\Sm$ (stable if the second variation is non-negative).  We now recall the second variational formula of the volume functional.  The detail can be found in \cite[\S3.2]{ref_Simons}.

Suppose that $V$ is a normal vector field on $\Sm$.  There are two linear operators on $\N\Sm$ in the second variation formula.  The first one is the partial Ricci operator defined by
\begin{align*}
\CR(V) &= \tr_\Sm \big( R(\,\cdot\,,V)\,\cdot\,\big)^\perp
\end{align*}
where $R$ is the Riemann curvature tensor of $(M,g)$.  The second one is basically the norm-square of the second fundamental form along $V$.  The shape operator along $V$ is a symmetric map from $T\Sm$ to itself, and is defined by
\begin{align*}
\CS_V(X) = -(\namb_X V)^{T\Sm} = -\namb_X V + \nnor_X V ~\text{, or equivalently}~ \ip{\CS_V(X)}{Y} = \ip{\II(X,Y)}{V} 
\end{align*}
for any tangent vectors $X$ and $Y$ of $\Sm$.  By regarding $\CS$ as a map from $N\Sm$ to $\text{Sym}^2(T\Sm)$, define
\begin{align*}
\CA(V) &= \CS^t\circ\CS(V)
\end{align*}
where $\CS^t:\text{Sym}^2(T\Sm)\to N\Sm$ is the transpose map of $\CS$.

With this understanding, the second variation of the volume functional in the direction of $V$ is
\begin{align}
\int_\Sm |\nnor V|^2 + \ip{\CR(V)}{V} - \ip{\CA(V)}{V}
\label{secondv} \end{align}
Therefore, $\Sm$ is strictly stable if and only if $(\nnor)^*\nnor + \CR - \CA$ is a positive operator.  Note that $(\nnor)^*\nnor$ is always non-negative definite, and $\CR - \CA$ is a linear map on $\N\Sm$.  Hence, the positivity of $\CR - \CA$ is a condition easier to check, and implies the strict stability of $\Sm$.

\begin{defn}
A minimal submanifold $\Sm\subset(M,g)$ is said to be \emph{strongly stable} if $\CR - \CA$ is a (pointwise) positive operator on $\N\Sm$.
\end{defn}

In terms of the notations introduced in \S\ref{sec_notation}, $\Sm$ is strongly stable if there exists a constant $c_0>0$ such that
\begin{align}
- \sum_{\af,\bt,i}\RM{i}{\af}{i}{\bt} v^\af v^\bt - \sum_{\af,\bt,i,j}\hm{\af}{i}{j}\hm{\bt}{i}{j}v^\af v^\bt \geq c_0\sum_{\af} (v^\af)^2
\label{sstable}\end{align}
for any $(v^{n+1},\cdots,v^{n+m})\in\BR^{m}$.

In particular, for a hypersurface $\Sigma$, the condition is 
\[-\Ric(\nu, \nu)-|A|^2\geq c_0,\] where $\nu$ is a unit normal and $|A|^2=\sum_{i, j} h_{ij}^2$.

\subsection{Proof of Proposition A}\label{sec_ss}

It is easy to see that \eqref{sstable} holds for a totally geodesic submanifold in a manifold with negative sectional curvature. 
When the geometry has special properties, the condition \eqref{sstable} is equivalent to some natural curvature condition on the minimal submanifold.

\subsubsection{Complex submanifolds in K\"ahler manifolds}

Let $(M^{2n},g,J,\om)$ be a K\"ahler manifold, and $\Sm^{2p}\subset M$ be a complex submanifold.  The submanifold $\Sm$ is automatically minimal.  In fact, the second variation \eqref{secondv} is always non-negative.  In this case, the operator $\CR-\CA$ was studied by Simons in the famous paper \cite[\S3.5]{ref_Simons}.  We briefly summarize his results.  The strongly stable condition is equivalent to that
\begin{align*}
\left\langle -J\left(\sum_{i=1}^p R^\perp(e_i,f_i)(V)\right), V \right\rangle \geq c_0\,|V|^2
\end{align*}
where $\{e_1,\cdots,e_p,f_1,\cdots,f_p\}$ is an orthonormal frame for $T\Sm$ with $f_i = Je_i$.  In other words, the normal bundle curvature contracting with $\om_\Sm$ is positive definite.  It implies that the normal bundle of $\Sm$ admits no non-trivial holomorphic cross section.

\subsubsection{Minimal Lagrangians in K\"ahler--Einstein manifolds}

Let $(M^{2n},g,\om)$ be a K\"ahler--Einstein manifold, where $\om$ is the K\"ahler form.  Denote the Einstein constant by $c$, i.e.\ $\Ric = c\,g$.  A half-dimensional submanifold $L^n\subset M$ is said to be Lagrangian if $\om|_L$ vanishes.  Suppose that $L$ is both minimal and Lagrangian.  Then, \eqref{sstable} is equivalent to the condition that
\begin{align}
\Ric^L - c  ~\text{ is a positive definite operator on }~ TL ~,
\end{align}
where $\Ric^L$ is the Ricci curvature of $g|_L$.  For completeness, the derivation is included in Appendix \ref{apx_Lag}.
We remark that when $c<0$, a minimal Lagrangian is always stable.  That is to say, the second variation \eqref{secondv} is strictly positive for any non-identically zero $V$; see \cite{ref_Chen, ref_Oh}.

A case of particular interest is \emph{special Lagrangians} in a Calabi--Yau manifold; see \cite[\S III]{ref_HL}.  The constant $c=0$ for a Calabi--Yau manifold, and the strong stability condition \eqref{sstable} is equivalent to the positivity of $\Ric^L$.  By the Bochner formula, it implies that the first Betti number of $L$ is zero.  According to the result of McLean \cite[Corollary 3.8]{ref_McLean}, $L$ is infinitesimally rigid as a special Lagrangian submanifold.

\subsubsection{Coassociatives in $G_2$ manifolds}

A $G_2$ manifold $(M,g)$ is a $7$-dimensional Riemannian manifold whose holonomy is contained in $G_2$.  A \emph{coassociative} submanifold is a special class of minimal, $4$-dimensional submanifold in $M$.  A complete story can be found in \cite[\S IV]{ref_HL} and \cite[ch.11--12]{ref_Joyce1}, and a brief summary is included in Appendix \ref{apx_coa}.

Suppose that $\Sm^4\subset M$ is coassociative.  The strong stability condition \eqref{sstable} is equivalent to that
\begin{align}
-2W_- + \frac{s}{3}  ~\text{ is a positive definite operator on }~ \Ld^2_-(\Sm) ~,
\label{coa_stable} \end{align}
where $W_-$ is the anti-self-dual part of the Weyl curvature of $g|_\Sm$, and $s$ is the scalar curvature of $g|_\Sm$.  The computation bears its own interest in $G_2$ geometry, and is included in Appendix \ref{apx_coa}.  According to the Weitzenb\"ock formula for anti-self-dual $2$-forms \cite[Appendix C]{ref_FUl}, \eqref{coa_stable} implies that $\Sm$ has no non-trivial anti-self-dual harmonic $2$-forms.  Due to \cite[Corollary 4.6]{ref_McLean}, $\Sm$ is infinitesimally rigid as a coassociative submanifold.

\subsection{The Codazzi equation on a minimal submanifold}

Suppose that $\Sm$ is a minimal submanifold.  Choose a local orthonormal frame $\{\me_1,\cdots,\me_{n+m}\}$ such that the restriction of $\{\me_1,\cdots,\me_n\}$ on $\Sm$ are tangent to $\Sm$ and the restriction of $\{\me_{n+1},\cdots,\me_{n+m}\}$ to $\Sm$ are normal to $\Sm$.  Consider the following equation on $\Sm$:
\begin{align*}
\hmd{\af}{i}{i}{k} = \me_k(\hm{\af}{i}{i}) - 2\ip{\nsub_{\me_j}\me_i}{\me_k}\hm{\af}{k}{i} - \ip{\nnor_{\me_j} \me_\af}{\me_\bt}\hm{\bt}{i}{i} ~.
\end{align*}
Since the mean curvature vanishes, the first and third terms are zero.  For the second term, $\ip{\nsub_{\me_j}\me_i}{\me_k}$ is skew-symmetric in $i$ and $k$, and $\hm{\af}{k}{i}$ is symmetric in $i$ and $k$.  Hence, the second term is also zero.  By combining it with the Codazzi equation \eqref{Codazzi1},
\begin{align}
\RM{\af}{i}{i}{j} &= \hmd{\af}{j}{i}{i} ~.
\label{Codazzi3} \end{align}

\section{The convexity of $\psi$ and a local uniqueness theorem of minimal submanifolds}

Suppose that $\Sm$ is a minimal submanifold in $(M,g)$ and consider the function $\psi = \sum_{\af}(y^\af)^2$ on the tubular neighborhood $U_\vep$ of $\Sigma$ as in \S\ref{sec_distance}.  Similar to \cite{ref_TsaiW}, the strong stability of $\Sm$ is closely related to the positivity of the trace of $\Hess(\psi)$ over an $n$-dimensional subspace.

\begin{prop}\label{prop_cvx}
Let $\Sm^n\subset(M,g)$ be a compact minimal submanifold that is strongly stable in the sense of \eqref{sstable}.  There exist positive constants $\vep_1$ and $c$ which depend on the geometry of $M$ and $\Sm$ and which have the following property.  For any $q\in U_{\vep_1}$ and any $n$-plane $L\subset T_q M$,
\begin{align}
\tr_L\Hess(\psi) &\geq c\left( \big(\fs(L)\big)^2 + \psi(q) \right)
\end{align}
where $\fs(L)$ is defined by \eqref{linear4}.
\end{prop}

\begin{proof}
Let $p\in\Sm$ be the point such that there is a normal geodesic in $U_\vep$ connecting $p$ and $q$.  To calculate $\Hess(\psi)$, take the frame $\{\me_1,\cdots,\me_n,$ $\me_{n+1},\cdots,\me_{n+m}\}$ constructed in \S\ref{sec_coordinate}. Let $\{\om^1,\cdots,\om^n,$ $\om^{n+1},\cdots,\om^{n+m}\}$ be the dual coframe.  According to \eqref{grad1} and \eqref{grad2}, $\dd\psi = 2y^\af\om^\af$, and thus $e_j(\psi)\equiv0$.  By \eqref{conn02},
\begin{align*}
&\quad \Hess(\psi)(\me_i,\me_j) \\
&= \me_i(\me_j(\psi)) - ({\namb_{\me_i}\me_j})(\psi) = -2y^\af\,\ta^\af_j(\me_i) \\
&= -2y^\af\, \hm{\af}{i}{j}\big|_p - 2y^\af x^k\,\hmd{\af}{i}{j}{k}\big|_p - 2y^\af y^\bt\, \RM{\af}{j}{\bt}{i}\big|_p - 2y^\af y^\bt\,(\hm{\af}{j}{k}\hm{\bt}{i}{k})\big|_p + \CO((|\bx|^2 + |\by|^2)^\frac{3}{2}) ~.
\end{align*}
By \eqref{vf_expansion1}, \eqref{conn02} and \eqref{conn03},
\begin{align}
\Hess(\psi)(\me_\af,\me_i)&= \me_\af(\me_i(\psi)) - ({\namb_{\me_\af}\me_i})(\psi) = -2y^\bt\,\ta^\bt_i(\me_\af) \notag \\
&= \CO(|\bx|^2 + |\by|^2) ~, \notag \\
\begin{split} \Hess(\psi)(\me_\af, \me_\bt) &= \me_\af(\me_\bt(\psi)) - ({\namb_{\me_\af}\me_\bt})(\psi) = 2\me_\af(y^\bt) - 2y^\gm\,\ta^\gm_\bt(\me_\af) \\
&= 2\dt_{\af\bt} + \CO(|\bx|^2 + |\by|^2) ~. \label{Hess_vv} \end{split}
\end{align}

We take the frame \eqref{linear1} for $L$ to evaluate $\tr_L\Hess(\psi)$; see Remark \ref{rmk_o}.  Note that all the $x^j$-coordinate of $q$ are zero.  By using $\sin\phi_j\leq\fs(L)$, $\cos\phi_j\leq1$ and the above expansions of $\Hess(\psi)$,
\begin{align*}
\tr_L\Hess(\psi)
&= \sum_j\Hess(\psi)(\cos\phi_j\,\me_j + \sin\phi_j\,\me_{n+j}, \cos\phi_j\,\me_j + \sin\phi_j\,\me_{n+j}) \\
&\geq \sum_j\left[ 2\cos^2\phi_j \left(-y^\af\, \hm{\af}{j}{j}\big|_p -y^\af y^\bt\, \RM{\af}{j}{\bt}{j}\big|_p - y^\af y^\bt\,(\hm{\af}{j}{k}\hm{\bt}{j}{k})\big|_p \right) + 2\sin^2\phi_j \right] \\
&\qquad -c'\,|\by|^3 - c''\,\fs(L)\cdot |\by|^2 \\
&\geq 2\sum_{j,\af,\bt}\left[ -y^\af y^\bt\, \RM{\af}{j}{\bt}{j}\big|_p - y^\af y^\bt\,(\hm{\af}{j}{k}\hm{\bt}{j}{k})\big|_p \right] + 2\sum_j\sin^2\phi_j \\
&\qquad + 2\sum_{j,\af,\bt}\sin^2\phi_j \left[{y^\af\, \hm{\af}{j}{j}\big|_p} + y^\af y^\bt\, \RM{\af}{j}{\bt}{j}\big|_p + y^\af y^\bt\,(\hm{\af}{j}{k}\hm{\bt}{j}{k})\big|_p  \right] - \fs^2(L) - c''' |\by|^3
\end{align*}
for some positive constants $c',c''$ and $c'''$.  For the last inequality, the minimal condition $\sum_j\hm{\af}{j}{j}|_p = 0$ has been used.  It is not hard to see that there exists $\vep'>0$ such that
$$\left| \sum_{j,\af,\bt}\sin^2\phi_j \left[{y^\af\, \hm{\af}{j}{j}\big|_p} + y^\af y^\bt\, \RM{\af}{j}{\bt}{j}\big|_p + y^\af y^\bt\,(\hm{\af}{j}{k}\hm{\bt}{j}{k})\big|_p  \right] \right| \leq \frac{1}{4}\sum_j\sin^2\phi_j$$
for $|\by|<\vep'$.  By using the strong stability condition \eqref{sstable}, this finishes the proof of the proposition.
\end{proof}

By the same argument as in \cite{ref_TsaiW}, the convexity of $\psi$ implies the following local uniqueness theorem of minimal submanifolds near $\Sm$.

\begin{thm}\label{uniqueness} (Theorem A) Let $\Sm^n\subset (M,g)$ be a compact minimal submanifold which is strongly stable in the sense of \eqref{sstable}.  Then, there exists a tubular neighborhood $U$ of $\Sm$ such that any compact minimal submanifold $\Gm$ in $U$ with $\dim\Gm\geq n$ must be contained in $\Sm$.
In other words, $\Sm$ is the only compact minimal submanifold in $U$ with dimension no less than $n$.
\end{thm}

\begin{proof}
It basically follows from \cite[Lemma 5.1]{ref_TsaiW} and Proposition \ref{prop_cvx}.  The only point to check is that the estimate of Proposition \ref{prop_cvx} holds for dimension greater than $n$.  Namely, it remains to show that for any $q\in U_{\vep_1}$ and any $\bar{n}$-plane $\bar{L}\subset T_q M$ with $\bar{n}>n$,
\begin{align*}
\tr_{\bar{L}}\Hess(\psi) \geq c_0
\end{align*}
for some positive constant $c_0$.

The argument is similar to Remark \ref{rmk_linear}.  Pick an $(\bar{n}-n)$-subspace of $\ker(\pi_\CH:\bar{L}\to\CH_q)$.  Denote it by $L_\CV$.  Note that $L_\CV$ belong to $\CV_q$.  Let $L$ be the orthogonal complement of $L_\CV$ in $\bar{L}$.  The dimension of $L$ is $n$.  By Proposition \ref{prop_cvx} and \eqref{Hess_vv}, the trace of the Hessian of $\psi$ over $\bar{L}$ has the following lower bound:
\begin{align*}
\tr_{\bar{L}}\Hess(\psi) &= \tr_{L}\Hess(\psi) + \tr_{L_\CV}\Hess(\psi) \\
&\geq c\,\psi(q) + \left(2(\bar{n} - n) - c'\psi(q)\right) ~.
\end{align*}
Thus, the quantity is positive when $\psi(q)$ is sufficiently small.
\end{proof}

\section{Further estimates needed for the stability theorem}

From now on, $\Sigma$ is taken to be an \emph{oriented}, \emph{strongly stable} minimal submanifold, and we see in the last section that the distance function $\psi$ to $\Sigma$ defined on $U_\vep$ satisfies a convexity condition.

To study the dynamical stability of mean curvature flows near $\Sigma$, we need to measure how close a nearby submanifold is to $\Sigma$. The distance function $\psi$ gives such a measurement in $C^0$. In order to obtain measurements in higher derivatives, we extend the volume form and the second fundamental form of $\Sigma$ to the tubular neighborhood $U_\vep$. In particular, in \S\ref{sec_linear}, the volume form of $\Sm$ is extended to an $n$-form $\Om$ on $U_\vep$. The restriction of $\Omega$ to another $n$-dimensional submanifold $\Gamma$, which is denoted by $*\Omega$, measures how close $\Gamma$ is to $\Sigma$ in $C^1$.  The evolution equation of $*\Omega$ along the mean curvature flow plays an essential role for the estimates.  The equation naturally involves the restriction of the covariant derivatives/second covariant derivatives of $\Omega$ on $\Gamma$.  In this section, we derive estimates of these quantities in preparation for the proof of the stability theorem.

\subsection{Extension of auxiliary tensors to $U_\vep$}
We adopt the frame and coordinate constructed in \S\ref{sec_local}.

The second fundamental form of $\Sm$ can also be extended to $U_\vep$ by parallel transport along normal geodesics, as explained in \S\ref{sec_linear}.  Denote the extension by $\II^\Sm$, which, in terms of the frames, is given by
\begin{align}
\II^\Sm = \hm{\af}{i}{j}\,\om^i\ot\om^j\ot\me_\af ~.
\label{ext_II} \end{align}
In other words, for any $q\in U_\vep$,  $\hm{\af}{i}{j}(q)=\hm{\af}{i}{j}(p)$ where $p\in\Sm$ is the unique point such that there is a normal geodesic in $U_\vep$ connecting $p$ and $q$, see Definition \ref{u_epsilon}.  To avoid introducing more notations, we use the metric $g$ to lower the indices of $\II^\Sm$, and then $\II^\Sm = \hm{\af}{i}{j}\,\me_i\ot\me_j\ot\me_\af$.

Suppose that $\Gm$ is an oriented, $n$-dimensional submanifold in $U_\vep\subset M$ with $\Om(\Gm)>0$.  With the above extension, we can compare the second fundamental form of $\Gm$ with that of $\Sm$.  For any $q\in\Gm$, choose a local orthonormal frame $\{\te_1,\cdots,\te_n$, $\te_{n+1},\cdots,\te_{n+m}\}$ on a neighborhood of $q$ in $M$ such that the restriction of $\{\te_1,\cdots, \te_n\}$ on $\Gm$ form an oriented frame for $T\Gm$, and the restriction of $\{\te_{n+1},\cdots,\te_{n+m}\}$ on $\Gm$ form a frame for $\N\Gm$.  With this, the second fundamental form of $\Gm$ is
\begin{align}
\II^\Gm = \htd{\af}{i}{j}\,\te_i\ot\te_j\ot\te_\af \quad\text{where}\quad \htd{\af}{i}{j} = \ip{\namb_{\te_i}\te_j}{\te_\af} ~.
\label{second_Gm} \end{align}
As explained in \S\ref{sec_linear}, we may assume that these frames are of the form \eqref{linear1} and \eqref{linear2} at $q$.  The inverse transform reads
\begin{align}
\me_j = \cos\phi_j\,\te_j - \sin\phi_j\,\te_{n+j}  \quad\text{and}\quad  \me_\af = \sin\phi_\af\,\te_{\af-n} + \cos\phi_\af\,\te_\af ~.
\label{inverse1} \end{align}
It follows that
\begin{align}
\II^\Sm\big|_q &= \hm{\af}{i}{j} (p)\,(\cos\phi_i\,\te_i - \sin\phi_i\,\te_{n+i})\ot(\cos\phi_j\,\te_j - \sin\phi_j\,\te_{n+j})\ot(\sin\phi_\af\,\te_{\af-n} + \cos\phi_\af\,\te_\af).
\label{second_Sm} \end{align} 
Hence,
\begin{align}
\ip{\II^\Gm}{\II^\Sm}\big|_q &= \sum_{\af,i,j}\left( \cos\phi_i\cos\phi_j\cos\phi_\af\,\htd{\af}{i}{j}\,\hm{\af}{i}{j} (p) \right) ~.
\label{pairing1}\end{align}
In the above expression, $\hm{\af}{i}{j} (p)$ depends only on $p\in \Sigma$, while $\phi_i, \phi_\alpha$, and $\htd{\af}{i}{j}$ all depend
on $q\in \Gamma$.

We extend  another tensor which is related to the strong stability condition \eqref{sstable}.  Consider the parallel transport of the following tensor on $\Sm$ along normal geodesics:
\begin{align*}
\left(\RM{\af}{i}{\bt}{j} + \hm{\af}{i}{k}\hm{\bt}{j}{k}\right)\,(\om^i\ot\om^j)\ot(\me^\af\ot\me^\bt) ~,
\end{align*} which is considered to be defined on $U_\vep$. 
Pairing the last component with $\namb\psi/2$ produces a tensor of the same type as $\II^\Sm$, which is denoted by $S^\Sm$:
\begin{align}
S^\Sm|_q &= y^\bt\left(\RM{\af}{i}{\bt}{j}(p) + (\hm{\af}{i}{k}\hm{\bt}{j}{k})(p)\right)\,\om^i\ot\om^j\ot\me_\af
\label{ext_S} \end{align}
where $p\in\Sm$ is the point such that there is a unique normal geodesic in $U_\vep$ connecting $p$ and $q$.   Similarly,
\begin{align}
\ip{\II^\Gm}{S^\Sm}\big|_{q} &= \sum_{\af,\bt,i,j}\left( \cos\phi_i\cos\phi_j\cos\phi_\af\,\htd{\af}{i}{j}\,y^\bt\big(\RM{\af}{i}{\bt}{j}(p) + \sum_k(\hm{\af}{i}{k}\hm{\bt}{j}{k})(p)\big) \right) ~.
\label{pairing2} \end{align}

Again in the above expression, $ \RM{\af}{i}{\bt}{j}(p) + \sum_k(\hm{\af}{i}{k}\hm{\bt}{j}{k})(p)$ depends only on $p\in \Sigma$, while $\phi_i, \phi_\alpha$,$y^\bt$, and $\htd{\af}{i}{j}$ all depend on $q\in \Gamma$.

In the rest of this subsection, we assume $\Om(T_q\Gm) > \oh$ and estimate $\ip{\II^\Gm}{\II^\Sm}\big|_q$ and $\ip{\II^\Gm}{S^\Sm}\big|_{q}$.
We assume that $T_q\Gm$ has an oriented frame of the form \eqref{linear1}, and $\N_q\Gm$ has a frame of the form \eqref{linear2}.  Since $\Om(T_q\Gm) > \oh$, it follows from \eqref{linear3} that
\begin{align*}
\cos\phi_j \geq \cos\phi_1\cdots\cos\phi_n >\oh \quad\text{for }j\in\{1,\ldots,n\} ~.
\end{align*}
{In particular, $\sin^2\phi_j=1-\cos^2\phi_j<3/4$ for each $j$.  It can be checked directly that a real number $x$ with $x^2<\frac{3}{4}$ satisfies the inequalities $1-\sqrt{1-x^2}< \frac{2}{3} x^2$ and $\frac{x^2}{\sqrt{1-x^2}}<2x^2$. Therefore, 
\begin{align} \begin{split}
&0 < 1 - \cos\phi_j = 1 - \sqrt{1-\sin^2\phi_j} < \frac{2}{3} \left(\fs(T_q\Gm)\right)^2 ~, \\
&\left| \frac{1}{\cos\phi_j} - \cos\phi_j \right| = \frac{\sin^2\phi_j}{\cos\phi_j} < 2\left(\fs(T_q\Gm)\right)^2.
\end{split} \label{est6} \end{align}}
Suppose that $\fs$ in \eqref{linear4} is achieved at $\phi_1$, and then
\begin{align}
& \Om(T_q\Gm) \leq \cos\phi_1 = \sqrt{1 - \sin^2\phi_1} \leq 1 - \oh\sin^2\phi_1  \notag \\
\Rightarrow~ & 1-\Om(T_q\Gm) \geq \oh \left(\fs(T_q\Gm)\right)^2 ~.
\label{boundfs}\end{align}
On the other hand,
\begin{align}
1-\Om(T_q\Gm) &\leq 1 - (\Om(T_q\Gm))^2 = 1 - \prod_{j=1}^n(1-\sin^2\phi_j) \leq c(n)\,(\fs(T_q\Gm))^2
\label{boundfs1} \end{align}
for some dimensional constant $c(n)$.

Applying the estimate \eqref{est6} to \eqref{pairing1} and \eqref{pairing2}, we obtain
\begin{align} \begin{split}
\left| \ip{\II^\Gm}{\II^\Sm}\big|_q - \sum_{\af,i,j} \htd{\af}{i}{j}\,\hm{\af}{i}{j}(p) \right| &\leq c\left(\fs(q)\right)^2 |\II^\Gm|  ~, \\
\left| \ip{\II^\Gm}{S^\Sm}\big|_{q} - \sum_{\af,\bt,i,j}\left( \htd{\af}{i}{j}\,y^\bt\big(\RM{\af}{i}{\bt}{j}(p) + \sum_k(\hm{\af}{i}{k}\hm{\bt}{j}{k})(p)\big) \right) \right| &\leq c\left(\fs(q)\right)^2 \sqrt{\psi(q)}\, |\II^\Gm|
\end{split} \label{est12} \end{align}
for some constant $c$ depending on the geometry of $\Sm$ and $M$.

\subsection{Estimates involving the derivatives of $\Om$}
In this subsection, we derive estimates that involve derivatives of $\Om$, which are needed in the proof of Theorem B.  In the following three lemmas, we estimate quantities that appear naturally in the evolution equation of $*\Omega$ \eqref{*Omega}.

Let $\Gm$ be an $n$-dimensional submanifold in the tubular neighborhood of $\Sm$.
The function $*\Om$ is the Hodge star of $\Om|_\Gm$ with respect to the induced metric on $\Gm$, and is the same as $\Om(T_q\Gm)$. We assume throughout this subsection that  $*\Om(q)>\oh$ for any $q\in\Gm$.  For each $q\in\Gm$, let $p\in\Sm$ be the point such that there is a unique normal geodesic in $U_\vep$ connecting $p$ and $q$; see Definition \ref{u_epsilon}.  We use the coordinate and frame constructed in \S\ref{sec_notation} to carry out the computation.  Moreover, we assume that $T_q\Gm$ has an oriented frame of the form \eqref{linear1}, and $\N_q\Gm$ has a frame of the form \eqref{linear2}.
For $q\in \Gm$, $\psi (q)$ is a $\CC^0$ order quantity.  $*\Omega (q)$ and $\fs(q)$ are both $\CC^1$ order quantities that depend on the tangent space $T_q\Gm$ at $q$, where $\fs(q)=\fs(T_q\Gm)$ is defined in \eqref{linear4}.

\subsubsection{The restriction of the derivative of $\Om$ to $\Gamma$}

To compute $\namb\Om$, it is convenient to introduce the following shorthand notations:
\begin{align}
\Om^j &= \iota(\me_j)\Om = (-1)^{j+1}\om^1\w\cdots\w\widehat{\om^j}\w\cdots\w\om^n ~, \label{Om1} \\
\Om^{jk} &= \iota(\me_k)\iota(\me_j)\Om = \begin{cases}
(-1)^{j+k}\,\om^1\w\cdots\w\widehat{\om^k}\w\cdots\w\widehat{\om^j}\w\cdots\w\om^n   &\text{if } k<j ~, \\
(-1)^{j+k+1}\,\om^1\w\cdots\w\widehat{\om^j}\w\cdots\w\widehat{\om^k}\w\cdots\w\om^n   &\text{if } k>j ~.
\end{cases} \label{Om2}
\end{align}
The covariant derivative of $\Om$ is
\begin{align}
\namb\Om &= (\nabla\om^j)\w\Om^j = \ta^{\af}_j\ot(\om^{\af}\w\Om^j) ~.
\label{cov1} \end{align}

\begin{lem} \label{lem_cov1}
Let $\Sm^n\subset(M,g)$ be a compact, oriented minimal submanifold.  Then, there exist a positive constant $c$ which depends on the geometry of $M$ and $\Sm$ and which has the following property.  Suppose that $\Gm\subset U_{\vep}$ is an oriented $n$-dimensional submanifold with $*\Om(q)>\oh$ for any $q\in\Gm$. Then,
\begin{align*}
&\quad \left| \sum_{\af,j,k}\left[(-1)^j\,\htd{\af}{j}{k}\,(\namb_{\te_k}\Om)(\te_\af,\te_1,\cdots,\widehat{\,\te_j\,},\cdots,\te_n)\right] + (*\Om)\,\ip{\II^\Gm}{\II^\Sm + S^\Sm} \right|\\
&\leq c\left( \big(\fs(q)\big)^2 + \psi(q) \right)\! \big|\II^\Gm\!\big|
\end{align*}
at any $q\in\Gm$.  The summation is indeed a contraction between $\II^\Gm$ and $\namb\Om$, and is independent of the choice of the orthonormal frame.
\end{lem}

\begin{proof}
By \eqref{cov1}, \eqref{conn02} and \eqref{linear9},
\begin{align*}
& \left| \sum_{\af,j,k}\left[(-1)^j\,\htd{\af}{j}{k}\,(\namb_{\te_k}\Om)(\te_\af,\te_1,\cdots,\widehat{\,\te_j\,},\cdots,\te_n)\right]\Big|_q + \sum_{\af,j,k}\left[\htd{\af}{j}{k}\ta^\af_j(\te_k)\frac{\cos\phi_\af}{\cos\phi_j}(*\Om)\right] \right| \\
\leq\,& c_0\big(\fs(q)\big)^2\big|\II^\Gm\!\big| ~.
\end{align*}
According to \eqref{conn02}, \eqref{linear1} and the fact that the $x^j$-coordinates of $q$ are all zero,
\begin{align}
\left| \ta^\af_j(\te_k)\big|_q - \cos\phi_k\,\hm{\af}{j}{k}(p) - \cos\phi_k\, y^\bt\,(\RM{\af}{j}{\bt}{k} + \hm{\af}{j}{l}\hm{\bt}{k}{l})(p) \right| &\leq c_1\left( \big(\fs(q)\big)^2 + \psi(q) \right)
\label{est1} \end{align}
at $q$.  Due to \eqref{est6},
\begin{align*}
\left| \frac{\cos\phi_k\cos\phi_\af}{\cos\phi_j} - 1 \right| &\leq 100\,\big(\fs(q)\big)^2 ~.
\end{align*}
Putting these together with \eqref{est12} finishes the proof of this lemma.
\end{proof}

\subsubsection{The restriction of the second derivative of $\Om$ to $\Gamma$}

We now compute $\nabla^2\Om$, which is a section of $(T^*M\ot T^*M)\ot\Ld^n(T^*M)$.
Since
\begin{align*}
\namb \om^{\af} &= -\ta^{\af}_i\ot\om^i - \ta^{\af}_{\bt}\ot\om^{\bt} \quad\text{and} \\
\namb \Om^j &= (\namb\om^k)\w\Om^{jk} = \ta_j^k\ot\Om^k + \ta_k^{\af}\ot(\om^{\af}\w\Om^{jk}) ~,
\end{align*}
the covariant derivative of \eqref{cov1} is
\begin{align} \begin{split}
\nabla^2\Om &= -(\ta^{\af}_i\ot\ta^{\af}_i)\ot\Om + (\ta^{\af}_k\ot\ta^{\bt}_j)\ot(\om^{\bt}\w\om^{\af}\w\Om^{jk}) \\
&\quad + \left( \namb\ta^{\af}_i + \ta^{\af}_{\bt}\ot\ta^{\bt}_i + \ta^i_k\ot\ta^{\af}_k \right)\ot(\om^{\af}\w\Om^i)
\end{split} \label{cov2} \end{align}
where $\namb\ta_i^\af$ is the covariant derivative of a local section of $T^*M$.

\begin{lem} \label{lem_cov2}
Let $\Sm^n\subset(M,g)$ be a compact, oriented minimal submanifold.  Then there exists a positive constant $c$ which depends on the geometry of $M$ and $\Sm$ and which has the following property.  Suppose that $\Gm\subset U_{\vep}$ is an oriented $n$-dimensional submanifold with $*\Om(q)>\oh$ for any $q\in\Gm$.  Then,
\begin{align*}
&\quad \left| \sum_k\left[(\namb^2_{\te_k,\te_k}\Om)(\te_1,\cdots,\te_n)\right] - \sum_{\af,i,k}\left[(-1)^i\Om(\te_\af,\te_1,\cdots,\widehat{\,\te_i\,},\cdots,\te_n)\RTD{\af}{k}{k}{i}\right] + (*\Om)\,\left|\II^\Sm\! + S^\Sm\right|^2\right| \\
&\leq c\left( \fs^2(q)  + \psi(q) \right)
\end{align*}
at any $q\in\Gm$, where $\RTD{\af}{k}{k}{j} = \ip{R(\te_k,\te_j)\te_k}{\te_\af}$ are components of the restriction of the curvature tensor of $M$ along $\Gamma$.   Note that the two summations are independent of the choice of the orthonormal frame.
\end{lem}

\begin{proof}

We examine the components on the right hand side of \eqref{cov2}.  Due to \eqref{conn02} and \eqref{conn03},
\begin{align} \begin{split}
|\ta_i^j(\te_k)| &\leq c_2\left( \fs^2(q) + \sqrt{\psi(q)} \right) ~, \\
|\ta_\af^\bt(\te_k)| &\leq c_2\left( \fs^2(q) + \sqrt{\psi(q)} \right)
\end{split} \label{est4} \end{align}
for any $i,j,k\in\{1,\ldots,n\}$ and $\af,\bt\in\{n+1,\ldots,n+m\}$.  With \eqref{est1} and the third and fifth line of \eqref{linear6},
\begin{align} \begin{split}
\left| \big(\ta_j^\af (\te_k)\big)\, \big(\ta_i^\bt (\te_k)\big)\, \Big((\om^\bt\w\om^\af\w\Om^{ij})(\te_1,\cdots,\te_n)\Big) \right| &\leq c_3\big(\fs(q)\big)^2 ~, \\
\left| \big(\ta_\af^\bt(\te_k)\big)\, \big(\ta_j^\bt(\te_k)\big)\, \Big((\om^\af\w\Om^{j})(\te_1,\cdots,\te_n)\Big) \right| &\leq c_3\left( \big(\fs(q)\big)^2 + \psi(q) \right)~, \\
\left| \big(\ta_i^j (\te_k)\big)\, \big(\ta_i^\af(\te_k)\big)\, \Big((\om^\af\w\Om^{j})(\te_1,\cdots,\te_n)\Big) \right| &\leq c_3\left( \big(\fs(q)\big)^2 + \psi(q) \right) ~.
\end{split} \label{est2} \end{align}
According \eqref{est6} and the triangle inequality, we may replace $\cos\phi_k$ by $1$ in \eqref{est1}, and the error term is of the same order.  It follows that
\begin{align}
\left| \sum_{\af,j,k}\big(\ta^\af_j(\te_k)\big)^2 - |\II^\Sm+S^\Sm|^2 \right| &\leq c_4\left( \left(\fs(q)\right)^2 + \psi(q) \right) ~.
\label{est3} \end{align}
To estimate $(\namb^2_{\te_k,\te_k}\Om)(\te_1,\cdots,\te_n)$, apply the triangle inequality on the right hand side of $\eqref{cov2}$.  Due to \eqref{est2}, all terms, except the contribution from $-(\ta_i^\af\ot\ta_i^\af)\ot\Om$ and $(\namb\ta^\af_i)\ot(\om^\af\w\Om^i)$, can be bounded by some multiple of $\big(\fs(q)\big)^2 + \psi(q)$.  The term $\sum_{i,k}(\ta_i^\af(\te_k))^2\,\Om(\te_1,\cdots\,\te_n)$ is estimated by \eqref{est3}.  Hence,
\begin{align} \begin{split}
&\left| \left[(\namb^2_{\te_k,\te_k}\Om)(\te_1,\cdots,\te_n)\right] - \left((\namb\ta_i^\af)(\te_k,\te_k)\right)\left((\om^\af\w\Om^i)(\te_1,\cdots,\te_n)\right) + (*\Om)\,\left|\II^\Sm\! + S^\Sm\right|^2\right| \\
\leq\;& c_5 \left( \left(\fs(q)\right)^2 + \psi(q) \right) ~.
\end{split} \label{est10} \end{align}

The next step is to compute $\namb\ta^\af_i$:
\begin{align}
\ta_i^\af &= \ta_i^\af(e_j)\,\om^j + \ta_i^\af(e_\bt)\,\om^\bt  \notag \\
\Rightarrow~ \namb\ta_i^\af &= \dd(\ta_i^\af(e_j))\ot\om^j + \dd(\ta_i^\af(e_\bt))\ot\om^\bt + \ta_i^\af(e_j)\,\namb\om^j + \ta_i^\af(e_\bt)\,\namb\om^\bt ~.  \label{est7}
\end{align}
By \eqref{est4} and \eqref{est1}, we have the following estimate at $q$:
\begin{align*}
\left|(\namb\om^j)(\te_k,\te_k)\right| &= \left| \ta^j_i(\te_k)\,\om^i(\te_k) - \ta^\af_j(\te_k)\,\om^\af(\te_k) \right| \leq c_6 \left( \fs(q) + \sqrt{\psi(q)} \right) ~, \\
\left|(\namb\om^\bt)(\te_k,\te_k)\right| &= \left| \ta^\bt_j(\te_k)\,\om^j(\te_k) + \ta^\bt_\gm(\te_k)\,\om^\gm(\te_k) \right| \leq c_6 ~.
\end{align*}
Together with \eqref{conn02},
\begin{align}
\left|\left(\ta_i^\af(e_j)\,\namb\om^j + \ta_i^\af(e_\bt)\,\namb\om^\bt\right)(\te_k,\te_k)\right| &\leq c_7\left( \fs(q) + \sqrt{\psi(q)} \right) ~.
\label{est8} \end{align}
It follows from \eqref{conn02} and \eqref{vf_expansion} that
\begin{align} \begin{split}
\left| \dd(\ta_i^\af(e_j)) - \left(\hmd{\af}{i}{j}{k}\right)(p)\om^k + \left(\RM{\af}{i}{\bt}{j} + \hm{\af}{i}{k}\hm{\bt}{j}{k}\right)(p)\om^\bt \right| &\leq c_8 \sqrt{\psi(q)} ~, \\
\left| \dd(\ta_i^\af(e_\bt)) - \oh \RM{\af}{i}{\gm}{\bt}(p) \om^\gm \right| &\leq c_8
\end{split} \label{est9} \end{align}
where the norm on the left hand side is induced by the Riemannian metric $g$.  By combining \eqref{est7}, \eqref{est8} and \eqref{est9},
\begin{align*}
\left| (\namb\ta^\af_i)(\te_k,\te_k) - \cos^2\phi_k\, \hmd{\af}{i}{k}{k}(p) \right| &\leq c_9 \left( \fs(q) + \sqrt{\psi(q)} \right) ~.
\end{align*}
It together with \eqref{linear8} and \eqref{est6} gives that
\begin{align} \begin{split}
&\left| -\big((\namb\ta^\af_i)(\te_k,\te_k)\,(\om^\af\w\Om^i)(\te_1,\cdots,\te_n)\big) + (*\Om)\,{\sin\phi_i}\,\hmd{(n+i)}{i}{k}{k}(p) \right| \\
\leq\,& c_{10}\left( \left(\fs(q)\right)^2 + \psi(q) \right) ~.
\end{split} \label{est5} \end{align}

It remains to calculate the second term in the asserted inequality of the lemma.  By \eqref{linear7},
\begin{align*}
\sum_{\af,i,k}(-1)^i\Om(\te_\af,\te_1,\cdots,\widehat{\,\te_i\,},\cdots,\te_n)\RTD{\af}{k}{k}{i}
&= (*\Om)\sum_{i,k}\frac{\sin\phi_i}{\cos\phi_i}R(\te_{n+i},\te_k,\te_k,\te_i) ~.
\end{align*}
With \eqref{linear1}, \eqref{linear2} and \eqref{est6},
\begin{align*}
\left| \sum_{\af,i,k}(-1)^i\Om(\te_\af,\te_1,\cdots,\widehat{\,\te_i\,},\cdots,\te_n)\RTD{\af}{k}{k}{i} + (*\Om)\sum_{i,k}{\sin\phi_i}\RM{(n+i)}{k}{k}{i} \right| \leq c_{11} \left(\fs(q)\right)^2 ~.
\end{align*}
Since $\big|\RM{(n+i)}{k}{k}{i}|_q - \RM{(n+i)}{k}{k}{i}|_p\big|\leq c_{12}\sqrt{\psi(q)}$, we have
\begin{align}
\left| \sum_{\af,i,k}(-1)^i\Om(\te_\af,\te_1,\cdots,\widehat{\,\te_i\,},\cdots,\te_n)\RTD{\af}{k}{k}{i} + (*\Om)\sum_{i,k}{\sin\phi_i}\RM{(n+i)}{k}{k}{i}(p) \right| \leq c_{13}\left( \left(\fs(q)\right)^2 + \psi(q) \right) ~.
\label{est11} \end{align}

To conclude the lemma, apply the triangle equality on \eqref{est10}, \eqref{est5} and \eqref{est11}, and note that $\RM{(n+i)}{k}{k}{i}(p) = \hmd{(n+i)}{i}{k}{k}(p)$ by \eqref{Codazzi3}.
\end{proof}

\begin{rmk}
The tensor $S^\Sm$ is needed for Lemma \ref{lem_cov2}; otherwise the error term would be bigger.  However, $S^\Sm$ will only be used in some intermediate steps in the proof of Theorem B.
\end{rmk}

\subsubsection{The derivative of $*\Om$ along $\Gamma$}

The following lemma relates the derivative of $*\Om$ along $\Gamma$ and the second fundamental form of $\Gamma$.

\begin{lem} \label{lem_cov*}
Let $\Sm^n\subset(M,g)$ be a compact, oriented minimal submanifold.  Then, there exist a positive constant $c$ which depends on the geometry of $M$ and $\Sm$ and which has the following property.  Suppose that $\Gm\subset U_{\vep}$ is an oriented $n$-dimensional submanifold with $*\Om(q)>\oh$ for any $q\in\Gm$.  Then,
\begin{align*}
|\nabla^{\Gm}(*\Om)|^2 &\leq c \left(\fs(q) (*\Om) \right)^2 |\II^\Gm - \II^\Sm|^2 + c\left( \left(\fs(q)\right)^2 +\psi(q) \right)^2
\end{align*}
for any $q\in\Gm$.
\end{lem}

\begin{proof}
We compute
\begin{align*}
\nabla^{\Gm}(*\Om) &= \left[\te_j(\Om(\te_1,\cdots,\te_n)) \right]\tom^j \\
&= \left[ (\namb_{\te_j}\Om)(\te_1,\cdots,\te_n) + \sum_{i=1}^n\Om(\te_1,\cdots,\te_{i-1},\namb_{\te_j}\te_i,\te_{i+1},\cdots,\te_n) \right]\tom^j \\
&= \left[ (\namb_{\te_j}\Om)(\te_1,\cdots,\te_n) + \sum_{i=1}^n\htd{\af}{i}{j}\Om(\te_1,\cdots,\te_{i-1},\te_\af,\te_{i+1},\cdots,\te_n) \right]\tom^j ~.
\end{align*}
Note that the expression is tensorial, and we use the frame \eqref{linear1} and \eqref{linear2} to proceed.

Due to \eqref{cov1} and \eqref{linear8},
\begin{align*}
(\nabla_{\te_j}\Om)(\te_1,\cdots,\te_n) &= \ta_i^\af(\te_j)\,(\om^\af\w\Om^i)(\te_1,\cdots,\te_n) \\
&= \left( \cos\phi_j\,\ta_i^{n+i}(e_j) + \sin\phi_j\,\ta_i^{n+i}(e_{n+j}) \right)\frac{\sin\phi_i}{\cos\phi_i}(*\Om) ~.
\end{align*}
By \eqref{conn02} and \eqref{est6}, at $q$,
\begin{align*}
& \left| (\nabla_{\te_j}\Om)(\te_1,\cdots,\te_n) - (*\Om)\sum_{i=1}^n\left[\frac{\sin\phi_i}{\cos\phi_i}\cos\phi_{n+i}\cos\phi_i\cos\phi_j \hm{(n+i)}{i}{j}(p)\right] \right| \\
\leq\;& c_1 \left( \left(\fs(q)\right)^2 +\psi(q) \right) ~.
\end{align*}
According to \eqref{linear7},
\begin{align*}
\sum_{i=1}^n\htd{\af}{i}{j}\Om(\te_1,\cdots,\te_{i-1},\te_\af,\te_{i+1},\cdots,\te_n)  &= -(*\Om)\sum_{i=1}^n\left[\frac{\sin\phi_i}{\cos\phi_i}\htd{(n+i)}{i}{j}\right] ~.
\end{align*}

To sum up,
\begin{align*}
\left| \nabla^\Gm(*\Om) \right|^2 &= \sum_{j=1}^n \left[\te_j(\Om(\te_1,\cdots,\te_n)) \right]^2 \\
&\leq 2\sum_{j=1}^n \left| (\nabla_{\te_j}\Om)(\te_1,\cdots,\te_n) - (*\Om)\sum_{i=1}^n\left[\frac{\sin\phi_i}{\cos\phi_i}\cos\phi_{n+i}\cos\phi_i\cos\phi_j \hm{(n+i)}{i}{j}(p)\right] \right|^2 \\
&\quad + 2(*\Om)^2 \sum_{j=1}^n \left| \sum_{i=1}^n\frac{\sin\phi_i}{\cos\phi_i}\left( \cos\phi_{n+i}\cos\phi_i\cos\phi_j \hm{(n+i)}{i}{j}(p) -\htd{(n+i)}{i}{j} \right) \right|^2 \\
&\leq 4n c_1^2 \left( \left(\fs(q)\right)^2 +\psi(q) \right)^2 \\
&\quad + 8n \left(\fs(q)\right)^2 (*\Om)^2 \sum_{i,j}     \left| \htd{(n+i)}{i}{j} - \cos\phi_{n+i}\cos\phi_i\cos\phi_j \hm{(n+i)}{i}{j}(p)\right|^2
\end{align*}
By \eqref{second_Gm} and \eqref{second_Sm},
\begin{align*}
|\II^\Gm - \II^\Sm|^2 \geq \sum_{\af,i,j}\left| \htd{\af}{i}{j} - \cos\phi_i\cos\phi_j\cos\phi_\af\hm{\af}{i}{j}(p) \right|^2
\end{align*}
This completes the proof of this lemma.
\end{proof}

\section{Stability of the mean curvature flow}

After the preparation in the last sections, we consider the mean curvature flow. 
We first recall the following proposition from \cite{ref_W4}*{Proposition 3.1}.

\begin{prop} Along the mean curvature flow $\Gm_t$ in $M$, 
$*\Omega =\Omega(\te_1, \cdots, \te_n)$ satisfies
\begin{align}\label{*Omega}
\begin{split}
\frac{\dd}{\dd t} *\Omega &=\Dt *\Omega + *\Omega (\sum_{\alpha ,i,k}\htd{\alpha}{i}{k}^2)\\
&-2\sum_{\alpha, \beta, k}[\OMTD{\alpha \beta 3\cdots n}
\htd{\alpha}{1}{k}\htd{\beta}{2}{k} + \OMTD{\alpha 2 \beta\cdots n}
\htd{\alpha}{1}{k}\htd{\beta}{3}{k} +\cdots +\OMTD{1\cdots (n-2) \alpha\beta} \htd{\alpha}{(n-1)}{k}
\htd{\beta}{n}{k}]\\
&-2(\namb_{\te_k}\Omega) (\te_\alpha, \cdots, \te_{n})\htd{\alpha}{1}{k}
-\cdots-2(\nabla_{\te_k}\Omega) (\te_{1}, \cdots, \te_\alpha
)\htd{\alpha}{n}{k}\\
&-\sum_{\alpha, k}[\OMTD{\alpha 2 \cdots n}\RTD{\alpha}{k}{k}{1}
+\cdots +\OMTD{1\cdots (n-1)\alpha}\RTD{\alpha}{k}{k}{n}]-(\namb^2_{\te_k,\te_k}\Omega)(\te_{1},
\cdots,\te_{n})
\end{split}
\end{align}
where $\Dt$ denotes the time-dependent Laplacian on $\Gm_t$, $\OMTD{\alpha\beta 3\cdots n}=\Omega(\te_\alpha, \te_\beta, \te_3, \cdots, \te_n)$ etc., and $\RTD{\alpha}{k}{k}{1}=\ip{R(\te_k,\te_1)\te_k}{\te_\alpha}$, etc.\ are the coefficients of the curvature operator. 
\end{prop}
When $\Omega$ is a parallel form in $M$, $\nabla\Omega \equiv 0 $, this recovers an important formula in proving the long time existence result of the graphical mean curvature flow in \cite{ref_W3}.

\subsection{Proof of Theorem B}
A finite time singularity of the mean curvature flow happens exactly when the second fundamental becomes unbounded; see Huisken \cite{ref_Huisken}, also \cite{ref_W4}.  The following theorem shows that if we start with a submanifold which is $\CC^1$ close to a strongly stable minimal submanifold $\Sigma$, then the mean curvature flow exists for all time, and converges smoothly to $\Sigma$.

\begin{thm} \label{longtime} (Theorem B)
Let $\Sm^n\subset(M,g)$ be a compact, oriented, strongly stable minimal submanifold.  Then, there exist positive constants $\kp<\!\!<1$ and $c$ which depend on the geometry of $M$ and $\Sm$ and which have the following significance.  Suppose that $\Gm\subset U_\vep$ is an oriented $n$-dimensional submanifold satisfying
\begin{align}
\sup_{q\in\Gm} \left( 1 - (*\Om) + \psi \right) < \kp ~.
\label{smallC1} \end{align}
Then, the mean curvature flow $\Gm_t$ with $\Gm_0 = \Gm$ exists for all $t>0$.  Moreover, $\sup_{q\in\Gm_t}|\IIt|\leq c$ for any $t>0$, where $\IIt$ is the second fundamental form of $\Gm_t$, and $\Gm_t$ converges smoothly to $\Sigma$ as $t\rightarrow \infty$. 
\end{thm}

\begin{proof}
The constant $\kp$ will be chosen to be smaller than $\vep^2$ and $\oh$; its precise value will be determined later.  Suppose that the condition \eqref{smallC1} holds for all $\{\Gm_t\}_{0\leq t<T}$.

Denote by $H_t$ the mean curvature vector of $\Gm_t$.  According to Proposition \ref{prop_cvx}
\begin{align}
\frac{\dd}{\dd t}\psi &= H_t(\psi) = \Dt\psi - \tr_{\Gm_t}\Hess\psi \leq \Delta^{\Gm_t}\psi - c_1 (\psi + \fs^2) ~.
\label{evol1} \end{align}
By applying Lemma \ref{lem_cov1}, Lemma \ref{lem_cov2} and the second line of \eqref{linear6} to \eqref{*Omega},
\begin{align*}
\frac{\dd}{\dd t}(*\Om) &\geq \Dt(*\Om) + (*\Om)|\IIt|^2 - 2c_2 \fs^2 |\IIt|^2 \\
&\quad - 2(*\Om)\ip{\IIt}{\II^\Sm + S^\Sm} - c_2 (\fs^2 + \psi) |\IIt| \\
&\quad + (*\Om) |\II^\Sm + S^\Sm|^2 - c_2(\fs^2 + \psi) \\
&\geq \Dt(*\Om) + (*\Om)|\IIt - \II^\Sm - S^\Sm|^2 - c_3(\fs^2 + \psi) |\IIt|^2 - c_3(\fs^2 + \psi) \\
&\geq \Dt(*\Om) + \oh (*\Om)|\IIt - \II^\Sm|^2 - (*\Om)|S^\Sm|^2 \\
&\quad -2c_3(\fs^2 + \psi) |\IIt - \II^\Sm|^2 - c_3(\fs^2 + \psi) (1 + 2|\II^\Sm|^2) ~.
\end{align*}
The last inequality uses the fact that $|\IIt - \II^\Sm|^2 \leq 2|\IIt - \II^\Sm - S^\Sm|^2 + 2|S^\Sm|^2$ and $|\IIt|^2 \leq 2|\IIt - \II^\Sm|^2 + 2|\II^\Sm|^2$.
If $\kp\leq 1/(48c_3)$, it follows from \eqref{boundfs} that $2c_3(\fs^2 + \psi)\leq (*\Om)/6$.  Since $|S^\Sm|^2\leq c_4\psi$ and $|\II^\Sm|^2\leq c_4$,
\begin{align}
\frac{\dd}{\dd t}(*\Om) &\geq \Dt(*\Om) + \frac{1}{3} (*\Om)|\IIt - \II^\Sm|^2 - c_5(\fs^2 + \psi) ~.
\label{evol2} \end{align}
By combining it with \eqref{evol1}, \eqref{boundfs} and \eqref{boundfs1}, we have
\begin{align} \begin{split}
\frac{\dd}{\dd t}\left(1-(*\Om) + c_6\psi\right) &\leq \Dt \left(1-(*\Om) + c_6 \psi\right) - \frac{1}{3}(*\Om)|\IIt - \II^\Sm|^2 - c_7\left( 1-(*\Om) + \psi \right) \\
&\leq \Dt \left(1-(*\Om) + c_6 \psi\right)   - \frac{1}{3}(*\Om)|\IIt - \II^\Sm|^2 - \frac{c_7}{c_6}\left( 1-(*\Om) + c_6\psi \right)
\end{split} \label{evol6} \end{align}
where $c_6 = 1+{c_5}/{c_1}$.  By the maximum principle, $\max_{\Gm_t}\left(1 - (*\Om) + c_6\psi\right)$ is non-increasing.

The evolution equation for the norm of the second fundamental form for a mean curvature flow is derived in \cite{ref_W1}*{Proposition 7.1}.  In particular, $|\IIt|^2 = \sum_{\af, i, k} \htd{\af}{i}{k}^2$ satisfies the following
equation along the flow:
\begin{align} \label{|A|^2} \begin{split}
\frac{\dd}{\dd t}|\IIt|^2 &= \Dt |\IIt|^2 -2|\nabla^{\Gm_t} \IIt|^2
+2\left[ (\namb_{\te_k} R)_{\td{\af}\td{i}\td{j}\td{k}} + (\namb_{\te_j} R)_{\td{\af}\td{k}\td{i}\td{k}} \right] \htd{\af}{i}{j}  \\
& - 4 \RTD{l}{i}{j}{k}\htd{\af}{l}{k}\htd{\af}{i}{j} + 8 \RTD{\af}{\bt}{j}{k}\htd{\bt}{i}{k}\htd{\af}{i}{j}
- 4 \RTD{l}{k}{i}{k}\htd{\af}{l}{j}\htd{\af}{i}{j} + 2 \RTD{\af}{k}{\bt}{k}\htd{\bt}{i}{j}\htd{\af}{i}{j}  \\
& + 2\sum_{\af,\gm, i,j} \left( \sum_k \htd{\af}{i}{k}\htd{\gm}{j}{k} - \htd{\af}{j}{k}\htd{\gm}{i}{k}\right)^2
+ 2 \sum_{i,j,k,l} \left( \sum_{\af} \htd{\af}{i}{j}\htd{\af}{k}{l} \right)^2 ~.
\end{split} \end{align}
It follows that
\begin{align}
\frac{\dd}{\dd t}|\IIt|^2 &\leq \Dt |\IIt|^2 -2|\nabla^{\Gm_t} \IIt|^2 + c_8 \left( |\IIt|^4 + |\IIt|^2 + 1 \right) ~.
\label{evol3} \end{align}

The quartic term $|\IIt|^4$ could potentially lead to the finite time blow-up of $|\IIt|$.  We apply the same method in \cite{ref_W3}: use the evolution equation of $(*\Om)^p$ to help.  Let $p$ be a constant no less than $1$, whose precise value will be determined later.  According to \eqref{evol2},
\begin{align*}
\frac{\dd}{\dd t} (*\Om)^p &= p(*\Om)^{p-1}\frac{\dd}{\dd t}(*\Om) \\
&\geq p(*\Om)^{p-1}\Dt(*\Om) + \frac{p}{3} (*\Om)^p|\IIt - \II^\Sm|^2 - c_5 p (\fs^2 + \psi) \\
&= \Dt(*\Om)^p - p(p-1)(*\Om)^{p-2}|\nabla^{\Gm_t}(*\Om)|^2 + \frac{p}{3} (*\Om)^p|\IIt - \II^\Sm|^2 - c_5 p (\fs^2 + \psi) ~.
\end{align*}
After an appeal to Lemma \ref{lem_cov*},
\begin{align*}
\frac{\dd}{\dd t}(*\Om)^p &\geq \Dt(*\Om)^p + \frac{p}{3}\left(1 - c_9\,p\,\fs^2\right) (*\Om)^p|\IIt - \II^\Sm|^2 - c_9\,p^2 (\fs^2 + \psi) ~.
\end{align*}
If $\kp\leq 1/(24c_9 p)$, it follows from \eqref{boundfs} that $c_9\,p\,\fs^2\leq1/12$.  It together with \eqref{evol1} gives that
\begin{align}
\frac{\dd}{\dd t}\left( (*\Om)^p - K p^2\psi \right) &\geq \Dt\left( (*\Om)^p - K p^2\psi \right) + \frac{p}{4} \left( (*\Om)^p - K p^2\psi \right) |\IIt - \II^\Sm|^2
\label{evol4} \end{align}
where $K = {c_9}/{c_1}$.  The maximum principle implies that if $(*\Om)^p - K p^2\psi > 0$ on $\Gm$, then $\min_{\Gm_t}\left( (*\Om)^p - K p^2\psi \right)$ is non-decreasing.  Moreover, for any $p\geq1$, we may choose $\kp$ such that \eqref{smallC1} implies that $(*\Om)^p - K p^2\psi > 1/2$ on $\Gm$.

Denote $(*\Om)^p - K p^2\psi$ by $\eta$.  Due to \eqref{evol3} and \eqref{evol4},
\begin{align*}
\frac{\dd}{\dd t} (\eta^{-1}|\IIt|^2) &\leq \eta^{-1}\Dt |\IIt|^2 - 2\eta^{-1}|\nabla^{\Gm_t} \IIt|^2 + c_8\eta^{-1} \left( |\IIt|^4 + |\IIt|^2 + 1 \right) \\
&\quad - \eta^{-2}|\IIt|^2 \left( \Dt\eta + \frac{p}{4}\eta |\IIt - \II^\Sm|^2 \right) ~.
\end{align*}
Since
\begin{align*}
\Dt (\eta^{-1}|\IIt|^2) &= \eta^{-1}\Dt|\IIt|^2 + |\IIt|^2\Dt\eta^{-1} + 2\ip{\nabla^{\Gm_t}\eta^{-1}}{\nabla^{\Gm_t}|\IIt|^2} \\
&= \eta^{-1}\Dt|\IIt|^2 - \eta^{-2}|\IIt|^2\Dt\eta + 2\eta^{-3}|\nabla^{\Gm_t}\eta|^2|\IIt|^2- 2\eta^{-2}\ip{\nabla^{\Gm_t}\eta}{\nabla^{\Gm_t}|\IIt|^2}\\
&= \eta^{-1}\Dt|\IIt|^2 - \eta^{-2}|\IIt|^2\Dt\eta - 2\eta^{-1}\ip{\nabla^{\Gm_t}\eta}{\nabla^{\Gm_t}(\eta^{-1}|\IIt|^2)}
\end{align*}
and $|\IIt - \II^\Sm|^2 \geq \oh |\IIt|^2 - c_{10}$, we have
\begin{align}
\frac{\dd}{\dd t} (\eta^{-1}|\IIt|^2) &\leq \Dt (\eta^{-1}|\IIt|^2) + 2\eta^{-1}\ip{\nabla^{\Gm_t}\eta}{\nabla^{\Gm_t}(\eta^{-1}|\IIt|^2)}  \notag \\
&\quad - \frac{p}{8}\eta^{-1}|\IIt|^4 + c_{10}\frac{p}{4}\eta^{-1}|\IIt|^2 + c_{8}\eta^{-1} \left( |\IIt|^4 + |\IIt|^2 + 1 \right)  \notag \\
\begin{split}  &\leq \Dt (\eta^{-1}|\IIt|^2) + 2\eta^{-1}\ip{\nabla^{\Gm_t}\eta}{\nabla^{\Gm_t}(\eta^{-1}|\IIt|^2)} \\
&\quad - \oh\left( \frac{p}{8} - c_8 \right) (\eta^{-1}|\IIt|^2)^2 + \left( c_{8}+c_{10}\frac{p}{4}\right) (\eta^{-1}|\IIt|^2) + 2c_{8}
\end{split} \label{evol5} \end{align}
(provided that $p\geq 8c_8$).  Choose $p>8c_{8}$.  It follows from the maximum principle that $\eta^{-1}|\IIt|^2$ is uniformly bounded, and hence there is no finite time singularity.

\

The $\CC^0$ convergence is easy to come by.  The differential inequality \eqref{evol1} implies that $\psi$ converges to zero exponentially.  Similarly, it follows from \eqref{evol6} that $1-(*\Om) + c_6\psi$ converges to zero exponentially.  Therefore, $*\Om$ converges to $1$ exponentially, and we conclude the $\CC^1$ convergence.

\

For the $\CC^2$ and smooth convergence, consider $\eta = (*\Om)^p - Kp^2\psi$.  It follows from the above discussion that $\eta$ has a positive lower bound.  It is clear that $\eta\leq 1$.  Moreover, $\eta$ converges to $1$ as $t\to\infty$.  Integrating \eqref{evol4} gives
\begin{align*}
\int_{\Gm_t} |\IIt - \II^\Sm|^2 \,\dd\mu_t &\leq c_{12} \int_{\Gm_t} \frac{\dd\eta}{\dd t}\,\dd\mu_t ~.
\end{align*}

Recall that the Lie derivative of $\dd\mu_t$ in $H_t$ is $-|H_t^2|\,\dd\mu_t$; see \cite[\S2]{ref_W1}.  It follows that
\begin{align}
\frac{1}{c_{12}} \int_{\Gm_t} |\IIt - \II^\Sm|^2 \,\dd\mu_t &\leq \frac{\dd}{\dd t} \int_{\Gm_t} \eta\,\dd\mu_t + \int_{\Gm_t} \eta\,|H_t|^2\,\dd\mu_t
\label{evol9} \end{align}

We claim that the improper integral of the right hand side for $0\leq t<\infty$ converges.  To start, note that $\frac{\dd}{\dd t}\int_{\Gm_t}\dd\mu_t = -\int_{\Gm_t}|H_t|^2\,\dd\mu_t \leq 0$.  Thus, $\vol(\Gm_t) = \int_{\Gm_t}\dd\mu_t$ is positive and non-increasing, and must converge as $t\to\infty$.  For the first term on right hand side of \eqref{evol9},
\begin{align*}
\int_0^{t} \left(\frac{\dd}{\dd s} \int_{\Gm_s} \eta\,\dd\mu_s\right)\,\dd s &= \int_{\Gm_{t}} \eta\,\dd\mu_{t} - \int_{\Gm_0} \eta\,\dd\mu_0
\end{align*}
Since $\eta$ converges to $1$ (uniformly) and $\vol(\Gm_t)$ converges as $t\to\infty$, $\int_{\Gm_{t}} \eta\,\dd\mu_{t}$ converges as $t\to\infty$.  For the second term on the right hand side of \eqref{evol9},
\begin{align*}
\int_0^{t} \left( \int_{\Gm_s} \eta\,|H_s|^2\,\dd\mu_s \right)\dd s &\leq \int_0^{t} \left( \int_{\Gm_s} |H_s|^2\,\dd\mu_s \right)\dd s = \int_0^{t} \left( -\frac{\dd}{\dd s}\int_{\Gm_s} \dd\mu_s \right)\dd s \\
&= \vol(\Gm_0) - \vol(\Gm_t) \leq \vol(\Gm_0) ~.
\end{align*}
It is bounded from above, and is clearly non-decreasing in $t$.  Therefore, it converges as $t\to\infty$.

It follows from the claim and \eqref{evol9} that
\begin{align}
\int_0^\infty \left( \int_{\Gm_t} |\IIt - \II^\Sm|^2 \,\dd\mu_t \right)\dd t  &<\infty ~.
\label{evol7} \end{align}
On the other hand, $|\IIt - \II^\Sm|^2$ obeys a differential inequality of the same form as \eqref{evol3}:
\begin{align}
\frac{\dd}{\dd t}|\IIt - \II^\Sm|^2 &\leq \Dt |\IIt - \II^\Sm|^2 + c_{13} \left( |\IIt - \II^\Sm|^4 + |\IIt - \II^\Sm|^2 + 1 \right) ~.
\label{evol8} \end{align}
The derivation for this inequality is in Appendix \ref{apx_evol}.
By \eqref{evol8} and the uniform boundedness of $|\IIt|$,
$\frac{\dd}{\dd t}\int_{\Gm_t} |\IIt - \II^\Sm|^2 \,\dd\mu_t $ is bounded from above uniformly.  Due to Lemma \ref{cal}, which is proved at the end of this subsection, we find that
\begin{align}
\lim_{t\to 0} \int_{\Gm_t} |\IIt - \II^\Sm|^2 \,\dd\mu_t &=0 ~.
\label{int} \end{align}


{Since we have shown long time existence and convergence in $\CC^1$, for $t$ large enough $\Gamma_t$ can be written as a graph (in the geodesic coordinate defined in \S\ref{sec_coordinate}) over $\Sigma$ defined by $y^\alpha=f^\alpha_t, \alpha=n+1,\cdots n+m$ for $\CC^1$ functions $f^\alpha_t$ on $\Sigma$.
With  \eqref{int}, a Moser iteration argument similar to \cite[\S5]{ref_Ilmanen} shows that $f^\af_t$ converges to $0$ in $\CC^2$.  The detail of this argument is included in Appendix \ref{apx_Moser}.
With the $\CC^2$ convergence, standard arguments for a second order quasilinear parabolic system lead to the smooth convergence of $f^\alpha_t$.}
\end{proof}

\begin{lem} \label{cal}
Let $a>0$ and $f(t)$ be a smooth function for $t\in(a,\infty)$.  Suppose that $f(t)\geq0$, $\int_a^\infty f(t)\,\dd t$ converges, and $f'(t)\leq C$ for some constant $C>0$.  Then, $f(t)\to 0$ as $t\to\infty$.
\end{lem}

\begin{proof}
It follows from $f'(t)\leq C$ that $f(t)\geq f(t_1) - C(t_1 - t)$ for any $t_1>t>a$.  Since $f(t)\geq0$ and $\int_a^\infty f(t)\,\dd t < \infty$, given any $\ep\in(0,1)$, there exists an $A_\ep>a$ such that $\int_{A_\ep}^\infty f(t)\,\dd t<\ep$.  Thus, for any $t_1 > A_\ep + 1 > A_\ep + \sqrt{\ep}$,
\begin{align*}
\ep > \int_{t_1 - \sqrt{\ep}}^{t_1} f(t)\,\dd t &\geq \int_{t_1 - \sqrt{\ep}}^{t_1} \left( f(t_1) - C(t_1 - t) \right)\dd t \\
& = \sqrt{\ep}\left( f(t_1) - C\,t_1\right) + C\left( \sqrt{\ep}\,t_1 - \oh\ep \right) ~.
\end{align*}
It follows that $f(t)<(1+\oh C)\sqrt{\ep}$ for any $t > A_\ep+1$.
\end{proof}

\subsection{With only the stability condition}\label{only_stable}

Theorem \ref{longtime} asserts that a strongly stable minimal submanifold is $\CC^1$ dynanical stable under the mean curvature flow.  Recall that a minimal submanifold is said to be stable if the Jacobi operator, $(\nnor)^*\nnor + \CR - \CA$, is a positive operator.  It is a natural question whether the stability condition already implies the dynamical stability.  This was investigated by Naito in \cite{ref_Naito}.  The answer is yes, but initial submanifold has to be close to the stable one in a \emph{higher} norm.

The approach of Naito is to consider only graphical/sectional type submanifold, i.e.\ $\Gm_t$ is given by a section $\bs$ of the normal bundle of $\Sm$.  Then, the mean curvature flow equation takes the following form
\begin{align}
\frac{\pl\bs}{\pl t} &= -\left( (\nnor)^*\nnor + \CR - \CA \right)(\bs) + \CN(\bs)
\label{MCF_graph} \end{align}
where $\CN$ is a ``small" operator in the sense of \cite[(C3) on p.222]{ref_Naito}.  The positivity of the Jacobi operator takes over the behavior of the parabolic system if one works with a suitable norm.  In the mean curvature flow case, \cite[Theorem 5.3]{ref_Naito} reads as follows.

\begin{thm}\label{stable}
Let $\Sm^n\subset(M,g)$ be a compact, oriented, stable minimal submanifold.  Then, for any $r>\frac{n}{2} + 2$, there exists a positive constant $\kp'<\!\!<1$ such that for any section $\bs_0$ of $N\Sm$ with $\ltn{\bs_0}{r} \leq\kp'$, the mean curvature flow \eqref{MCF_graph} exists for any $t>0$.  Moreover, the solution converges to $0$ exponentially as $t\to\infty$ in $H^{r}$ norm.
\end{thm}

According to the Sobolev embedding theorem, $H^{r}\hookrightarrow \CC^{\lfloor r-\frac{n}{2} \rfloor,\af}$ where $\af = r- \frac{n}{2} - \lfloor r-\frac{n}{2} \rfloor\in(0,1]$.  In other words, the assumption implies that $s_0$ has small $\CC^{2,\af}$ norm.

\begin{rmk}
The initial condition in \cite[Proposition 5.2 and Theorem 5.3]{ref_Naito} looks more complicated than above stated.  See \cite[p.223--224]{ref_Naito} for the norm.  The reason is that Naito also discussed the case when the Jacobi operator has non-positive spectrum.  If the Jacobi operator is positive, one can check directly by using the spectral decomposition that the norm condition there is equivalent to the $H^r$ norm ($r$ here corresponds to $m$ in \cite{ref_Naito}).
\end{rmk}

In the rest of this subsection, we will briefly explain how to set up \eqref{MCF_graph} and check the condition \cite[(C3)]{ref_Naito}.  Although this was demonstrated in \cite[p.233--235]{ref_Naito} under a general setting, it is instructive to do it for the mean curvature flow case.

\subsubsection{The expansion in radial direction}\label{sec_nonpara}

To highlighting the computation for a section of $N\Sm$, we build the expansion of the metric coefficients in the radial directions.  Let $\{x^i\}$ be an oriented, local coordinate system of $\Sm$.  Choose a local orthonormal frame $\{e_\mu\}$ for $N\Sm$.  As in section \ref{sec_coordinate}, construct a local coordinate of $M$ by
\begin{align*}
\left( (x^1,\cdots,x^n), (y^{n+1},\cdots,y^{n+m}) \right) &\mapsto \exp_{(x^1,\cdots,x^n)}(y^\bt e_\bt) ~.
\end{align*}
Here, $\{x^i\}$ needs not to be the Gaussian coordinate for $\Sm$.  Denote by $g_{AB}(\bx,\by)$ the coefficients of the Riemannian metric $g$ in this coordinate system.

Denote the normal bundle connection by $A_\mu^\nu$, i.e.\ $\nabla^\perp e_\mu = (A_{\mu\,i}^\nu\,\dd x^i)\,e_\nu$.  By the Jacobi field argument similar to that in section \ref{sec_coordinate},
\begin{align} \begin{split} \begin{cases}
g_{ij}(\bx,\by) = g_{ij} - 2y^\bt\hm{\bt}{i}{j} - y^\mu y^\nu(\RM{j}{\mu}{i}{\nu} - \hm{\mu}{i}{k}\hm{\nu}{j}{\ell}g^{k\ell} - A_{\mu\,i}^\bt A_{\nu\,j}^\bt) + \CO(|\by|^3) ~, & \\
g_{i\mu}(\bx,\by) = y^\nu A_{\nu\,i}^\mu + \CO(|\by|^2) ~, & \\
g_{\mu\nu}(\bx,\by) = \dt_{\mu\nu} + \CO(|\by|^2) &
\end{cases} \end{split} \label{graph1} \end{align}
where all the coefficient functions on the right hand side are evaluated at $(\bx,0)\in\Sm$.

For a section $\Gm = \{y^\mu = y^\mu(\bx)\}$, the $ij$ component of the induced metric is
\begin{align}
\td{g}_{ij} &= g_{ij}(\bx,\by) + g_{i\mu}(\bx,\by)\,\pl_j y^\mu + g_{j\mu}(\bx,\by)\,\pl_i y^\mu + g_{\mu\nu}(\bx,\by)(\pl_i y^\mu)(\pl_j y^\nu) \notag \\
\begin{split} &= g_{ij} - 2y^\bt\hm{\bt}{i}{j} - y^\mu y^\nu(\RM{j}{\mu}{i}{\nu} - \hm{\mu}{i}{k}\hm{\nu}{j}{\ell}g^{k\ell}) + y^\mu_{\;;i}y^\mu_{\;;j} \\
&\qquad + \CO(|\by|^3) + \CO(|\by|^2)(\pl\by) + \CO(|\by|^2)(\pl\by)^2
\end{split} \label{graph2} \end{align}
where $y^\mu_{\;;i} = \pl_i y^\mu + A^\mu_{\nu\,i} y^\nu$.  We will also need
\begin{align}
\td{g}_{\mu\nu} &= \left\langle \left(\frac{\pl~}{\pl y^\mu}\right)^\perp, \left(\frac{\pl~}{\pl y^\nu}\right)^\perp \right\rangle = g_{\mu\nu} - \td{g}^{ij} ( g_{\mu j} + g_{\mu\af}\pl_j y^\af ) ( g_{\nu i} + g_{\nu\bt}\pl_i y^\bt )
\label{graph3} \end{align}
where $\perp$ is the orthogonal projection onto $N\Gm$.

It is convenient to think $\pl_j y^\nu$ as a dummy variable $p^\nu_j$, and \emph{regard them as the same order as $y^\mu$}.  The negative gradient flow of $\vol(\Gm) = \int \sqrt{\det\td{g}} \, \dd x^1\w\cdots\w\dd x^n$ is almost the mean curvature flow.
Recall that
\begin{align*}
\det(I+T) &= 1 + \tr(T) + \oh(\tr^2(T) - \tr(T^2)) + \CO(|T|^3)
\end{align*}
for $T$ small.  With $g^{ij}\hm{\bt}{i}{j} = 0$,
\begin{align*}
\frac{\det\td{g}}{\det g} &= 1 + g^{ij}\left[ y^\mu_{\;;i}y^\mu_{\;;j} - y^\mu y^\nu(R_{j\mu i\nu} - h_{\mu ik}h_{\nu j\ell}g^{k\ell})\right] + \oh (2y^\mu h_{\mu ik}g^{k\ell})(2y^\nu h_{\nu \ell j}g^{ji}) + \CO(3) \\
&= 1 + g^{ij}\left[ y^\mu_{\;;i}y^\mu_{\;;j} - y^\mu y^\nu(R_{j\mu i\nu} + h_{\mu ik}h_{\nu j\ell}g^{k\ell})\right] + \CO(3) ~.
\end{align*}
The notation $\CO(3)$ means $\CO((|\by|^2+|\bp|^2)^{\frac{3}{2}})$ for $\by$, $\bp$ small; see also \eqref{orderk} for the definition of this notation.  Hence,
\begin{align}
\sqrt{\det\td{g}} &= \sqrt{\det g} + \frac{\sqrt{\det g}}{2} g^{ij}\left[ y^\mu_{\;;i}y^\mu_{\;;j} - y^\mu y^\nu(R_{j\mu i\nu} + h_{\mu ik}h_{\nu j\ell}g^{k\ell})\right] + \CE(\bx,\by,\bp)
\label{Lagf} \end{align}
where $\CE = \CO(3)$.
Since the computation is based on the graphical setting, the mean curvature flow equation shall be $(\pl_t\by)^\perp = H_{\Gm_t}$.  With some linear algebraic computations, the mean curvature flow equation reads
\begin{align}
\frac{\pl y^\mu}{\pl t} &= \frac{\td{g}^{\mu\nu}}{\sqrt{\det\td{g}}} \left[ \frac{\dd~}{\dd x^k}\left(\frac{\pl\sqrt{\det\td{g}}}{\pl p^\nu_k}\right) - \frac{\pl\sqrt{\det\td{g}}}{\pl y^\nu} \right] ~.
\label{MCF_graph1} \end{align}
According to \eqref{graph1}, \eqref{graph2}, \eqref{graph3} and \eqref{Lagf}, $\td{g}^{\mu\nu} = \dt^{\mu\nu} + \CO(2)$.

The zero-th order term of \eqref{Lagf} gives the volume of the zero section, and has no contribution in the Euler--Lagrange equation.  The integral of the quadratic term of \eqref{Lagf} is
\begin{align*}
\oh \int_\Sm \left(|\nabla^\perp \by|^2 + \ip{\CR(\by)}{\by} - \ip{\CA(\by)}{\by}\right)\dv_\Sm ~.
\end{align*}
Therefore, its Euler--Lagrange operator is the Jacobi operator $(\nabla^\perp)^*\nabla^\perp + \CR - \CA$ at $\by = 0$, with respect to the volume form $\dv_\Sm$.  It follows that the contribution of the quadratic terms of \eqref{Lagf} to the right hand side of \eqref{MCF_graph1} is
\begin{align*}
& {\td{g}^{\mu\nu}}\sqrt{\frac{\det g}{\det\td{g}}} \left(-[(\nabla^\perp)^*\nabla^\perp + \CR - \CA]\by\right)^\mu \\
=\,& \left(-[(\nabla^\perp)^*\nabla^\perp + \CR - \CA]\by\right)^\mu + \CE^\mu_\nu(\bx,\by,\bp)\left(-[(\nabla^\perp)^*\nabla^\perp + \CR - \CA]\by\right)^\nu
\end{align*}
where $\CE^\mu_\nu = \CO(2)$.

To sum up, the mean curvature flow equation takes the following form
\begin{align*}
\frac{\pl y^\mu}{\pl t} &= \left(-[(\nabla^\perp)^*\nabla^\perp + \CR - \CA]\by\right)^\mu \\
&\quad + \CE^\mu_\nu\cdot\left(-[(\nabla^\perp)^*\nabla^\perp + \CR - \CA]\by\right)^\nu + \frac{\td{g}^{\mu\nu}}{\sqrt{\det\td{g}}} \left[ \frac{\dd~}{\dd x^k}\left(\frac{\pl\CE}{\pl p^\nu_k}\right) - \frac{\pl\CE}{\pl y^\nu} \right] ~.
\end{align*}
The operator in the last line in the remainder operator $\CN$; see \eqref{MCF_graph}.

\subsubsection{The smallness of the remainder operator}

The condition \cite[(C3)]{ref_Naito} says that for any $r>\frac{n}{2}+2$ there exists a constant $C>0$ such that
\begin{align}
\ltn{\CN(\by) - \CN(\td{\by})}{r-1} &\leq C\left( \ltn{\by}{r} \ltn{\by-\td{\by}}{r+1} + \ltn{\by - \td{\by}}{r} \ltn{\td{\by}}{r+1} \right)
\label{C3} \tag{C3} \end{align}
for any two sections $\by,\td{\by}\in L^2_{r+1}$ with $\ltn{\by}{r+1}<1$, $\ltn{\td{\by}}{r+1}<1$.

There are some different types of terms in $\CN$.  What follows is a brief explanation for terms like $f(\bx,\by,\pl\by) (\pl\by) (\pl^2\by)$.  Write
\begin{align*}
& f(\bx,\by,\pl\by) (\pl\by) (\pl^2\by) - f(\bx,\td{\by},\pl\td{\by}) (\pl\td{\by}) (\pl^2\td{\by}) \\
=\,& f(\bx,\by,\pl\by) (\pl\by) \left[ (\pl^2\by) - (\pl^2\td{\by}) \right] \\
&\quad + \left[ h_1(\bx,\by,\td{\by},\pl{\by},\pl\td{\by})\,(\by-\td{\by}) +  h_2(\bx,\by,\td{\by},\pl{\by},\pl\td{\by})\,(\pl\by-\pl\td{\by})\right] (\pl^2\td{\by})
\end{align*}
where $h_1$ and $h_2$ are constructed from the derivatives of $f(\bx,\by,\pl\by) (\pl\by)$ in $\by$ and in $\pl\by$, respectively.  With the following two bullets, $f(\bx,\by,\pl\by) (\pl\by) (\pl^2\by)$ does obey the condition \eqref{C3}.
\begin{itemize}
\item According to \cite[Lemma 9.9]{ref_Palais}, the ``coefficients" $f(\bx,\by,\pl\by)$, $h_1(\bx,\by,\td{\by},\pl{\by},\pl\td{\by})$ and $h_1(\bx,\by,\td{\by},\pl{\by},\pl\td{\by})$ have bounded $H^r$ norm.
\item Due to \cite[Theorem 9.5 and Corollary 9.7]{ref_Palais}, for any $r-1>\frac{n}{2}$, there exists $C'>0$ such that
\begin{align*}
\ltn{\bu\,\bv}{r-1} \leq C'\, \ltn{\bu}{r-1}\, \ltn{\bv}{r-1}
\end{align*}
for any $\bu,\bv\in H^{r-1}$.
\end{itemize}

For other types of terms, the argument is similar.

\begin{appendix}

\section{Computations related to strong stability} \label{apx_stable}

For minimal Lagrangians in a K\"ahler--Einstein manifold and coassociatives in a $G_2$ manifold, the condition \eqref{sstable} can be rewritten as a curvature condition on the submanifold.  One ingredient is the geometric properties of ${\rm U}(n)$ and $G_2$ holonomy.  Another ingredient is the Gauss equation:
\begin{align}
\RM{i}{j}{k}{\ell} - \RM{i}{j}{k}{\ell}^\Sm &= \hm{\af}{i}{\ell}\hm{\af}{j}{k} - \hm{\af}{i}{k}\hm{\af}{j}{\ell} ~.
\end{align}

\subsection{Minimal Lagrangians in K\"ahler--Einstein manifolds} \label{apx_Lag}

Let $(M^{2n},g,J,\om)$ be a K\"ahler--Einstein manifold, where $J$ is the complex structure and $\om$ is the K\"ahler form.  Denote the Einstein constant by $c$; namely,
\begin{align*}
\sum_C\RM{A}{C}{B}{C} = \Ric_{AB} &= c\,g_{AB} ~.
\end{align*}
A submanifold $L^n\subset M^{2n}$ is Lagrangian if $\om|_L$ vanishes.  It implies that $J$ induces an isomorphism between its tangent bundle $TL$ and normal bundle $NL$.  In terms of the notations introduced in \S\ref{sec_notation}, the correspondence is
\begin{align}
v^i e_i \longleftrightarrow v^i\, Je_i ~.
\label{Lag_TN} \end{align}

In particular, if $\{e_1,\cdots,e_n\}$ is an orthonormal frame for $TL$, $\{Je_1,\cdots, Je_n\}$ is an orthonormal frame for $NL$.  Denote $Je_k$ by $e_{J(k)}$, and let
\begin{align*}
C_{kij} &= \hm{J(k)}{i}{j} = \ip{\nabla_{e_i}e_j}{Je_k} ~.
\end{align*}
Since $J$ is parallel, it is easy to verify that $C_{kij}$ is totally symmetric.

Now, suppose that $L$ is also minimal.  By using the correspondence \eqref{Lag_TN}, the strong stability condition \eqref{sstable} can be rewritten as follows.
\begin{align*}
-\RM{i}{J(k)}{i}{J(\ell)}\,v^k\,v^\ell - C_{kij}\,C_{\ell ij}\,v^k\,v^\ell
&= - c\,g_{k\ell}\,v^k\,v^\ell + \RM{J(i)}{J(k)}{J(i)}{J(\ell)}\,v^k\,v^\ell - C_{kij}\,C_{\ell ij}\,v^k\,v^\ell \\
&= -c\,|v|^2 + \RM{i}{k}{i}{\ell}\,v^k\,v^\ell - C_{kij}\,C_{\ell ij}\,v^k\,v^\ell \\
&= -c\,|v|^2 + \RM{i}{k}{i}{\ell}^L\,v^k\,v^\ell + C_{jki}\,C_{j\ell i}\,v^k\,v^\ell - C_{kij}\,C_{\ell ij}\,v^k\,v^\ell \\
&= -c\,|v|^2 + \Ric^L(v,v) ~.
\end{align*}
The first equality uses the K\"ahler--Einstein condition.  The second equality follows from the parallelity of $J$.  The third equality uses the Gauss equation and the minimal condition.  The last equality relies on the fact that $C_{kij}$ is totally symmetric.  This computation says that \eqref{sstable} is equivalent to the condition that $\Ric^L - c$ is a positive definite operator on $TL$.

\subsection{Coassociative submanifolds in $G_2$ manifolds} \label{apx_coa}

In this case, the ambient space is $7$-dimensional, and the submanifold is $4$-dimensional.

\subsubsection{Four dimensional Riemannian geometry}
The Riemann curvature tensor has a nice decomposition in $4$ dimensions.  What follows is a brief summary of the decomposition; readers are directed to \cite{ref_AHS} for more.

Let $\Sm$ be an oriented, $4$-dimensional Riemannian manifold.  The Riemann curvature tensor in general defines a self-adjoint transform on $\Ld^2$ by
\begin{align*}
\fR(e_i\w e_j) &= \oh\RM{k}{\ell}{i}{j}^\Sm\,e_k\w e_\ell ~.
\end{align*}
In $4$ dimensions, $\Ld^2$ decomposes into self-dual, $\Ld_+^2$, and anti-self-dual part, $\Ld_-^2$.  In terms of the decomposition $\Ld^2 = \Ld_+^2\oplus\Ld_-^2$, the curvature map $\fR$ has the form
\begin{align*}
\fR &= \begin{bmatrix}
W_+ + \frac{s}{12}\,\mathbf{I} & B \\ B^T & W_- + \frac{s}{12}\,\mathbf{I}
\end{bmatrix} ~.
\end{align*}
Here, $s = R_{ijij}^\Sm$ is the scalar curvature, $W_\pm$ is the self-dual and anti-self-dual part of the Weyl tensor, $B$ is the traceless Ricci tensor, and $\mathbf{I}$ is the identity homomorphism.

With respect to the basis $\{e_1\w e_2 - e_3\w e_4, e_1\w e_3 + e_2\w e_4, e_1\w e_4 - e_2\w e_3\}$, the lower-right block $W_- + \frac{s}{12}\,\mathbf{I}$ is
\begin{align} {\small \oh \begin{bmatrix}
R_{1212}^\Sm + R_{3434}^\Sm - 2R_{1234}^\Sm & R_{1213}^\Sm + R_{1224}^\Sm - R_{3413}^\Sm - R_{3424}^\Sm & R_{1214}^\Sm - R_{1223}^\Sm - R_{3414}^\Sm + R_{3423}^\Sm \bigskip \\
R_{1312}^\Sm - R_{1334}^\Sm + R_{2412}^\Sm - R_{2434}^\Sm & R_{1313}^\Sm + R_{2424}^\Sm + 2R_{1324}^\Sm & R_{1314}^\Sm - R_{1323}^\Sm + R_{2414}^\Sm - R_{2423}^\Sm \bigskip \\
R_{1412}^\Sm - R_{1434}^\Sm - R_{2312}^\Sm + R_{2334}^\Sm & R_{1413}^\Sm + R_{1424}^\Sm - R_{2313}^\Sm - R_{2324}^\Sm & R_{1414}^\Sm + R_{2323}^\Sm - 2R_{1423}^\Sm
\end{bmatrix}}~.
\label{ASD_curv} \end{align}
The operator will be needed is $W_- - \frac{s}{6}\,\mathbf{I} = (W_- + \frac{s}{12}\,\mathbf{I}) - \frac{s}{4}\,\mathbf{I}$.  One-fourth of the scalar curvature is
\begin{align}
\frac{s}{4} &= \oh\left( R_{1212}^\Sm + R_{3434}^\Sm + R_{1313}^\Sm + R_{2424}^\Sm + R_{1313}^\Sm + R_{2424}^\Sm \right) ~.
\label{4d_scurv} \end{align}

\subsubsection{$G_2$ geometry}

A $7$-dimensional Riemannian manifold $M$ whose holonomy is contained in $G_2$ can be characterized by the existence of a parallel, positive $3$-form $\vph$.  A complete story can be found in \cite[ch.11]{ref_Joyce1}.  In terms of a local orthonormal coframe, the $3$-form and its Hodge star are
\begin{align} \begin{split}
\vph &= \om^{567} + \om^{125} - \om^{345} + \om^{136} + \om^{246} + \om^{147} - \om^{237} ~, \\
*\vph &= \om^{1234} - \om^{1267} + \om^{3467} + \om^{1357} + \om^{3457} - \om^{1456} + \om^{2356}
\end{split} \label{G2_form} \end{align}
where $\om^{123}$ is short for $\om^1\w\om^2\w\om^3$.  It is known that the holonomy is $G_2$ if and only if $\nabla\vph=0$, which is also equivalent to $\dd\vph = 0 = \dd*\vph$.

\begin{rmk}
There are two commonly used conventions for the $3$-form; see \cite{ref_Spiro} for instance.  The convention here is the same as that in \cite{ref_McLean}; the deformation of coassociatives will then be determined by anti-self-dual harmonic forms.  If one use the convention in \cite{ref_Joyce1}, the deformation of coassociatives will be determined by self-dual harmonic forms.
\end{rmk}

The $3$-form $\vph$ determines a product map $\times$ for tangent vectors of $M$.  For any two tangent vectors $X$ and $Y$,
\begin{align*}
X\times Y &= \big( \vph(X,Y,\,\cdot\,) \big)^\sharp ~.
\end{align*}
For instance, $e_1\times e_2 = e_5$.  Since $\vph$ and the metric tensor are both parallel, $\times$ is parallel as well.

As a consequence,
\begin{align*}
R(e_A,e_B)(e_1\times e_2) &= \big(R(e_A,e_B)e_1\big)\times e_2 + e_1\times\big(R(e_A,e_B)e_2\big) ~,
\end{align*}
and its $e_3$-component gives $\RM{5}{3}{A}{B} - \RM{6}{2}{A}{B} - \RM{7}{1}{A}{B} = 0$ for any $A,B\in\{1,\ldots,7\}$.  In total, the parallelity of $\times$ leads to following seven identities:
\begin{align}
\begin{split}
\RM{5}{2}{A}{B} + \RM{6}{3}{A}{B} + \RM{7}{4}{A}{B} &= 0 ~, \\
\RM{5}{1}{A}{B} - \RM{6}{4}{A}{B} + \RM{7}{3}{A}{B} &= 0 ~, \\
\RM{5}{4}{A}{B} + \RM{6}{1}{A}{B} - \RM{7}{2}{A}{B} &= 0 ~, \\
-\RM{5}{3}{A}{B} + \RM{6}{2}{A}{B} + \RM{7}{1}{A}{B} &= 0 ~,
\end{split} \begin{split}
\RM{6}{7}{A}{B} + \RM{1}{2}{A}{B} - \RM{3}{4}{A}{B} &= 0 ~, \\
-\RM{5}{7}{A}{B} + \RM{1}{3}{A}{B} + \RM{2}{4}{A}{B} &= 0 ~, \\
\RM{5}{6}{A}{B} - \RM{1}{4}{A}{B} - \RM{2}{3}{A}{B} &= 0 ~.
\end{split} \label{G2_curv} \end{align}
These identities imply that a $G_2$ manifold is always Ricci flat.

\subsubsection{Coassociative geometry}

According to \cite[\S IV]{ref_HL}, an oriented, $4$-dimensional submanifold $\Sm$ of a $G_2$ manifold is said to be coassociative if $*\vph|_\Sm$ coincides with the volume form of the induced metric.  Harvey and Lawson also proved that if $\vph|_\Sm$ vanishes, there is an orientation on $\Sm$ so that it is coassociative.  Similar to the Lagrangian case, the normal bundle of a coassociative submanifold is canonically isomorphic to an intrinsic bundle.  The following discussion is basically borrowed from \cite[\S4]{ref_McLean}.

\paragraph{\it \;\, Orthonormal frame}
Suppose that $\Sm\subset M$ is coassociative.  One can find a local orthonormal frame $\{e_1,\cdots, e_7\}$ such that $\{e_1,e_2,e_3,e_4\}$ are tangent to $\Sm$, $\{e_5,e_6,e_7\}$ are normal to $\Sm$, and $\vph$ takes the form \eqref{G2_form} in this frame.  Here is a sketch of the construction.  Start with a unit normal vector, $e_5$, and a unit tangent vector, $e_1$, of $\Sm$.  Let $e_2 = e_5\times e_1$.  Then, set $e_3$ to be a unit vector tangent to $\Sm$ and orthogonal to $\{e_1,e_2\}$.  Finally, let $e_4 = e_3\times e_5$, $e_6 = e_1\times e_3$ and $e_7 = e_3\times e_2$.

\paragraph{\it \;\, Normal bundle and second fundamental form}
The normal bundle of $\Sm$ is isomorphic to the bundle of anti-self-dual $2$-forms of $\Sm$ via the following map:
\begin{align}
V &\mapsto (V\lrcorner\,\vph)|_\Sm ~.
\label{coa0} \end{align}
In terms of the above frame, $e_5$ corresponds to $\om^{12}-\om^{34}$, $e_6$ corresponds to $\om^{13}+\om^{24}$, and $e_7$ corresponds to $\om^{14}-\om^{23}$.

As shown in \cite{ref_HL}, a coassociative submanifold must be minimal.  In fact, its second fundamental form has certain symmetry.  For instance,
\begin{align*}
\hm{5}{1}{i} &= \ip{\namb_{e_i}e_1}{e_5} = -\ip{e_1}{\namb_{e_i}(e_6\times e_7)} \\
&= -\ip{e_1}{(\namb_{e_i}e_6)\times e_7} - \ip{e_1}{e_6\times(\namb_{e_i}e_7)} \\
&= -\ip{e_4}{\namb_{e_i}e_6} + \ip{e_3}{\namb_{e_i}e_7} = \hm{6}{4}{i} - \hm{7}{3}{i} ~.
\end{align*}
What follows are all the relations:
\begin{align}  \begin{split}
\hm{5}{2}{i} + \hm{6}{3}{i} + \hm{7}{4}{i} &= 0 ~,  \\
\hm{5}{1}{i} - \hm{6}{4}{i} + \hm{7}{3}{i} &= 0 ~,
\end{split}  \begin{split}
\hm{5}{4}{i} + \hm{6}{1}{i} - \hm{7}{2}{i} &= 0 ~,  \\
-\hm{5}{3}{i} + \hm{6}{2}{i} + \hm{7}{1}{i} &= 0
\end{split}  \label{coa1}  \end{align}
for any $i\in\{1,2,3,4\}$.  These relations imply that the mean curvature vanishes.  They can be encapsulated as $\sum_j e_j\times\II(e_i,e_j) = 0$.

\subsubsection{Strong stability for coassociatives}
For any sections of $N\Sm$, $\bv$, denote the symmetric bilinear form on the left hand side of \eqref{sstable} by $Q(\bv,\bv)$.  Under the identification \eqref{coa0}, $\td{Q}(\bv,\bv) = -2\,\bv^T\,W_-\,\bv + \frac{s}{3}|\bv|^2$ is also a symmetric bilinear form.

We now check that $Q(\bv,\bv) = \td{Q}(\bv,\bv)$ for any unit vector $\bv\in N_p\Sm$ at any $p\in\Sm$.  As explained above, we may take $e_5 = \bv$ and construct the other orthonormal vectors.  With respect such a frame, it follows from \eqref{ASD_curv} and \eqref{4d_scurv} that
\begin{align*}
\td{Q}(\bv,\bv) &= \RM{1}{3}{1}{3}^\Sm + \RM{2}{4}{2}{4}^\Sm + \RM{1}{3}{1}{3}^\Sm + \RM{2}{4}{2}{4}^\Sm + 2\RM{1}{2}{3}{4}^\Sm ~.
\end{align*}
The quantity $Q(\bv,\bv)$ can be rewritten as follows.
\begin{align*}
Q(\bv,\bv) &= -\sum_{i}\RM{i}{5}{i}{5} - \sum_{i,j}(\hm{5}{i}{j})^2 \\
&= \RM{6}{5}{6}{5} + \RM{7}{5}{7}{5} - \sum_{i,j}(\hm{5}{i}{j})^2 \\
&= \RM{1}{3}{1}{3} + \RM{2}{4}{2}{4} + \RM{1}{3}{1}{3} + \RM{2}{4}{2}{4} - 2\RM{1}{4}{2}{3} + 2\RM{1}{3}{2}{4} - \sum_{i,j}(\hm{5}{i}{j})^2 \\
&= \RM{1}{3}{1}{3} + \RM{2}{4}{2}{4} + \RM{1}{3}{1}{3} + \RM{2}{4}{2}{4} +2\RM{1}{2}{3}{4} - \sum_{i,j}(\hm{5}{i}{j})^2 ~.
\end{align*}
The second equality follows from Ricci flatness.  The third equality uses \eqref{G2_curv}.  The last equality is the first Bianchi identity.  With the Gauss equation and some simple manipulation, 
\begin{align}
& Q(\bv,\bv) - \td{Q}(\bv,\bv) \notag \\
=\,& \sum_{\af}\left( (\hm{\af}{1}{4} + \hm{\af}{2}{3})^2 + (\hm{\af}{1}{3} - \hm{\af}{2}{4})^2 - (\hm{\af}{1}{1} + \hm{\af}{2}{2})(\hm{\af}{3}{3} + \hm{\af}{4}{4}) \right) - \sum_{i,j}(\hm{5}{i}{j})^2 ~. \label{coa2}
\end{align}

By appealing to \eqref{coa1},
\begin{align*}  \begin{split}
\hm{6}{1}{4} + \hm{6}{2}{3} &= - \hm{5}{2}{2} - \hm{5}{4}{4} = \hm{5}{1}{1} + \hm{5}{3}{3} ~, \\
\hm{7}{1}{4} + \hm{7}{2}{3} &= - \hm{5}{1}{2} + \hm{5}{3}{4} ~,
\end{split}  \begin{split}
\hm{6}{1}{3} - \hm{6}{2}{4} &= - \hm{5}{1}{2} - \hm{5}{3}{4} ~, \\
\hm{7}{1}{3} - \hm{7}{2}{4} &= - \hm{5}{2}{2} + \hm{5}{3}{3} = - \hm{5}{1}{1} - \hm{5}{4}{4} ~,
\end{split}  \end{align*}
and
\begin{align*}
\hm{6}{1}{1} + \hm{6}{2}{2} = - \hm{6}{3}{3} - \hm{6}{4}{4} &= - \hm{5}{1}{4} + \hm{5}{2}{3} ~, \\
\hm{7}{1}{1} + \hm{7}{2}{2} = - \hm{7}{3}{3} - \hm{7}{4}{4} &= \hm{5}{1}{3} + \hm{5}{2}{4} ~.
\end{align*}
By using these relations, it is not hard to verify that \eqref{coa2} vanishes.  Therefore, the strong stability condition \eqref{sstable} is equivalent to the positivity of $-2W_- + \frac{s}{3}$.

As a final remark, this equivalence can also be seen by combining \cite[Theorem 4.9]{ref_McLean} and the Weitzenb\"ock formula \cite[Appendix C]{ref_FUl}.  Nevertheless it is nice to derive the equivalence directly by highlighting the geometry of $G_2$.

\section{Evolution equation for tensors} \label{apx_evol}

Suppose that $\Psi$ be a tensor defined on $M$ of type $(0,3)$.  The main purpose of this section is to calculate its evolution equation along the mean curvature flow.  Since there will be some different connections, we denote the Levi-Civita connection of $(M,g)$ by $\bn$ to avoid confusions.

Let $\Gm_t$ be the mean curvature flow at time $t$.  The tensor $\Psi$ is a section of $(T^*M\ot T^*M\ot T^*M)|_{\Gm^t}$.  The connection $\bn$ naturally induces a connection $\nres$ on this bundle. The only difference between $\bn$ and $\nres$ is that the direction vector in $\nres$ must be tangent to $\Gm_t$.

\begin{center}\begin{tabular}{ c | l}
connection & bundle and base \\ \hline
$\bn$ & Levi-Civita connection of $(M,g)$ \\
$\nt$ & Levi-Civita connection of $\Gm_t$ with the induced metric \\
$\nnor$ & connection of the normal bundle of $\Gm_t$ \\
$\nres$ & connection of $(T^*M\ot T^*M\ot T^*M)|_{\Gm^t}$ \\
$\namb$ & connection of $T^*\Gm_t\ot T^*\Gm_t\ot N^*\Gm_t$ defined by \eqref{Codazzi2}
\end{tabular}\end{center}
From the construction, $\namb$ is the composition of $\nres$ with the orthogonal projection.


\begin{prop}
Let $\Psi$ be a tensor of type $(0,3)$ defined on the ambient manifold $M$.  Along the mean curvature flow $\Gm_t$ in $M$,
\begin{align}
\frac{\dd}{\dd t}|\IIt-\Psi|^2 \leq \Dt|\IIt - \Psi|^2 - |\nres(\IIt - \Psi)|^2 + c(|\IIt - \Psi|^4 + |\IIt - \Psi|^2 + 1)
\end{align}
where $c>0$ is determined by the Riemann curvature tensor of $M$ and the sup-norm of $\Psi$, $\bn\Psi$, $\bn^2\Psi$.
\end{prop}

\begin{proof}

The mean curvature flow can be regarded as a map from $\Gm_0\times[0,\epsilon)\to M$.  For any $p\in\Gm_0$ and $t_0\in[0,\epsilon)$, choose a geodesic coordinate for $\Gm_0$ at $p$: $\{\td{x}^1,\cdots,\td{x}^n\}$.  We also choose a local orthonormal frame $\{\te_{\af}\}$ for $N\Gm_t$.  The following computations on derivatives are \emph{always evaluated at the point $(p,t_0)$}.

Let $H = \td{h}_\af \td{e}_{\af}$ be the mean curvature vector of $\Gm_t$.  The components of the second fundamental form and its covariant derivative are denoted by
\begin{align*}
\htd{\af}{i}{j} &= \ip{\bn_{\pt{i}}\pt{j}}{\te_{\af}} = \II(\pt{i},\pt{j},\te_{\af}) ~, \\
\htdf{\af}{i}{j}{k} &= (\namb_{\pt{k}}\II)(\pt{i},\pt{j},\te_{\af}) ~, \\
\td{h}_{\af,k} &= \ip{\nnor_{\pt{k}} H}{\te_{\af}} ~.
\end{align*}
At $(p,t_0)$, $\td{h}_\af = \htd{\af}{k}{k}$ and $\td{h}_{\af,i} = \htdf{\af}{k}{k}{i}$.

Note that on $\Gm_0\times[0,\epsilon)$, $H$ is $\pl_t$, and thus commutes with $\pt{i}$.  It follows that the evolution of the metric is:
\begin{align}
\frac{\dd}{\dd t} \td{g}_{ij} &= H\ip{\pt{i}}{\pt{j}} = \ip{\bn_{H}\pt{i}}{\pt{j}} + \ip{\pt{i}}{\bn_H\pt{j}} \notag \\
&= -\ip{H}{\bn_{\pt{i}}\pt{j}} - \ip{\bn_{\pt{j}}\pt{i}}{H} = -2\td{h}_\af \htd{\af}{i}{j} ~,  \label{a_dg} \\
\frac{\dd}{\dd t}\td{g}^{ij} &= 2\td{h}_\af\htd{\af}{i}{j} ~.  \label{a_dgi}
\end{align}
The covariant derivative of $\pt{i}$ and $\te_\af$ along $H$ can be expressed as follows:
\begin{align}
\bn_H\pt{i} &= \ip{\bn_H\pt{i}}{\pt{j}}\pt{j} + \ip{\bn_H\pt{i}}{\te_\af}\te_\af \notag \\
&= - \td{h}_\af \htd{\af}{i}{j}\,\pt{j} + \td{h}_{\af,i}\,\te_\af ~,  \label{a_xinH}\\
\bn_H\te_\af &= \ip{\bn_H\te_\af}{\pt{i}}\pt{i} + \ip{\bn_H\te_\af}{\te_\bt}\te_\bt \notag \\
&= -\td{h}_{\af,i}\,\pt{i} + \ip{\bn_H\te_\af}{\te_\bt}\te_\bt ~.  \label{a_einH}
\end{align}

The last part of the preparation is to relate the covariant derivative of $\Psi$ in $H$ to its Bochner--Laplacian in the ambient manifold $M$.
\begin{align}
\bn_H\Psi &= \bn_{(\bn_{\pt{j}}\pt{j})^\perp}\Psi = - \bn_{\nsub_{\pt{j}}\pt{j}}\Psi + \bn_{\bn_{\pt{j}}\pt{j}}\Psi \notag \\
&= \left( \nres_{\pt{j}}\nres_{\pt{j}}\Psi - \nres_{\nsub_{\pt{j}}\pt{j}}\Psi \right) + \left( -\bn_{\pt{j}}\bn_{\pt{j}}\Psi + \bn_{\bn_{\pt{j}}\pt{j}}\Psi \right) \notag \\
&= - \nres^*\nres\Psi + \tr_{\Gm_t}(\bn^2\Psi) ~. \label{a_PsiH}
\end{align}
Indeed, $\nsub_{\pt{j}}\pt{j}$ is zero at $(p,t_0)$.  The tensor $\bn^*\bn\Psi$ is defined in the ambient space, and has nothing to do with the submanifold $\Gm_t$.  It follows from \eqref{a_PsiH} that the evolution of $|\Psi|^2$ is
\begin{align}
\frac{\dd}{\dd t}|\Psi|^2 &= H\big(\ip{\Psi}{\Psi}\big) = 2\ip{\bn_H\Psi}{\Psi} \notag \\
&= -2\ip{\nres^*\nres\Psi}{\Psi} + 2\ip{\tr_{\Gm_t}(\bn^2\Psi)}{\Psi} \notag \\
&= \Dt |\Psi|^2 - 2|\nres\Psi|^2 + 2\ip{\tr_{\Gm_t}(\bn^2\Psi)}{\Psi}  \label{a_dPsi}
\end{align}

The next task is to calculate the evolution equation for $\ip{\IIt}{\Psi} = \td{g}^{ik} \td{g}^{jl} \htd{\af}{k}{l}\td{\Psi}_{\af ij}$ where $\td{\Psi}_{\af ij} = \Psi(\pt{i},\pt{j},\te_\af)$.  According to \eqref{a_xinH} and \eqref{a_einH},
\begin{align}
\frac{\dd}{\dd t}\td{\Psi}_{\af ij} &= H\big(\Psi(\pt{i},\pt{j},\te_\af)\big) \notag \\
&= (\bn_H\Psi)(\pt{i},\pt{j},\te_\af) + \Psi(\pt{i},\pt{j},\bn_H\te_\af) + \Psi(\bn_H\pt{i},\pt{j},\te_\af) + \Psi(\pt{i},\bn_H\pt{j},\te_\af) \notag \\
\begin{split} &= (\bn_H\Psi)(\pt{i},\pt{j},\te_\af) - \td{h}_{\af,k}\td{\Psi}_{kij} + \ip{\bn_H\te_\af}{\te_\bt}\td{\Psi}_{\bt ij} \\
&\quad - \td{h}_\gm \htd{\gm}{i}{k}\td{\Psi}_{\af kj} + \td{h}_{\gm,i}\td{\Psi}_{\af\gm j} - \td{h}_\gm \htd{\gm}{j}{k}\td{\Psi}_{\af ik} +　\td{h}_{\gm,j}\td{\Psi}_{\af i\gm}
\end{split}  \label{a_dPsi1}  \end{align}

The difference between $\nres^*\nres\IIt$ and  $\namb^*\namb\IIt$ is:
\begin{align}
\nres^*\nres\IIt - \namb^*\namb\IIt &= \nres_{\pt{k}}\nres_{\pt{k}}\IIt - \namb_{\pt{k}}\namb_{\pt{k}}\IIt
\label{a_dPsi2} \end{align}
Since
\begin{align*}
(\nres_{\pt{k}}\II^t)(\cdot,\cdot,\cdot) &= (\bn_{\pt{k}}\II^t)(\cdot,\cdot,\cdot) \\
&= \pt{k}\big( \IIt(\cdot,\cdot,\cdot) \big) - \IIt\big((\bn_{\pt{k}}\cdot)^{T},\cdot,\cdot\big) - \IIt\big(\cdot,(\bn_{\pt{k}}\cdot)^{T},\cdot\big) - \IIt\big(\cdot,\cdot,(\bn_{\pt{k}}\cdot)^{\perp}\big) ~,
\end{align*}
the tensor $\nres_{\pt{k}}\IIt$ has only the following components:
\begin{align} \begin{split}
(\nres_{\pt{k}}\IIt)(\pt{i},\pt{j},\te_\af) &= (\namb_{\pt{k}}\IIt)(\pt{i},\pt{j},\te_\af) = \htdf{\af}{i}{j}{k} ~, \\
(\nres_{\pt{k}}\IIt)(\pt{i},\pt{j},\pt{l}) &= -\htd{\af}{i}{j} \htd{\af}{k}{l} ~, \\
(\nres_{\pt{k}}\IIt)(\te_\bt,\pt{j},\te_\af) &= \htd{\bt}{k}{i} \htd{\af}{i}{j} ~, \\
(\nres_{\pt{k}}\IIt)(\pt{i},\te_\bt,\te_\af) &= \htd{\bt}{k}{j} \htd{\af}{i}{j} ~.
\end{split} \label{a_dII} \end{align}
The above four equations hold everywhere, but not only at $(p,t_0)$.  It follows that
\begin{align*}
(\nres_{\pt{k}}\nres_{\pt{k}}\IIt)(\pt{i},\pt{j},\te_\af) &= \pt{k} \left( (\nres_{\pt{k}}\IIt)(\pt{i},\pt{j},\te_\af) \right) - (\nres_{\pt{k}}\IIt)\big( \pt{i},\pt{j},\bn_{\pt{k}}\te_\af \big) \\
&\quad - (\nres_{\pt{k}}\IIt)\big( \bn_{\pt{k}}\pt{i},\pt{j},\te_\af \big) - (\nres_{\pt{k}}\IIt)\big( \pt{i},\bn_{\pt{k}}\pt{j},\te_\af \big) \\
&= (\namb_{\pt{k}}\namb_{\pt{k}}\IIt)(\pt{i},\pt{j},\te_\af) - (\nres_{\pt{k}}\IIt)\big( \pt{i},\pt{j},(\bn_{\pt{k}}\te_\af)^T \big) \\
&\quad - (\nres_{\pt{k}}\IIt)\big( (\bn_{\pt{k}}\pt{i})^\perp,\pt{j},\te_\af \big) - (\nres_{\pt{k}}\IIt)\big( \pt{i},(\bn_{\pt{k}}\pt{j})^\perp,\te_\af \big) \\
&= (\namb_{\pt{k}}\namb_{\pt{k}}\IIt)(\pt{i},\pt{j},\te_\af) - \htd{\bt}{i}{j} \htd{\bt}{k}{l} \htd{\af}{k}{l} - \htd{\bt}{k}{l} \htd{\af}{l}{j} \htd{\bt}{k}{i} - \htd{\bt}{k}{j} \htd{\bt}{k}{l} \htd{\af}{i}{l} ~.
\end{align*}
Use \eqref{a_dPsi2} to rewrite the above computation as
\begin{align}
(\nres^*\nres\IIt - \namb^*\namb\IIt)(\pt{i},\pt{j},\te_\af) &= \htd{\bt}{i}{j} \htd{\bt}{k}{l} \htd{\af}{k}{l} + \htd{\bt}{k}{l} \htd{\af}{l}{j} \htd{\bt}{k}{i} + \htd{\bt}{k}{j} \htd{\bt}{k}{l} \htd{\af}{i}{l}
\label{a_dPsi21}\end{align}

The tensor $\namb^*\namb\IIt$ does not have other components.  However, $\nres^*\nres\IIt$ does.
\begin{align}
(\nres^*\nres\IIt)(\pt{i},\pt{j},\pt{l}) &= -(\nres_{\pt{k}}\nres_{\pt{k}}\IIt)(\pt{i},\pt{j},\pt{l}) \notag \\
&= - \pt{k}\big((\nres_{\pt{k}}\IIt)(\pt{i},\pt{j},\pt{l})\big) + (\nres_{\pt{k}}\IIt)(\bn_{\pt{k}}\pt{i},\pt{j},\pt{l}) \notag \\
&\quad + (\nres_{\pt{k}}\IIt)(\pt{i},\bn_{\pt{k}}\pt{j},\pt{l}) + (\nres_{\pt{k}}\IIt)(\pt{i},\pt{j},\bn_{\pt{k}}\pt{l}) \notag \\
&= \pt{k}(\htd{\af}{i}{j} \htd{\af}{k}{l}) + \htdf{\af}{i}{j}{k} \htd{\af}{k}{l}  \notag \\
&= 2\htdf{\af}{i}{j}{k} \htd{\af}{k}{l} + \htd{\af}{i}{j} \htdf{\af}{k}{l}{k} \notag \\
&= 2\htdf{\af}{i}{j}{k} \htd{\af}{k}{l} + \htd{\af}{i}{j} \td{h}_{\af,l} + \htd{\af}{i}{j}\RTD{\af}{k}{k}{l} ~.  \label{a_dPsi22}
\end{align}
The second last equality uses the fact that $0 = \ip{\nnor_{\pt{k}}\te_\af}{\te_\bt}\htd{\bt}{i}{j}\htd{\af}{k}{l} - \ip{\nnor_{\pt{k}}\te_\af}{\te_\bt}\htd{\af}{i}{j}\htd{\bt}{k}{l}$.  The last equality uses the Codazzi equation \eqref{Codazzi1}.  Similarly,
\begin{align} \begin{split}
(\nres^*\nres\IIt)(\te_\bt,\pt{j},\te_\af) &= - 2\htdf{\af}{l}{j}{k} \htd{\bt}{k}{l} - \htdf{\bt}{k}{l}{k} \htd{\af}{j}{l} \\
&= - 2\htdf{\af}{l}{j}{k} \htd{\bt}{k}{l} - \td{h}_{\bt,l} \htd{\af}{j}{l} - \RTD{\bt}{k}{k}{l} \htd{\af}{j}{l} \\
(\nres^*\nres\IIt)(\pt{i},\te_\bt,\te_\af) &= - 2\htdf{\af}{i}{l}{k} \htd{\bt}{k}{l} - \htdf{\bt}{k}{l}{k} \htd{\af}{i}{l} \\
&= - 2\htdf{\af}{i}{l}{k} \htd{\bt}{k}{l} - \td{h}_{\bt,l} \htd{\af}{i}{l} - \RTD{\bt}{k}{k}{l} \htd{\af}{i}{l} \\
(\nres^*\nres\IIt)(\pt{i},\te_\af,\pt{j})  &= - 2\htd{\af}{k}{l} \htd{\bt}{i}{l} \htd{\bt}{k}{j} \\
(\nres^*\nres\IIt)(\te_\bt,\te_\gm,\te_\af)  &= - 2\htd{\bt}{k}{l} \htd{\gm}{k}{j} \htd{\af}{l}{j} \\
(\nres^*\nres\IIt)(\te_\bt,\pt{j},\pt{i}) &= 2\htd{\bt}{k}{l} \htd{\af}{l}{j} \htd{\af}{k}{i} \\
(\nres^*\nres\IIt)(\te_\bt,\te_\af,\pt{j}) &= 0
\end{split}  \label{a_dPsi23}  \end{align}


The evolution equation for $\htd{\af}{i}{j}$ was derived in \cite[Proposition 7.1]{ref_W1}.  With \eqref{a_dPsi1} and \eqref{a_PsiH}, we have
\begin{align*}
\frac{\dd}{\dd t}\ip{\IIt}{\Psi} &= \frac{\dd}{\dd t}\left(\td{g}^{ik} \td{g}^{jl} \htd{\af}{k}{l}\td{\Psi}_{\af ij}\right) \\
&= 2\htd{\bt}{i}{k}\td{h}_\bt\htd{\af}{k}{j}\td{\Psi}_{\af i j} + 2\htd{\bt}{j}{l}\td{h}_\bt\htd{\af}{i}{l}\td{\Psi}_{\af i j} + \left(\frac{\dd}{\dd t}\htd{\af}{i}{j}\right) \td{\Psi}_{\af ij} + \htd{\af}{i}{j}\left(\frac{\dd}{\dd t}\td{\Psi}_{\af ij}\right) \\
&= 2\htd{\bt}{i}{k}\td{h}_\bt\htd{\af}{k}{j}\td{\Psi}_{\af i j} + 2\htd{\bt}{j}{l}\td{h}_\bt\htd{\af}{i}{l}\td{\Psi}_{\af i j} - \ip{\namb^*\namb\IIt}{\Psi} \\
&\quad + (\bn_{\te_k} R)_{\td{\af}\td{i}\td{j}\td{k}}\td{\Psi}_{\af i j} + (\bn_{\te_j} R)_{\td{\af}\td{k}\td{i}\td{k}}\td{\Psi}_{\af i j} - 2\RTD{l}{i}{j}{k} \htd{\af}{l}{k} \td{\Psi}_{\af i j} + 2\RTD{\af}{\bt}{j}{k} \htd{\bt}{i}{k} \td{\Psi}_{\af i j} \\
&\quad + 2\RTD{\af}{\bt}{i}{k} \htd{\bt}{j}{k} \td{\Psi}_{\af i j}  - \RTD{l}{k}{i}{k} \htd{\af}{l}{j} \td{\Psi}_{\af i j} - \RTD{l}{k}{j}{k} \htd{\af}{l}{i} \td{\Psi}_{\af i j} + \RTD{\af}{k}{\bt}{k} \htd{\bt}{i}{j} \td{\Psi}_{\af i j} \\
&\quad - \htd{\af}{i}{l} \left( \htd{\bt}{l}{j}\td{h}_\bt - \htd{\bt}{l}{k}\htd{\bt}{j}{k} \right) \td{\Psi}_{\af i j} - \htd{\af}{l}{k} \left( \htd{\bt}{l}{j}\htd{\bt}{i}{k} - \htd{\bt}{l}{k}\htd{\bt}{i}{j} \right) \td{\Psi}_{\af i j} \\
&\quad - \htd{\bt}{i}{k} \left( \htd{\bt}{l}{j}\htd{\af}{l}{k} - \htd{\bt}{l}{k}\htd{\af}{l}{j} \right) \td{\Psi}_{\af i j} - \htd{\af}{j}{k}\htd{\bt}{i}{k}\td{h}_\bt \td{\Psi}_{\af i j} + \htd{\bt}{i}{j} \ip{\te_\bt}{\bn_{H}\te_\af} \td{\Psi}_{\af i j} \\
&\quad - \ip{\IIt}{\nres^*\nres\Psi} + \ip{\IIt}{\tr_{\Gm_t}(\bn^*\bn\Psi)} - \htd{\af}{i}{j}\td{h}_{\af,k}\td{\Psi}_{kij} + \htd{\af}{i}{j}\ip{\bn_H\te_\af}{\te_\bt}\td{\Psi}_{\bt ij} \\
&\quad - \htd{\af}{i}{j}\td{h}_\gm \htd{\gm}{i}{k}\td{\Psi}_{\af kj} + \htd{\af}{i}{j}\td{h}_{\gm,i}\td{\Psi}_{\af\gm j} - \htd{\af}{i}{j}\td{h}_\gm \htd{\gm}{j}{k}\td{\Psi}_{\af ik} +　\htd{\af}{i}{j}\td{h}_{\gm,j}\td{\Psi}_{\af i\gm} \\
&= - \ip{\nres^*\nres \IIt}{\Psi} - \ip{\IIt}{\nres^*\nres \Psi} + \ip{\IIt}{\tr_{\Gm_t}(\bn^2\Psi)} \\
&\quad + (\bn_{\te_k} R)_{\td{\af}\td{i}\td{j}\td{k}}\td{\Psi}_{\af i j} + (\bn_{\te_j} R)_{\td{\af}\td{k}\td{i}\td{k}}\td{\Psi}_{\af i j} - 2\RTD{l}{i}{j}{k} \htd{\af}{l}{k} \td{\Psi}_{\af i j} + 2\RTD{\af}{\bt}{j}{k} \htd{\bt}{i}{k} \td{\Psi}_{\af i j} \\
&\quad + 2\RTD{\af}{\bt}{i}{k} \htd{\bt}{j}{k} \td{\Psi}_{\af i j}  - \RTD{l}{k}{i}{k} \htd{\af}{l}{j} \td{\Psi}_{\af i j} - \RTD{l}{k}{j}{k} \htd{\af}{l}{i} \td{\Psi}_{\af i j} + \RTD{\af}{k}{\bt}{k} \htd{\bt}{i}{j} \td{\Psi}_{\af i j} \\
&\quad + \RTD{\af}{k}{k}{l}\htd{\af}{i}{j}\td{\Psi}_{lij} - \RTD{\bt}{k}{k}{l}\htd{\af}{j}{k}\td{\Psi}_{\af\bt j} - \RTD{\bt}{k}{k}{l}\htd{\af}{i}{l}\td{\Psi}_{\af i \bt} \\
&\quad - 2\htd{\af}{i}{l} \left( \htd{\bt}{l}{j}\td{h}_\bt - \htd{\bt}{l}{k}\htd{\bt}{j}{k} \right) \td{\Psi}_{\af i j} - 2 \htd{\bt}{i}{k} \left( \htd{\bt}{l}{j}\htd{\af}{l}{k} - \htd{\bt}{l}{k}\htd{\af}{l}{j} \right) \td{\Psi}_{\af i j} \\
&\quad - 2\htd{\af}{j}{k}\htd{\bt}{i}{k}\td{h}_\bt \td{\Psi}_{\af i j} - 2\htd{\af}{k}{l}\htd{\bt}{i}{l}\htd{\bt}{k}{j}\td{\Psi}_{ji\af} - 2\htd{\bt}{k}{l}\htd{\gm}{k}{j}\htd{\af}{l}{j}\td{\Psi}_{\af\bt\gm} + 2\htd{\bt}{k}{l}\htd{\af}{l}{j}\htd{\af}{k}{i}\td{\Psi}_{i\bt j} \\
&\quad + 2\htdf{\af}{i}{j}{k}\htd{\af}{k}{l}\td{\Psi}_{lij} - 2\htdf{\af}{l}{j}{k}\htd{\bt}{k}{l}\td{\Psi}_{\af \bt j} - \htdf{\af}{i}{l}{k}\htd{\bt}{k}{l}\td{\Psi}_{\af i\bt} ~.
\end{align*}
The last equality uses \eqref{a_dPsi21}, \eqref{a_dPsi22} and \eqref{a_dPsi23} to replace $\namb^*\namb\IIt$ by $\nres^*\nres\IIt$.

By the Cauchy--Schwarz inequality,
\begin{align*}
\left| \frac{\dd}{\dd t}\ip{\IIt}{\Psi} - \Dt\ip{\IIt}{\Psi} + 2\ip{\nres\IIt}{\nres\Psi} \right| &\leq \oh|\namb\IIt|^2 + c(|\IIt|^4 + |\IIt|^2 + 1) ~.
\end{align*}
This together with \eqref{evol3} and \eqref{a_dPsi} imply that
\begin{align*}
\frac{\dd}{\dd t}|\IIt - \Psi|^2 &\leq \Dt|\IIt - \Psi|^2 - 2|\nres(\IIt - \Psi)|^2 + |\namb\IIt|^2 + c'(|\IIt|^4 + |\IIt|^2 + 1) ~.
\end{align*}
According to \eqref{a_dII},
\begin{align*}
|\nres(\IIt - \Psi)|^2 \geq |\nres\II^2| - |\nres\Psi|^2 \geq |\namb\IIt|^2 + c''|\IIt|^4 - |\bn\Psi|^2 ~.
\end{align*}
Hence,
\begin{align*}
\frac{\dd}{\dd t}|\IIt - \Psi|^2 &\leq \Dt|\IIt - \Psi|^2 - |\nres(\IIt - \Psi)|^2 + c'''(|\IIt|^4 + |\IIt|^2 + 1)
\end{align*}
By the triangle inequality $|\IIt|^2\leq |\IIt-\Psi|^2 + |\Psi|^2$, it finishes the proof of the proposition.
\end{proof}

\section{Moser iteration for $\CC^2$ convergence} \label{apx_Moser}

The main purpose of this appendix is to prove the $\CC^2$ convergence part of Theorem \ref{longtime}, in particular $|\IIt - \II^\Sm|^2\rightarrow 0$.  We already show that
\begin{itemize}
\item The mean curvature flow $\{\Gm_t\}$ exists for all time, and the second fundamental form $\IIt$ is uniformly bounded.
\item $\Gm_t$ converges to $\Sm$ in $\CC^1$.
\item The $L^2$ convergence, \eqref{int}, of $\IIt$: $\lim_{t\to 0} \int_{\Gm_t} |\IIt - \II^\Sm|^2 \,\dd\mu_t =0$.
\end{itemize}

When $\IIt$ is uniformly bounded, it is known that all the higher order derivatives of $\IIt$ remains uniformly bounded.  This can be proved by using the evolution equation of $\nabla^{(k)}\IIt$.  See \cite[Proposition 4.8]{ref_Baker}.

With the $\CC^1$ convergence,  for $t$ large enough $\Gm_t$ can be written as the graph of a section of the normal bundle, as in section \ref{sec_nonpara}.  Taking second order derivatives (in space) gives the evolution equation of $\IIt-\II^\Sm$ in the non-parametric form.  The strategy is to apply the Moser iteration argument \cite{ref_Moser,ref_Trudinger} to estimate its sup-norm in terms of the $L^2$ norm.  It together with \eqref{int} would lead to the $\CC^2$ convergence.

The argument will be done over open balls of $\Sm$, and will be demonstrated on the ball of radius $1$.

\subsection{Second fundamental form and second order derivative}
Denote by $\Gm_{**}^*$ the Christoffel symbols of the ambient metric \eqref{graph1}.
It follows from \eqref{graph1} that there are constants $\vep$ and $C$ which have the following significances.
\begin{enumerate}
\item For any section $\by(\bx)$ with $|\by|_{\CC^0}\leq\vep$,
\begin{align*}
\left| \Gm_{ij}^k - \ul{\Gm}_{ij}^k \right| + \left| \Gm_{i\mu}^k + \ul{g}^{kj}h_{\mu ij} \right| + \left| \Gm_{\mu\nu}^k \right| + \left| \Gm_{ij}^\gm - h_{\gm ij} \right| + \left| \Gm_{i\mu}^\gm - A_{\mu i}^\gm \right| + \left| \Gm_{\mu\nu}^\gm \right| &\leq C|\by|_{\CC^0} ~.
\end{align*}
Here, we denote the induced metric on $\Sm$ by $\ul{g}_{ij}$ and its Christoffel symbols by $\ul{\Gm}_{ij}^k$ to avoid confusion.  Underlined geometric quantities depend on $\bx$ only.

\item For any section $\by(\bx)$ with $|\by|_{\CC^1}\leq\vep$, the orthogonal projection of the ambient coordinate vector field $\frac{\pl}{\pl y^\mu}$ to the normal of the section is surjective.  Moreover,
\begin{align*}
\left|\left(\frac{\pl}{\pl x^i}\right)^\perp\right| + \left| \left(\frac{\pl}{\pl y^\mu}\right)^\perp - \frac{\pl}{\pl y^\mu} \right| &\leq C|\by|_{\CC^1} ~.
\end{align*}
\end{enumerate}
Assume $\by = \by(\bx)$ has small $\CC^1$ norm.  Denote its graph, $\{(\bx,\by(\bx))\}$, by $\Gm$.  The tangent space of $\Gm$ is spanned by
\begin{align*}
F_*(\frac{\pl}{\pl x^i}) &= \frac{\pl}{\pl x^i} + \pl_i y^\mu\,\frac{\pl}{\pl y^\mu} ~.
\end{align*}
We compute
\begin{align*}
\nabla_{F_*(\frac{\pl}{\pl x^i})}F_*(\frac{\pl}{\pl x^j}) &= (\pl_i\pl_j y^\mu)\,\frac{\pl}{\pl y^\mu} + \Gm_{ij}^k\frac{\pl}{\pl x^k} + \Gm_{ij}^\mu\frac{\pl}{\pl y^\mu} + \text{terms with coefficients }(\pl_k y) ~,
\end{align*}
and thus
\begin{align*}
\left| \left( \nabla_{F_*(\frac{\pl}{\pl x^i})}F_*(\frac{\pl}{\pl x^j}) \right)^\perp - \left( \pl_i\pl_j y^\mu + h_{\mu ij} \right)\frac{\pl}{\pl y^\mu} \right| &\leq C|\by|_{\CC^1} ~.
\end{align*}
It follows that
\begin{align*}
\left| \II^\Gm - \II^\Sm - \pl^2 \by \right| &\leq C|\by|_{\CC^1} ~.
\end{align*}
Now, suppose that $\Gm_t = \{\by = \by(\bx,t)\}$ is a mean curvature flow which converges to $0$ uniformly in $\CC^1$.  To prove that $\II^{\Gm_t} - \II^\Sm$ converges to zero uniformly, it suffices to show that $\pl^2\by$ converges to zero uniformly.

\subsection{The mean curvature flow equation}
As in section \ref{sec_nonpara}, we introduce the dummy variable $p^\mu_l = \pl_k y^\mu$.  In the following discussion, a function $\CF(\bx,\by,\bp)$ is said to be $\CO(k)$ for some $k\in\BN\cup\{0\}$ if there exist constants $C_{\ell}$ and ${C}_{\ell,\ell'}$ such that
\begin{align} \begin{split}
|\dt^\ell_\bx \CF| &\leq C_\ell\,(|\by|^2+|\bp|^2)^{\frac{k}{2}} ~, \\
|\dt^\ell_\by\dt^{\ell'}_\bp \CF| &\leq {C}_{\ell,\ell'}\,(|\by|^2+|\bp|^2)^{\frac{k-\ell-\ell'}{2}} \quad\text{for any }\ell+\ell'\leq k
\end{split} \label{orderk} \end{align}
(for any $\bx$ in the region of consideration).  The derivative in the variables $(\bx,\by,\bp)$ is denoted by $\dt$ to avoid confusion.

Recall \eqref{Lagf}:
\begin{align}
\sqrt{g} &= \sqrt{\ul{g}} \left( 1 + \oh\left( \ul{g}^{ij}(p^\mu_i + A^\mu_{\nu i}\,y^\nu)(p^\mu_j + A^\mu_{\gm j}\,y^\gm) - \ul{g}^{ij}\,y^\mu\,y^\nu(R_{j\mu i\nu} + \ul{g}^{k\ell}\,h_{\mu ik}\,h_{\nu j\ell}) \right) \right) + \CE(\bx,\by,\bp)
\label{vol_exp} \end{align}
where $\CE(\bx,\by,\bp)$ is $\CO(3)$ in the sense of \eqref{orderk}.

With \eqref{vol_exp}, the mean curvature flow equation \eqref{MCF_graph1} takes the following form:
\begin{align}
\frac{\pl y^\mu}{\pl t} &= (\dt^{\mu\nu} + \CO(2))\left( \frac{\dd}{\dd x^i}\left( \ul{g}^{ij}(\frac{\pl y^\mu}{\pl x^j} + A^\mu_{\nu j}\,y^\nu) + \CO(2) \right) - \CO(1) \right) \notag \\
&= \ul{g}^{ij}(\bx)\cdot\frac{\pl^2 y^\mu}{\pl x^i\pl x^j} + \CE^{\mu ij}_{\nu}(\bx,\by,\bp)\cdot\frac{\pl^2 y^\nu}{\pl x^i\pl x^j} + \CE^\mu(\bx,\by,\bp)
\label{MCF_exp} \end{align}
where both $\CE^{\mu ij}_{\nu}$ and $\CE^\mu$ are of $\CO(1)$.

\subsection{The evolution equation for the second order derivatives}

Denote $\pl_k\pl_\ell y^\mu$ by $q^\mu_{k\ell}$.  We are going to derive the evolution equation for $q^\mu_{k\ell}$.  Remember that the $\CC^1$-norm of $\by$ converges to $0$ as $t\to\infty$, and the $\CC^2$-norm of $\by$ is uniformly bounded.

For the first term on the right hand side of \eqref{MCF_exp},
\begin{align*}
\frac{\pl^2}{\pl x^k\pl x^\ell}\left(\ul{g}^{ij}\frac{\pl^2 y^\mu}{\pl x^i\pl x^j}\right) &= \frac{\pl}{\pl x^i}\left( \ul{g}^{ij}\frac{\pl q^\mu_{k\ell}}{\pl x^j} \right) + \left[ \frac{\pl\ul{g}^{ij}}{\pl x^k}\frac{\pl q^\mu_{ij}}{\pl x^\ell} + \frac{\pl\ul{g}^{ij}}{\pl x^\ell}\frac{\pl q^\mu_{ij}}{\pl x^k} - \frac{\pl\ul{g}^{ij}}{\pl x^i}\frac{\pl q^\mu_{k\ell}}{\pl x^j} \right] + \frac{\pl^2\ul{g}^{ij}}{\pl x^k\pl x^\ell}q^\mu_{ij} \\
&= \frac{\pl}{\pl x^i}\left( \ul{g}^{ij}\frac{\pl q^\mu_{k\ell}}{\pl x^j} \right) + \CO(0)\cdot\pl\bq + \CO(0)\cdot\bq
\end{align*}
For the second term on the right hand side of \eqref{MCF_exp}, one performs the commuting derivatives to find that
\begin{align*}
&\frac{\pl^2}{\pl x^k\pl x^\ell}\left( \CE^{\mu ij}_\nu(\bx,\by,\bp)\cdot\frac{\pl^2 y^\nu}{\pl x^i\pl x^j} \right) \\
=\,& \frac{\pl}{\pl x^i}\left( \CE^{\mu ij}_{\nu}(\bx,\by,\bp)\frac{\pl q^\nu_{k\ell}}{\pl x^j} \right) + \tilde{\CE}_1(\bx,\by,\bp,\bq)\cdot\pl\bq + \tilde{\CE}_0(\bx,\by,\bp,\bq)\cdot\bq
\end{align*}
for some smooth function $\td{\CE}_1$ and $\td{\CE}_0$ in $\bx,\by,\bp = \pl\by,\bq = \pl^2\by$.  Since the $\CC^2$ norm of $\by$ is uniformly bounded, the coefficient functions $\td{\CE}_1$ and $\td{\CE}_0$ are uniformly bounded.
For the last term on the right hand side of \eqref{MCF_exp},
\begin{align*}
\frac{\pl^2}{\pl x^k\pl x^\ell}\CE^\mu &= \left[ \frac{\dt^2\CE^\mu}{\dt x^k\dt x^\ell} + \frac{\dt^2\CE^\mu}{\dt x^\ell\dt y^\nu}\frac{\pl y^\nu}{\pl x^k} + \frac{\dt^2\CE^\mu}{\dt x^k\dt y^\nu}\frac{\pl y^\nu}{\pl x^\ell} + \frac{\dt\CE^\mu}{\dt y^\nu\dt y^\gm}\frac{\pl y^\gm}{\pl x^k}\frac{\pl y^\nu}{\pl x^\ell}\right] \\
&\quad + \hat{\CE}_1(\bx,\by,\bp)\cdot\pl\bq + \hat{\CE}_0(\bx,\by,\bp,\bq)\cdot\bq ~.
\end{align*}
Since the $\CC^1$ norm of $\by$ converges to zero uniformly and $\CE^\mu = \CO(1)$, the first line on the right hand side converges to zero uniformly as $t\to\infty$.  Similar to above, $\hat{\CE}_1$ and $\hat{\CE}_0$ are uniformly bounded.

To sum up, write $q^\mu_{k\ell}$ as $u^A$.  Its evolution equation takes the following form:
\begin{align}
\frac{\pl u^A}{\pl t} &= \frac{\pl}{\pl x^i}\left( \ul{g}^{ij}\frac{\pl u^A}{\pl x^j} + \CP^{Aj}_{Bi}\frac{\pl u^B}{\pl x^j} \right) + \CS^{Aj}_{B}\,\frac{\pl u^B}{\pl x^j} + \CT^A_B\,u^B + \CR^A
\label{MCF_2nd} \end{align}
where $\CS,\CT$ are uniformly bounded, and $\CP,\CR$ converges to $0$ uniformly as $t\to\infty$.  The solution and the coefficient functions are all smooth.  As mentioned in beginning of this appendix, $\nabla^{\Gm_t}\II^t$ is uniformly bounded, and thus $\pl_ju^A$ is uniformly bounded.

\subsection{Moser iteration}

We now apply the Moser iteration to estimate the sup-norm of $u^A$ in terms of its $L^2$ norm.  The formulation here is modified from the argument of Trudinger \cite{ref_Trudinger}.

Let $B$ the open ball of radius $1$, and let $D_T = B\times[T-1,T]$.  Denote by $B_t$ be the slice $B\times\{t\}$ for any $t\in[T-1,T]$.  Introduce the $V^2(D_T)$ norm \cite[(1.5)]{ref_Trudinger}:
\begin{align*}
||f||_{V^2(D_T)} &= \sup_{t\in[T-1,T]}||f||_{L^2(B_t)} + ||f||_{L^2(D_T)} ~.
\end{align*}
The following Sobolev lemma is a fundamental tool for the iteration.
\begin{lem}{\cite[Lemma 1.1]{ref_Trudinger}}
There exist constants $\kp>1$ and $C>0$ (which are independent of $T$) such that for any $f$ with finite $V^2(D_T)$ norm and with compact support in $B_t$ (a.e.\ $t$), $f$ must have finite $L^{2\kp}(D_T)$ norm.  Moreover,
\begin{align}
||f||_{L^{2\kp}(D_T)} &\leq C\,||f||_{V^2(D_T)}
\label{Sobolev} \end{align}
\end{lem}

We proceed with the Moser iteration argument for solutions $u^A$ of \eqref{MCF_2nd}.
There exists $c_0>0$ such that
\begin{align}
\sum_{i,j}\ul{g}^{ij}v^i v^j \geq c_0 |v|^2
\label{Moser00} \end{align}
at any $x\in B$ and for any vector $\{v^i\}\in\BR^n$.  For any $\dt>0$, consider
\begin{align*}
f &= \left(\sum_A (u^A)^2 + \sup_{D_T}\sum_A|\CR^A|^2 + \sup_{D_T}\sum_{i,j,A,B}|\CP^{Aj}_{Bi}| + \dt\right)^\oh ~.
\end{align*}
For any $\bt\geq1$, the partial derivative (in space) of its $(\bt+1)/2$ power is
\begin{align*}
\pl_i f^{\frac{\bt+1}{2}} &= \frac{\bt+1}{2} f^{\frac{\bt-3}{2}}\sum_A(u^A\,\pl_i u^A) ~.
\end{align*}
By the Cauchy--Schwarz inequality,
\begin{align*}
\left| \pl_i f^{\frac{\bt+1}{2}} \right| &\leq \frac{\bt+1}{2} f^{\frac{\bt-3}{2}}\,(\sum_A (u^A)^2)^\oh\,\left(\sum_A(\pl_i u^A)^2\right)^\oh ~.
\end{align*}
It follows that
\begin{align} \begin{split}
\left| \nabla f^{\frac{\bt+1}{2}} \right|^2
&= \frac{(\bt+1)^2}{4}\,f^{{\bt-3}}\,\left|\oh\nabla f^2\right|^2 \\
&\leq \frac{(\bt+1)^2}{4}\,f^{{\bt-1}}\sum_{A,i}(\pl_i u^A)^2 ~.
\end{split} \label{Moser02} \end{align}

Let $B(\rho)$ the open ball of radius $\rho$, and denote $B(\rho)\times[T-\rho^2,T]$ by $R_\rho$.  Note that $R_1 = D_T$, and $R_{\rho'}\subset R_\rho$ for any $\rho'<\rho\leq 1$.  Fix $\rho$ and $\rho'$ with $\oh\leq\rho'<\rho\leq1$.  Let $\eta$ be a cut-off function which is $1$ on $B(\rho')\times[T-(\rho')^2,\infty)$, and vanishes outside $B(\rho)\times[T-\rho^2,\infty)$.  
We compute
\begin{align*}
\int\int \eta^2\frac{\pl}{\pl t}f^{\bt+1} &= (\bt+1) \int\int \eta^2 f^{\bt-1}(u^A\,\pl_t u^A) \\
&= (\bt+1)\int\int \eta^2 f^{\bt-1}u^A \left[ \pl_i(\ul{g}^{ij}\pl_j u^A + \CP^{Aj}_{Bi}\,\pl_j u^B) + \CS^{Aj}_{B}\,\pl_j u^B + \CT^A_B\,u^B + \CR^A \right] ~.
\end{align*}
We estimate each term on the right hand side of the last expression. For the first term, we consider
\begin{align*}
&-\int\int \eta^2 f^{\bt-1}u^A \pl_i(\ul{g}^{ij}\pl_j u^A) \\
=\,& \int\int\pl_i(\eta^2 f^{\bt-1}u^A)\,\ul{g}^{ij}\,\pl_j u^A \\
=\,& \int\int \eta^2 \left[ f^{\bt-1}\ul{g}^{ij}\pl_i u^A\pl_j u^A + {(\bt-1)} f^{\bt-3}\,\ul{g}^{ij}\pl_i(\frac{f^2}{2})\pl_j(\frac{f^2}{2}) \right] + 2\eta\pl_i\eta\ul{g}^{ij}f^{\bt-1} u^A\pl_j u^A \\
\geq\,& \frac{1}{c_1} \int\int \eta^2 \left[ \frac{1}{\bt+1}|\nabla f^{\frac{\bt+1}{2}}|^2 + f^{\bt-1}\sum_A|\nabla u^A|^2 \right] - c_1\int\int |\nabla\eta|^2f^{\bt+1}
\end{align*}
For the second term on the right hand side,
\begin{align*}
& \int\int \eta^2 f^{\bt-1} u^A \pl_i(\CP^{Aj}_{Bi}\,\pl_j u^B) \\
=\, & -\int\int \pl_i(\eta^2 f^{\bt-1} u^A)\,\CP^{Aj}_{Bi}\,\pl_j u^B \\
=\, & -\int\int \eta^2\left[ (\bt-1)f^{\bt-3}\,u^C\,\pl_i u^C\,u^A\,\CP^{Aj}_{Bi}\,\pl_j u^B + f^{\bt-1}\,\pl_i u^A\,\CP^{Aj}_{Bi}\,\pl_j u^B \right] \\
\leq\,& c_2(\bt+1) \int\int\eta^2f^{\bt+1} ~,
\end{align*}
where we have use the uniform boundedness of $\pl_j u^A$.
For the rest terms on the right hand side,
\begin{align*}
& \int\int \eta^2 f^{\bt-1}u^A \left[\CS^{Aj}_{B}\,\pl_j u^B + \CT^A_B\,u^B + \CR^A \right] \\
\leq\, & \frac{1}{c_1}\int\int \eta^2f^{\bt-1}\sum_A|\nabla u^A|^2 + c_3\int\int\eta^2 f^{\bt+1}
\end{align*}

Putting these computations together gives
\begin{align}
\int\int\frac{\pl}{\pl t}(\eta^2f^{\bt+1}) + \int\int|\nabla (\eta f^{\frac{\bt+1}{2}})|^2 &\leq c_4(\bt+1)^2\int\int(\eta^2+|\nabla\eta|^2+\eta\pl_t\eta)\, f^{\bt+1}
\label{Moser03} \end{align}
By definition, there exists $t'\in(T-1,T)$ such that
\begin{align*}
\int_{B_{t'}}\eta^2 f^{\bt+1} &> \oh\sup_{t\in[T-1,T]}\int_{B_t}\eta^2f^{\bt+1} ~.
\end{align*}
By considering \eqref{Moser03} on $B\times[T-1,t']$,
\begin{align*}
\sup_{t\in[T-1,T]}\int_{B_t}\eta^2f^{\bt+1} &\leq 2c_4(\bt+1)^2\int\int_{R_\rho}(\eta^2+|\nabla\eta|^2+\eta\pl_t\eta)\, f^{\bt+1} ~.
\end{align*}
Combining it with \eqref{Moser03} for $R_\rho$ gives
\begin{align*}
\sup_{t\in[T-1,T]}\int_{B_t}\eta^2f^{\bt+1} + \int\int_{D_T}|\nabla (\eta f^{\frac{\bt+1}{2}})|^2 &\leq 3c_4(\bt+1)^2\int\int_{D_T}(\eta^2+|\nabla\eta|^2+\eta\pl_t\eta)\, f^{\bt+1} ~.
\end{align*}

Due to \eqref{Sobolev} and the choice of $\eta$, we find that
\begin{align}
|| f^{\frac{\bt+1}{2}} ||_{L^{2\kp}(R_{\rho'})} &\leq c_5\frac{\bt+1}{\rho-\rho'} ||f^{\frac{\bt+1}{2}}||_{L^2(R_\rho)}
\label{Moser04} \end{align}
for some $\kp>1$.  Equivalently,
\begin{align*}
|| f ||_{L^{\kp(\bt+1)}}(R_{\rho'}) \leq \left(c_5\frac{\bt+1}{\rho-\rho'}\right)^\frac{2}{\bt+1}||f||_{L^{\bt+1}}(R_\rho)
\end{align*}
Let $\rho_n = \oh + 2^{-(n+1)}$ for $n\in\{0,1,2,\ldots\}$.  Denote by $\Phi(n)$ the $L^{2\kp^n}$ norm of $f$ on $R_{\rho_n})$.  It follows that
\begin{align*}
\Phi(n+1) &\leq (c_5\,2^{n+2}\,\kp^n)^{\frac{1}{\kp^n}}\Phi(n) \leq\cdots\leq (4c_5)^{\sum_{j=0}^n\frac{1}{\kp^j}}(2\kp)^{\sum_{j=0}^n\frac{j}{\kp^j}}\Phi(0) ~.
\end{align*}
It follows that the sup-norm of $f$ on $B(\oh)\times[T-1/4,T]$ is bounded by some multiple of its $L^2$ norm on $B(1)\times[T-1,T]$.  Since $\dt>0$ is arbitrary, the $L^2$ norm of $u^A$, $\CR^A$ and $\CP^{Aj}_{Bi}$ are uniformly small, this implies that $u^A$ converges to zero uniformly.

\end{appendix}


\end{document}